\documentclass[10pt]{amsart}

\PassOptionsToPackage{linktocpage}{hyperref} 

\usepackage[french, english]{babel}
\usepackage{a4wide}
\usepackage[utf8,latin1]{inputenc}
\usepackage{stmaryrd} 
\usepackage{mathpazo}
\usepackage{tikz-cd}

\usepackage[normalem]{ulem}

\usepackage{amsmath,amscd,amssymb,mathrsfs,amsthm,amsfonts}
\usepackage{mathrsfs} 
\usepackage{bbm}
\usepackage{url}

\usepackage{hyperref}

\newtheorem{theorem}{Theorem}[subsection]
\newtheorem{lemm}[theorem]{Lemma}

\newtheorem{prop-def}[theorem]{Proposition-Definition}
\newtheorem{remark}[theorem]{Remark}

\newtheorem{coro}[theorem]{Corollary}
\newtheorem{prop}[theorem]{Proposition}

\numberwithin{equation}{subsection}

\newcommand{\Hom}{\mathrm{Hom}}
\newcommand{\End}{\mathrm{End}}
\renewcommand{\d}{\mathrm{d}}

\renewcommand{\Re}{\mathrm{Re}}
\renewcommand{\Im}{\mathrm{Im}}

\renewcommand{\hat}{\widehat}
\newcommand{\vol}{\mathrm{vol}}
\newcommand{\bbb}{\mathbb}

\newcommand{\Aut}{\mathrm{Aut}}

\newcommand{\Gal}{\mathrm{Gal}}
\newcommand{\Ind}{\mathrm{Ind}}
\newcommand{\Irr}{\mathrm{Irr}}
\newcommand{\Pic}{\mathrm{Pic}}
\newcommand{\Higg}{\mathrm{Higg}}
\newcommand{\Tr}{\mathrm{Tr}}

\newcommand{\M}{\mathbf{M}}

\newcommand{\ad}{\mathrm{ad}}
\newcommand{\Ad}{\mathrm{Ad}}
\newcommand{\ooo}{\mathcal{O}}

\newcommand{\res}{\mathrm{res}}

\DeclareMathOperator{\Id}{Id}

\DeclareMathOperator{\Spec}{Spec}

\newcommand{\AAA}{\mathbb{A}}

\renewcommand{\ggg}{\mathfrak{g}}
\renewcommand{\bbb}{\mathfrak{b}}

\newcommand{\nnn}{\mathfrak{n}}

\newcommand{\zzz}{\mathfrak{z}}
\newcommand{\ttt}{\mathfrak{t}}

\renewcommand{\leq}{\leqslant}
\renewcommand{\geq}{\geqslant}
\newcommand{\htau}{\hat \tau}
\newcommand{\Fr}{\mathrm{Frob}}
\makeatletter
\newcommand*{\rom}[1]{\expandafter\@slowromancap\romannumeral #1@}
\makeatother


\begin{document}

\setcounter{secnumdepth}{3}
\setcounter{tocdepth}{1}

\hypersetup{							
pdfauthor = {HONGJIE YU},			
pdftitle = {rank 2},			
}					

\title{Rank 2 $\ell$-adic local systems and Higgs bundles over a curve} 
\author{HONGJIE YU}
\address{Department of Mathematics, Weizmann Institute of Science, Herzl St 234, Rehovot, Israel. \newline
 Current address: Morningside Center of Mathematics, Academy of Mathematics and Systems Science, Chinese Academy of Sciences, Beijing 100190, China }
\email{hongjie.yu@amss.ac.cn}

\maketitle 
  
        \renewcommand{\abstractname}{Abstract}
\begin{abstract}
Let $X$ be a smooth, projective, and geometrically connected curve defined over a finite field $\mathbb{F}_q$ of characteristic $p$ different from $2$ and $S\subseteq X$ a subset of closed points. Let $\overline{X}$ and $\overline{S}$ be their base changes to an algebraic closure of $\mathbb{F}_q$. 
We study the number of $\ell$-adic local systems $(\ell\neq p)$ in rank $2$ over $\overline{X}-\overline{S}$ with all possible prescribed tame local monodromies fixed by $k$-fold iterated action of Frobenius endomorphism for every $k\geq 1$. In all cases, we confirm conjectures of Deligne predicting that these numbers behave as if they were obtained from a Lefschetz fixed point formula. In fact, our counting results are expressed in terms of the numbers of some Higgs bundles.  
\end{abstract}

\tableofcontents 
\section{Introduction}
Let $X$ be a smooth, projective, and geometrically connected curve defined over a finite field $\mathbb{F}_q$ of genus $g$.  
In the two pages article \cite{Drinfeld} of Drinfeld, he counts the number of two-dimensional geometrically irreducible $\ell$-adic (in $\overline{\mathbb{Q}}_\ell$-coefficients with $\ell\nmid q$)  representations of $\pi_1(X\otimes\overline{\mathbb{F}}_q)$ that can be extended to a representation of $\pi_1(X)$ (here we ignore the base point in the notation).
These numbers behave as if they were expressed by a Lefschetz fixed-point formula on an algebraic variety over the finite field. Moreover, they are independent of $\ell$. 

It is equivalent to consider $\ell$-adic local systems (smooth $\overline{\mathbb{Q}}_{\ell}$-sheaves). Although the Langlands correspondence established by Drinfeld and Lafforgue shows the motivic nature of $\ell$-adic local systems counted by Drinfeld, their definition depends very much on $\ell$. We don't know how to construct a moduli space of $\ell$-adic local systems in a reasonable sense that can explain these counting results. In fact, the techniques that we dispose of will only produce a space over $\bar{\mathbb{Q}}_\ell$, instead of $\mathbb{F}_q$. In this direction, the works \cite{AGKRRV} and \cite{AGKRRV2}  are fascinating where an algebraic stack over $\bar{\mathbb{Q}}_\ell$ is constructed and many interesting applications are expected. These works present valuable insights; however, they do not offer a Lefschetz-type fixed point formula that would aid in comprehending the counting problems. For example, each irreducible $\ell$-adic local system induces one connected component in their stack. 
 

Deligne has made some conjectures (\cite{Deligne}) on counting $\ell$-adic  local systems with prescribed local monodromies, i.e., with prescribed ramification types, to extend and understand Drinfeld's result. More explicitly, Deligne conjectures that the number of $\ell$-adic local systems 
with a fixed rank and prescribed tame ramifications that are fixed by $k$-iterated action of Frobenius endomorphism looks like the number of  $\mathbb{F}_{q^k}$-points of a variety defined over $\mathbb{F}_q$.  
Kontsevich \cite{Kont} makes some conjectures toward an understanding of Drinfeld's result. It is worth noting that Kontsevich considers the Hecke operator as well, which adds an interesting perspective to the discussion.

Some progress has been made since Deligne raised his conjectures. In fact, when the ramifications are split semisimple and in general position (which ensures that an $\ell$-adic local system is automatically irreducible), Arinkin has verified that in these cases, similar results hold (\cite{Deligne}). When the ramifications are unipotent with one Jordan block, and there are at least two such ramifications, Deligne's conjecture has been verified by Deligne-Flicker \cite{Deligne-Flicker}. 
The case in rank $2$ with one unipotent ramification is verified by Flicker \cite{Flicker}. We have generalized Drinfeld's result to a higher rank in \cite{Yu1}, and Arinkin's result to allow semisimple regular in general position but possibly non-split ramifications in \cite{Yu}. In the case of rank $2$, Flicker (\cite{Flicker1}) also obtains an explicit expression of the number of $\ell$-adic local systems with prescribed semisimple regular in general position ramifications that are fixed by Frobenius endomorphism. But, further analysis is needed to verify Deligne's conjecture since it is essential to let Frobenius endomorphism acts iteratively to observe if it is of the form asked by Deligne or not.

In this article, we verify Deligne's predictions on counting of $\ell$-adic local systems in rank $2$ for all possible tame ramifications. We show that the number is always related to the number of Higgs bundles. The results show an interesting analogy with Simpson's non-abelian Hodge theory, especially when $g=0$ and the ramifications are in general position and the parabolic weights of the parabolic Higgs bundles are also in general position (these two conditions correspond in Simpson's theory). We will discuss it in more detail at the end of the Introduction.

The fundamental principle underlying the proofs of all existing cases is the same: employing the Langlands correspondence and addressing the corresponding question within the realm of automorphic forms instead. 
There are several difficulties and technical novelties in this article compared to the existing partial cases. 

On the automorphic side, we use Arthur-Selberg trace formula to do the counting. Since we are interested in absolutely cuspidal automorphic representations, i.e. those whose base change in the sense of Langlands functionality from $F:=\mathbb{F}_q(X)$ to $F\otimes_{\mathbb{F}_q}{\mathbb{F}}_{q^k}$ remain to be cuspidal for all $k\geq 1$, we use in fact a twisted trace formula. The spectral side of the trace formula for $GL_2$ consists of cuspidal, residual, and continuous parts. Different from higher rank cases, the continuous and residual parts are relatively easy to treat because the group is small. To allow all tame ramifications, the more difficult part is the cuspidal part where we need to calculate twisted traces. Using Whittaker models, the question is reduced to a local question and we use an explicit calculation of Whittaker function given by Paskunas-Stevens. The geometric side of the trace formula uses a similar approach that I already used in the generic case that we pass to Lie algebra and uses a Lie algebra trace formula and Weil's dictionary.

After computing the trace formula, we obtain an equation representing the count of absolutely cuspidal automorphic representations. However, this equation alone does not prove Deligne's conjectures. While we have an expression for the count in terms of $\mathbb{F}_q$-points of certain varieties, the nature of the expression changes when transitioning from $F$ to $F\otimes\mathbb{F}_{q^k}$. 
It is not obvious that the expression we have is of a Lefschetz type formula, as conjectured by Deligne. The reason behind these is that $(1)$ a closed point will split into several different closed points after base change $(2)$ ramification type on the automorphic side changes as well. Especially it is the mixture of twisted Steinberg components that creates problems. In this context, certain arguments involve combinatorial aspects, but it is crucial to analyze the Frobenius action on the moduli spaces of Higgs bundles in order to establish the dominant term as being of Lefschetz type. 
It is amusing to see that the calculations given by trace formula (Theorem \ref{automorphe}) are subdivided into 13 cases while the statement of our main theorem is rather clean.

\subsection{Main results}  
\subsubsection{}
Let us recall Deligne's conjectures that will be treated in this article. We follow Deligne's presentation in \cite{Deligne}, but we restrict to the rank $2$ cases.

Let $X$ be a smooth, projective, and geometrically connected curve defined over a finite field $\mathbb{F}_q$. Let $S\subseteq X$ be a subset of closed points. We fix an algebraic closure $\overline{\mathbb{F}}_q$ of $\mathbb{F}_q$. 
Let 
$\overline{X}:=X\otimes\overline{\mathbb{F}}_q$ and 
$\overline{S}:=S\otimes\overline{\mathbb{F}}_q$. For each point $x\in \overline{S}$, let $\overline{X}_x^{*}=\overline{X}_x-\{x\}$ be a punctured disc in $x$ ($\overline{X}_x$ is defined to be either the Henselization or the completion of $\overline{X}$ in $x$).  
We fix a rank $2$ $\ell$-adic local system ($\overline{\mathbb{Q}}_\ell$-smooth sheaf) $\mathfrak{R}_x$ over $\overline{X}_x^{*}$. Let $E_2(\mathfrak{R})$ be the set isomorphism classes of irreducible rank $2$ $\ell$-adic local systems over $\overline{X}-\overline{S}$ whose restriction to $\overline{X}_x^{*}$ is isomorphic to $\mathfrak{R}_x$ for every $x\in \overline{S}$. 
Let $\Fr$ be the Frobenius endomorphism of $\overline{X}$, i.e., the base change to $\overline{\mathbb{F}}_q$ of the morphism induced by the map $a\mapsto a^q$ on $X$. 
If \begin{equation}\label{Frobe}\Fr^{*}(\mathfrak{R}_{\Fr(x)})\cong \mathfrak{R}_x\end{equation} for every $x\in \overline{S}$, then the pullback of $\Fr$ permutes $E_2(\mathfrak{R})$. Let $ E_2(\mathfrak{R})^{\Fr^{*k}}$ be the set of fixed elements of $k$-iterated action of $\Fr^{*}$. 

Deligne conjectured that if all $\mathfrak{R}_x$ are tamely ramified, then there are $q$-Weil integers $\alpha$ and integers $m_\alpha$ such that 
\[\vert  E_2(\mathfrak{R})^{\Fr^{*k}}\vert    =\sum_{\alpha} m_\alpha\alpha^k, \quad\forall k\geq 1,  \]
where $\vert  E_2(\mathfrak{R})^{\Fr^{*k}}\vert$ is the cardinality of subset of the fixed points by $k$-fold iterated action of $\Fr^{*}$. 
To formalize this property, let us introduce some integral valued functions on $\mathbb{N}^{\ast}$. 
We say that a function $k\mapsto h(k)$ form $\mathbb{N}^{\ast}$ to $\mathbb{Z}$ is of Lefschetz type if there are $q$-Weil integers $\alpha$ and integers $m_\alpha\in\mathbb{Z}$ such that \[h(k)= \sum_{\alpha} m_\alpha \alpha^k.  \] 
Therefore, the conjecture is to prove that 
\[k\mapsto |E_2(\mathfrak{R})^{\Fr^\ast k }| \]
is of Lefschetz type. A typical example is a function $k\mapsto |V(\mathbb{F}_{q^k})|$ for a variety $V$ defined over $\mathbb{F}_q$. In particular, given a permutation $\sigma$ on a finite set $P$, the function $k\mapsto |P^{\sigma^k}|$ is a periodic function of Lefschetz type. Note that not all integral valued periodic functions are Lefschetz type as the integrality of $m_\alpha$ is essential.

\subsubsection{}
The tame étale fundamental group of $\overline{X}_x^{*}$ is topologically generated by one element. Therefore, an isomorphism class of tame local system of rank $2$ over $\overline{X}_x^{*}$ corresponds bijectively to conjugacy classes in $GL_2(\overline{\mathbb{Q}}_\ell)$. 

The set $\overline{S}\subseteq X(\overline{\mathbb{F}}_q)$ is fixed by $\Fr$ and its orbits correspond bijectively to $S$. Following the types of prescribed local monodromies, we can define a partition on $\overline{S}$ hence $S$ into a disjoint union of subsets:
\[\overline{S} = \overline{S}_{s} \cup \overline{S}_{u}\cup \overline{S}_{cr}, \]
where $\mathfrak{R}_x$ has different eigenvalues for $x\in \overline{S}_{cr}$, $\mathfrak{R}_x$ induces a scalar matrix in $GL_2(\overline{\mathbb{Q}}_\ell)$ for $x\in \overline{S}_s$ and $ \mathfrak{R}_x $ induces a quasi-unipotent conjugacy class with non-trivial Jordan block for $x\in \overline{S}_u$. As each of these sets is stable under $\Fr$, we have a partition
\[S=S_s\cup S_u\cup S_{cr}. \]

Let $x_1\in \overline{S}_{cr}$. Suppose $x_1\xrightarrow{\Fr} x_2\xrightarrow{\Fr} \cdots \xrightarrow{\Fr} x_{d+1}=x_1$ be the orbit containing $x_1$ of the Frobenius action ($x_i\neq x_1$ for any $1<i\leq d$).  
There are two non-isomorphic rank $1$ $\ell$-adic local systems $\mathfrak{L}_1$ and $\mathfrak{L}_2$ over $\overline{X}_{x_1}^{*}$  such that \[\mathfrak{R}_{x_1}\cong \mathfrak{L}_1\oplus \mathfrak{L}_2. \]
The condition \eqref{Frobe} implies that 
 \[\Fr^{*d}\mathfrak{L}_1\cong \mathfrak{L}_i,\]
for $i=1$ or $i=2$. 
 This allows us to further subdivide $\overline{S}_{cr}$ so that 
\[\overline{S}_{cr}=\overline{S}_{c}\cup \overline{S}_r, \]
where $\overline{S}_{r}$ is the set of points such that $i=1$ and $\overline{S}_{c}$ consists of those points such that $i=2$. Again, we deduce a partition
\[S_{cr}=S_c\cup S_r. \]

\subsubsection{}
Now we need to introduce some functions of Lefschetz type that are used to express the final results.

Let $\mathfrak{R}$ be a collection of tame local monodromies as above so that the condition \eqref{Frobe} is satisfied. Its eigenvalues for each $x\in \overline{S}$ define a couple of numbers $(\varepsilon_{x}(1), \varepsilon_{x}(2))\in \overline{\mathbb{Q}}_\ell^\times$ which could be the same. Let $\overline{S}=\{x_1, \cdots, x_r\}$. 
We define a set  $P_{\mathfrak{R}}$ by
\begin{equation}\label{PR}P_{\mathfrak{R}}:= \{ (\varepsilon_{x_1}(i_1),  \ldots, \varepsilon_{x_r}(i_r)) \mid   \prod_{j=1}^{r} \varepsilon_{x}(i_j)=1;   i_j\in \{1,2\}, j=1, 2, \ldots, r \}. \end{equation}
Let $\Fr^\ast$ be a permutation on $P_{\mathfrak{R}}$ defined so that for any $(\varepsilon_{x})_{x\in\overline{S} }\in (\overline{\mathbb{Q}}_\ell^\times)^{\overline{S}}$, we have
\[\Fr^{\ast}( (\varepsilon_{x})_{x\in\overline{S} })= (\varepsilon'_{x})_{x\in\overline{S} },\]
with 
\[\varepsilon'_{x}= \varepsilon_{\Fr(x)}^q , \quad \forall x\in\overline{S} .   \]
The relation \eqref{Frobe} tells us that it is a well-defined permutation, since $\varepsilon_{\Fr(x)}(1)^q $ equals either $\varepsilon_{x}(1) $ or $\varepsilon_{x}(2)$. 
We define a function $c_\mathfrak{R}:\mathbb{N}^{\ast}\longrightarrow \mathbb{Z}$ by
 \[c_\mathfrak{R}(k):= \vert  P_{\mathfrak{R}}^{\Fr^{\ast k}}    \vert , \]
the number of the fixed points of $\Fr^{\ast k}$ on $P_{\mathfrak{R}}$. It is of Lefschetz type. 
Let $\sigma$ be an involution on $P_{\mathfrak{R}}$ that sends ${(\varepsilon_x(i_{x}))_{x\in \overline{S}}}$ to  ${(\varepsilon_x(3-i_{x}))_{x\in \overline{S}}}$. Define \[ b_\mathfrak{R}(k):=|P_{\mathfrak{R}}^{\sigma=\Fr^{\ast k}}|\] as the cardinality of the fixed point set of the action of $\sigma\circ\Fr^{\ast k}$. We prove in Proposition \ref{Lefschetz} that it is also of Lefschetz type.

Now we introduce some functions of Lefschetz type coming from counting of Hitchin bundles. 
Suppose that $k\in\mathbb{N}^\ast$, and $V\subseteq {S}_u\otimes\mathbb{F}_{q^k}$. Let $\overline{V}=V\otimes_{\mathbb{F}_{q^k}}\overline{\mathbb{F}}_q$ and 
\[D=K_{\overline{X}}+\sum_{x\in \overline{V}\cup \overline{S}_{cr}} x\]
be a divisor over $\overline{X}$ where $K_{\overline{X}}$ is a canonical divisor on $\overline{X}$, i.e., a divisor whose associated line bundle is the canonical line bundle $\Omega_{\overline{X}/\overline{\mathbb{F}}_q}^{1}$. 
A parabolic Hitchin bundle of rank $2$ and degree $1$ with parabolic structures in $\overline{V}$ for the divisor $D$ is a triple $(\mathcal{E}, \varphi, (L_x)_{x\in \overline{V}})$ consisting of a vector bundle of rank $2$ and degree $1$ over $\overline{X}$, a bundle morphism
\[\varphi: \mathcal{E} \rightarrow \mathcal{E}\otimes \mathcal{O}_{\overline{X}}(D),\]
and a family of one dimensional $\overline{\mathbb{F}}_q$-subspace $L_x$ of $\mathcal{E}_x$ ($x\in \overline{V}$), the fiber of $\mathcal{E}$ in $x$, such that $\varphi_x(L_x)=0$ and $\Im(\varphi_x)\subseteq L_x$, for any $x\in \overline{V}$. 
We say that $(\mathcal{E}, \varphi, (L_x)_{x\in \overline{V}})$ is stable if for any subline bundle $\mathcal{L}$ of $\mathcal{E}$ satisfying $\varphi(\mathcal{L})\subseteq \mathcal{L}\otimes \mathcal{O}_{\overline{X}}(D)$, we have
 \[\deg(\mathcal{L})< \frac{\deg \mathcal{E}}{2}. \] 
We denote $M_{\overline{V}}^{1}(D)$ the coarse moduli space of these stable parabolic Hitchin bundles. It is a variety defined over $\overline{\mathbb{F}}_q$ (more details are given in Section \ref{HIG}). We  show in Section \ref{Fq} that it admits canonical $\mathbb{F}_{q^k}$-structure (i.e.  $M_{\overline{V}}^{1}(D)$ is the base change to $\overline{\mathbb{F}}_q$ of a variety defined over $\mathbb{F}_{q^k}$) whose $\mathbb{F}_{q^k}$-points classify isomorphism classes of stable parabolic Hitchin bundles over $X\otimes\mathbb{F}_{q^k}$, which we denote by $M_{V}^{1}(D)$.

For each $v\in S_{cr}$, we fix a monic polynomial  $o_v\in \kappa_v[t]$ of degree $2$ with coefficients in $\kappa_v$ (the residue field of the point $v$), so that we require that $o_v$ is irreducible for $v\in S_c$  and $o_v$ has distinct roots in $\kappa_v$ if $v\in S_{r}$. It defines a polynomial in $\kappa_x[t]$, for every closed point $x$ of $\overline{X}$ lying over $v$ via the isomorphism:\[\kappa_v[t]\otimes_{\mathbb{F}_q}\overline{\mathbb{F}}_q\cong \prod_{x\mapsto v}\overline{\mathbb{F}}_q[t].\]  
We define ${M}_{V}^{1}(o_{S_{cr}})$ to be the closed sub-variety of ${M}_{V}^{1}(D)$ over $\mathbb{F}_{q^k}$ consisting of those parabolic Hitchin bundles so that the characteristic polynomial of $\varphi_x$ at $x\in \overline{S}_{cr}$ is given by $o_x$. We suppose that the sum of roots of $o_x$ ($x\in\overline{S}_{cr}$) is zero (this is always possible if $p\neq 2$, see Remark \ref{pneq2}). 
We refer to Section \ref{semistabler} for a more precise and detailed definition. 
We define for each $k\geq 1$,
\[\mathrm{Higg}_{\mathfrak{R}}(k)=
\sum_{V\subseteq {S}_u\otimes\mathbb{F}_{q^k}} (-1)^{|{S}_u\otimes\mathbb{F}_{q^k} - V|} 2^{|V|}  q^{-k(4g-3+|\overline{V}|+|\overline{S}_{cr}|)}    |{M}_{V}^{1}(o_{S_{cr}})(\mathbb{F}_{q^k})|.
\]
This expression depends only on Frobenius action on $\mathfrak{R}$, but not on $(o_v)_{v\in S_{cr}}$. 
We show in Theorem \ref{Lefs} that this is a function of Lefschetz type in $k$.

Let $\Pic_{X}^{0}$ be the Jacobian variety of $X$. We also define for every $k\geq 1$
\[ \Pic(k):=|\Pic_{X}^{0}(\mathbb{F}_{q^k})|,\] and 
\[ \Pic^{(2)}(k):= |S^2\Pic_{X}^0(\mathbb{F}_{q^k})| ,\]
where $S^2\Pic_{X}^0:=(\Pic_X^{0})^{2}/\mathfrak{S}_2$ is the symmetric square of $\Pic_{X}^0$. They are surely also functions of Lefschetz type in $k\in \mathbb{N}^\ast$.

\subsubsection{}
The following theorem proves Deligne's conjectures \cite[2.15 (i)(iii)]{Deligne} when $n=2$ and ramifications are tame.  

\begin{theorem}\label{1}
Suppose that $p\neq 2$. 
Suppose that \eqref{Frobe} is satisfied, so that $\Fr^\ast$ acts on $E_2(\mathfrak{R})$. Suppose that 
\begin{equation}\label{prod1}\prod_{x\in \overline{S}}\varepsilon_x(1)\varepsilon_x(2)=1,\end{equation}
otherwise $E_2(\mathfrak{R})$ is empty. 
Then the function \[k\mapsto \vert E_2(\mathfrak{R})^{\Fr^{*k}} \vert \] is of Lefschetz type. 

More precisely, we have the following explicit identities that express $\vert E_2(\mathfrak{R})^{\Fr^{*k}} \vert$ following different cases. 
\begin{enumerate}
\item[i.] $\overline{S}_{cr}=\overline{S}_u=\emptyset$. Then $|E_2(\mathfrak{R})^{\Fr^{\ast  k}}|$ equals
\[\Higg_\mathfrak{R}(k)-
c_{\mathfrak{R}}(k) \biggr(\Pic(k)  ^2(g-1)+ \Pic({k})  \biggr). \]


\item[ii.] $\overline{S}_{cr}=\emptyset$, $\overline{S}_{u}\neq \emptyset$.  
Then $|E_2(\mathfrak{R})^{\Fr^{\ast  k}}|$ equals
  \[
\Higg_\mathfrak{R}(k) - c_{\mathfrak{R}}(k) \biggr(\beta_{\overline{S}_u} (k)(-1)^{|\overline{S}_u|+1}  \Pic^{(2)}(k)  +\gamma_{\overline{S}_u}(k)\Pic(k)+\omega_{\overline{S}_{u}}\Pic(k)  ^2\biggr) .   
  \]




\item[iii.] $\overline{S}_{cr}\not=\emptyset$,   $\overline{S}_u=\emptyset$.  If $|\overline{S}_{c}|$ is even, then $|E_2(\mathfrak{R})^{\Fr^{\ast k}}|$ equals 
\[\begin{split}
&\Higg_\mathfrak{R}(k) - \frac{c_{\mathfrak{R}}(k){(2g-2+|\overline{S}_{cr}|)}}{2} \Pic(k)^2 
 . \end{split} \]

\item[iv.] $\overline{S}_{cr}\not=\emptyset$,   $\overline{S}_u=\emptyset$. 
 If $|\overline{S}_{c}|$ is odd,  then $|E_2(\mathfrak{R})^{\Fr^{\ast k}}|$ equals

\[
\Higg_\mathfrak{R}(k) - \biggr( {c_{\mathfrak{R}}(k)}\frac{2g-1+|\overline{S}_{cr}|}{2} - \frac{c_{\mathfrak{R}}(k)+b_{\mathfrak{R}}(k)}{2}\biggr) \Pic(k)^2      
 - {b_{\mathfrak{R}} (k)  } \Pic^{(2)}(k). \]

\item[v.] $\overline{S}_{cr}\neq \emptyset$, $\overline{S}_u\neq \emptyset$. If $|\overline{S}_c|$ is even, then $|E_2(\mathfrak{R})^{\Fr^{\ast k}}|$ equals

\[\begin{split}
\Higg_{\mathfrak{R}}(k) - \biggr(\frac{c_{\mathfrak{R}}(k)\alpha_{\overline{S}_u}({k})}{2} +(-1)^{|\overline{S}_u|} \frac{b_{\mathfrak{R}}(k)\beta_{\overline{S}_u}(k) }{2}  \biggr)\Pic(k)^2  
 +(-1)^{|\overline{S}_u|}{b_{\mathfrak{R}}(k)} \beta_{\overline{S}_u}(k) \Pic^{(2)}(k). 
\end{split}
\]

\item[vi.] $\overline{S}_{cr}\neq \emptyset$, $\overline{S}_u\neq \emptyset$. If $|\overline{S}_c|$ is odd, then $|E_2(\mathfrak{R})^{\Fr^{\ast k}}|$ equals
\[\begin{split}
\Higg_{\mathfrak{R}}(k) - \biggr(\frac{c_{\mathfrak{R}}(k)\alpha_{\overline{S}_u}({k})}{2} -(-1)^{|\overline{S}_u|} \frac{b_{\mathfrak{R}}(k)\beta_{\overline{S}_u}(k)}{2}  \biggr)\Pic(k)^2  
 -(-1)^{|\overline{S}_u|}{b_{\mathfrak{R}}(k)}\beta_{\overline{S}_u}(k) \Pic^{(2)}(k). 
\end{split}
\]

\end{enumerate}
In the above expressions $\alpha_{\overline{S}_u}$, $\beta_{\overline{S}_u}$, $\gamma_{\overline{S}_u}$ and $\omega_{\overline{S}_u}$ are periodic functions of Lefschetz type (see Proposition \ref{Lefschetz} for their explicit expressions). When $\overline{S}_{cr}$ is non-empty, $c_{\mathfrak{R}}/2+b_{\mathfrak{R}}/2$ and ${c_{\mathfrak{R}}\alpha_{\overline{S}_u}}/{2} \pm {b_{\mathfrak{R}}\beta_{\overline{S}_u}}/{2}$ are of Lefschetz type. If $|\overline{S}_{r}|$ is odd, then ${c_{\mathfrak{R}}}/{2}$ is of Lefschetz type and $b_{\mathfrak{R}}$ is constantly zero. 
\end{theorem}
\begin{remark}\normalfont\begin{enumerate}
\item Even in characteristic $2$, the theorem holds as long as a residue datum $o_{S_{cr}}=(o_v)_{v\in S_{cr}}$ in the definition of $\Higg_{\mathfrak{R}}$ exists. If $p\neq 2$, it always exists. However, if $p=2$ and $S_{cr}$ is a singleton, then such a datum does not exist.  If $S_c=\emptyset$, Theorem \ref{NoSc} can be employed to express $\Higg_{\mathfrak{R}}$ without the necessity of choosing a residue datum. Therefore, Theorem \ref{1} is applicable in this case for $p=2$ as well.

\item A key point of the theorem is that $k\mapsto \Higg_{\mathfrak{R}}(k)$ is of Lefschetz type. This is non-trivial if not all points in $S_u$ have degree $1$.
\item
The necessity of the condition \eqref{prod1} for $E_2(\mathfrak{R})$ to be non-empty is explained in \cite[2.10]{Deligne}. It can be checked by passing to characteristic $0$ and then passing to $\mathbb{C}$, where one can use an explicit presentation of the topological fundamental group of a punctured Riemann surface (see the proof of Corollary 7.7 of \cite{Deligne-Flicker}). 
\item
The ramifications $\mathfrak{R}_x$ for $x\in \overline{S}_s$ only affects $c_{\mathfrak{R}}(k)$ and $b_{\mathfrak{R}}(k)$ and is not involved in any other term. 

\item In the case of general position, i.e., when the set $P_{\mathfrak{R}}=\emptyset$, we have \[ |E_2(\mathfrak{R})^{\Fr^{\ast k}}| = \Higg_{\mathfrak{R}}(k), \quad \forall k\geq 1. \]
The appearance of additional terms may be related to the singularity of the moduli space of (S-equivalent classes of) semistable parabolic Higgs bundles in cases that are not in general position.
\end{enumerate}
\end{remark}

The following corollary confirms Deligne's conjectures \cite[6.3]{Deligne}. 
\begin{coro}
The cardinality $\vert E_2(\mathfrak{R})^{\Fr^{*k}} \vert$ is divisible by $\Pic(k)$ and \[k\mapsto \vert E_2(\mathfrak{R})^{\Fr^{*k}} \vert/ \Pic(k)\]
is still a function of Lefschetz type. 
\end{coro}

Although we deal only with the tame local monodromies, the method of this article allows us to treat some wild ramified cases as well, namely, the cases in which the prescribed local monodromies $(\mathfrak{R}_x)$ can be wildly ramified.  
For example, we can allow some places to give the so-called simple supercuspidal representation on the automorphic side (see \cite[3.1, 3.4]{AL16} for the definition and explicit constructions of simple supercuspidal representations and their Langlands parameters). There could be a similar result involving wild/meromorphic Hitchin bundles. More precisely, we can consider those $\ell$-adic local systems on $\overline{X}-\overline{S}$ with $S\neq \emptyset$  whose ramifications are  restrictions of Langlands parameters of prescribed simple supercuspidal representations to the inertia Galois group. We conjecture that the number of such $\ell$-adic local systems fixed by the Frobenius endomorphism equals to, up to a power of $q$, the number of  Higgs bundles over $X$ whose Higgs fields have non-simple poles in $S$ and the principal part of the Higgs field around a point of $S$ belongs to some fixed coadjoint orbit of a Lie algebra (we refer to \cite{Bo} for related theories over complex numbers). 
In a private note by Zhiwei Yun, he has a simple geometric method to deal with some cases with wild ramifications.

\subsection{When $g=0$. }
\subsubsection{}
In the following, the characteristic of $\mathbb{F}_q$ is allowed to be $2$. 

Suppose that $g=0$, i.e., the curve $X$ is $\mathbb{P}^1$, and $S_c=\emptyset$. 
We will present an analogy with Simpson's non-abelian Hodge theory. One can compare with \cite[Theorem 1.4]{Yu}. New phenomena appear when there are quasi-unipotent local monodrmies. However, 
a naive generalization does not hold in a more general case which we hope to understand in the spirit of conjectures \cite[2.18, 2.21]{Deligne}. We are interested in the case that $\mathfrak{R}$ is in general position, i.e. when $P_{\mathfrak{R}}$ is empty.

Let \[R=S_r\cup S_u,\]
and \[D=K_{{X}}+\sum_{v\in {R}}v.\]
Note that we do not consider points in $S_s$. 
We use the same letter $D$ for the divisor $K_{\overline{X}}+\sum_{x\in \overline{R}}x$ over $\overline{X}$. 
Let $\overline{\xi}=(\overline{\xi}_x)_{x\in \overline{R}}\in (\mathbb{Q}^2)^{\overline{R}}$ such that $\overline{\xi}_{x,1}\geq \overline{\xi}_{x,2}\geq \overline{\xi}_{x,1}- 1$ and $\overline{\xi}_x=\overline{\xi}_y$ for any $x, y$ lying over the same closed point $v\in R$. These vectors serve as parabolic weights (stability parameters). 
We call a parabolic Hitchin bundle a Higgs bundle if $D=K_X+\sum_v v$ and parabolic structures are imposed in $\overline{R}$.
Let $(\mathcal{E}, \varphi, (L_x)_{x\in \overline{R}})$ be a parabolic Higgs bundle over $\overline{X}$. Let $\mathcal{L}$ be a sub-line bundle of $\mathcal{E}$, we define the parabolic degree $\mathrm{p\text{-}deg}(\mathcal{L})$ by
\[\mathrm{p\text{-}deg}(\mathcal{L}):= \deg(\mathcal{L})+\sum_{x\in \overline{R}}\begin{cases}  \overline{\xi}_{x, 1}, \text{ if }  \mathcal{L}_x=L_x; \\
\overline{\xi}_{x, 2},  \text{ if }  \mathcal{L}_x \neq L_x.
\end{cases}
\]
We say that $(\mathcal{E}, \varphi, (L_x)_{x\in \overline{R}})$ is $\overline{\xi}$-semistable if for any sub-line bundle $\mathcal{L}$ of $\mathcal{E}$ satisfying $\varphi(\mathcal{L})\subseteq \mathcal{L}\otimes\mathcal{O}_{\overline{X}}(D)$, we have
 \[\mathrm{p\text{-}deg}(\mathcal{L})\leq \frac{\deg \mathcal{E}+\sum_{x\in \overline{R}}(\overline{\xi}_{x,1}+\overline{\xi}_{x,2}  )}{2}. \] 
Note that if \begin{equation}\label{generalposition}\deg(\mathcal{E})+\sum_{x\in \overline{R} }\pm (\overline{\xi}_{x,1}-\overline{\xi}_{x,2}  ) \notin 2\mathbb{Z},\end{equation} then the equality can never be achieved. 
We say that such cases are in general position.

Choose $\overline{\xi}$ as above and suppose that it is in general position. 
The coarse moduli space of $\overline{\xi}$-semistable parabolic Higgs bundles of rank $2$ and of degree $e$ that are semistable with parabolic weights  $(\overline{\xi}_x)_{x\in \overline{R}}$ over $\overline{X}$ has a canonical $\mathbb{F}_q$-structure (see Section \ref{Fq}). We denote the moduli space by $M_{R}^{e,\xi}=M_{R}^{e, \xi}(D)$. 
We show in Theorem \ref{weights} that $|M_{R}^{e,\xi}(\mathbb{F}_q)|$ is independent of the choice of the parabolic weights as long as $\overline{\xi}$ is in general position. 
The space $M_{R}^{e,\xi}$ has a $\mathbb{G}_m$-action via dilation of the Higgs field. 
Let ${^{gr}M_{R}^{e, \xi}}:= (M_{R}^{e,\xi})^{\mathbb{G}_m}$, and ${^{gr}M_{R}^{e, \xi}}(S_u)$ be its open subvariety consisting of those parabolic Higgs bundles $(\mathcal{E}, \varphi, (L_x)_{x\in \overline{R}})$ whose Higgs field $\varphi$ does not vanish at $x\in \overline{S}_u$ (i.e. $\varphi_x\neq 0$).  

\begin{theorem}\label{2}
Suppose that $\mathbb{F}_q\neq \mathbb{F}_2$, $g=0$ and $S_c=\emptyset$. Suppose that $(e,\xi)$ is in general position and that \[\overline{\xi}_{x,1}=\overline{\xi}_{x,2}\] for $x\in \overline{S}_u$. Suppose that $\mathfrak{R}$ is in general position in the sense that $P_{\mathfrak{R}}=\emptyset$. Suppose that  
\[\prod_{x\in \overline{S}}\varepsilon_x(1)\varepsilon_x(2)=1.\]
If either $e$ is an odd integer or there is a place of odd degree in $R$, then we have
\[  |{^{gr}M_{R}^{e, \xi}}(S_u)(\mathbb{F}_{q^k})| = |E_2(\mathfrak{R})^{\Fr^{\ast k}}|  . \]
\end{theorem} 

\subsection{}
Given a compact Riemann surface $\Sigma$ and a finite set of points $R\subseteq \Sigma$,
Simpson has established a correspondence between semistable $\mathbb{C}$-local systems over $\Sigma-R$ of degree $0$ and semistable quasi-parabolic Higgs bundles over $\Sigma$ with parabolic structures in $R$ of parabolic degree $0$ (we refer to Simpson's original article \cite{Simpson} for more details). 
On the local system side, Simpson defines residual data for each $x\in R$ using the local monodromy and stability weight at the puncture $x$. 
Similarly, Simpson defines residual data for each $x\in R$ using the Higgs field and the parabolic weight in the Higgs bundle side. His correspondence preserves the nilpotent part of the residual data and permutes the stability weights and eigenvalues of the residual datum. Theorem \ref{2} presents an analogy with Simpson's theory if we choose the stability weights of the local systems to be trivial and we choose $(e,\xi)$ in accordance with Simpson's correspondence (cf. the diagram in \cite[p.720]{Simpson}). An interesting phenomenon is that $\mathfrak{R}$ being in general position corresponds to that $(e,\xi)$ being in general position under Simpson's correspondence. 

The dominant term in Theorem \ref{1} when $k$ varies is $(q^{4g-3+|\overline{S}_u|+|\overline{S}_{cr}|})^{k}$. It is half of the dimension of the moduli space of parabolic Higgs bundles of the relevant complex analogy in Simpson's theory. This may be related to the motivic nature of $\ell$-adic local systems over a curve over $\mathbb{F}_q$.

We can not expect a naive generalization of Theorem \ref{2} to the cases $g>0$. However, we expect it to be a special case of (a possible modification of) Deligne's conjecture in \cite[2.21]{Deligne}. Indeed, suppose that all ramifications are split regular semisimple (for $n=2$, it is the case that $S_s=S_c=S_u=\emptyset$), in the cases of in general position,  
we have (\cite[Th. 1.1]{Yu} for any rank)
\[ |E_n(\mathfrak{R})^{\Fr^{\ast k}}|=\sum_{i}(-1)^i\Tr(V^{\ast k} |   H_c^i(  (M_{n, S_r}^{e,\xi})_{\overline{\mathbb{F}}_q}, \mathbb{Q}_\ell) ) ,    \]
for any endomorphism $V^\ast$ conjugate to $q^{-\frac{1}{2}(n^{2}(g-1)+|\overline{S}_{r}|) }F_q^\ast$. We can expect to generalize it to the cases where ramifications are only supposed to be semisimple but remain in general position. The more demanding question is to generalize it to allow non-trivial quasi-unipotent ramifications or even cases not in general positions. 
 Now we may ask if $V^\ast$ is induced from a morphism of $(M_{n, S_r}^{e,\xi})_{\overline{\mathbb{F}}_q}$. This does not seem to be the case if we consider only algebraic varieties over $\overline{\mathbb{F}}_q$. We likely have to consider a lifting of the curve to characteristic $0$ and consider $p$-adic geometry which the author is not competent to comment on. Instead, we refer the reader to Deligne's course at IHES \cite{IHES}.

In this article, we only need semistable parabolic Higgs bundles with parabolic weights in general position. It is not necessary to do so. However, with our method, it is more natural to consider the algebraic stack version of the moduli of semistable parabolic Higgs bundles when the parabolic weights are not in general position, and we should expect a more complicated expression for the point counting problem in this case.

\subsection*{Acknowledgement} 
I thank Erez Lapid for his invaluable suggestions.  
I thank Kang Zuo for the fruitful discussions.
 This work continues the project that started from my thesis; I thank Pierre-Henri Chaudouard for giving this project to me.  The work is finished during my stay at IST Austria and Weizmann Institute of Science. I thank both institutes for providing me with excellent work conditions. Part of the work is finished with the support of the BSF grant 2019274.

\section{Notation}
We gather notation that will be used throughout the article. Other notation will be defined where they appear. 

 $\bullet F, |X|, F_v, \mathcal{O}_v,  \wp_v, \kappa_v, q_v,\mathbb{A}, \mathcal{O}.$  Let $F=\mathbb{F}_q(X)$ be the global function field of the curve $X$. Let $\vert  X\vert  $ be the set of closed points of $X$, identified with the set of places of $F$. For every $v\in \vert  X\vert  $, let $F_v$ be the local field in $v$, $\mathcal{O}_v$ the ring of integers in $F_v$ and $\kappa_v$ the residue field of $\mathcal{O}_v$. Let $\wp_v$ be the maximal ideal in $\mathcal{O}_v$, and we choose a uniformizer $\varpi_v$.  
Suppose that $\kappa_v$ has cardinality $q_v$, therefore $\kappa_v\cong \mathbb{F}_{q_v}$. 
Let $\mathbb{A}$ be the ring of adèles of $F$ and $\mathcal{O}$ be the sub-ring of integral ad\`eles. 

$\bullet  G, B,N,T, \overline{B}, \overline{N}  .$ If not specified otherwise, we use $G$ for $GL_2$. Let $B$ be the Borel subgroup of $G$ consisting of upper triangular matrices and $T$ be the torus consisting of diagonal matrices. Let $N$ be the unipotent radical of $B$, i.e., the group of upper triangular matrices with $1$ on the diagonal. Let $\overline{B}$ be the Borel subgroup that is opposite to $B$, i.e., consisting of lower triangular matrices, and $\overline{N}$ be the unipotent radical of $\overline{B}$. 

$\bullet    \ggg, \bbb, \nnn, \ttt. $  Let $\ggg$, $\bbb$, $\nnn$, and $\ttt$ be respectively the Lie algebra of $G$, $B$, $N$, and $T$. 

$\bullet  G_v, B_v,\mathcal{K}_v, \mathcal{I}_v .$   Given a variety $V$ defined over $\mathbb{F}_q$, we will use $V_v$ to denote $V(F_v)$  for any places $v\in \vert  X\vert  $. This notation applies in particular to $G_v$, $B_v$. 
We will denote $G(\mathcal{O}_v)$ by $\mathcal{K}_v$. Let $\mathcal{I}_v$ be the Iwahori subgroup consisting of matrix in $\mathcal{K}_v$ whose reduction modulo $\wp_v$ lies in $B(\kappa_v)$. 

$\bullet G(\mathbb{A})^{e}.$ For any $e\in \mathbb{Z}$, let \[ G(\mathbb{A})^{e}= \{x\in G(\AAA)| \deg\det x=e\}.\] 
Here the degree map of $\mathbb{A}^{\times}$ is normalized so that for $(a_v)_{v\in |X|}$,
\[ \deg (a_v)_{v\in |X|} = -\sum_{v\in }v(a_v) [\kappa_v: \mathbb{F}_q].  \]

$\bullet$ We fix Haar measures on $G(\AAA)$, $N(\AAA)$ so that $G(\mathcal{O})$ and $N(F)\backslash N(\AAA)$ (with counting measure on $N(F)$) have volume $1$. The local Haar measures on $G_v$, $B_v$ and $N_v$ are defined so that respectively the volumes of $\mathcal{K}_v$, $B(\mathcal{O}_v)$ and $N(\mathcal{O}_v)$ are $1$.

$\bullet \mathcal{E}_v, \varphi_v.$ Given a vector bundle $\mathcal{E}$ over $X$, and a place $v\in |X|$ identified as a $\kappa_v$-point of $X$, 
we use $\mathcal{E}_v$ to denote the fiber over $v$. It is a $\kappa_v$-vector scheme. Suppose $\varphi: \mathcal{E}\longrightarrow \mathcal{F}$ be a bundle morphism over $X$, then it induces a $\kappa_v$-linear map $\varphi_v: \mathcal{E}_v\longrightarrow \mathcal{F}_v$.

\section{Global and local Langlands correspondence for $GL_2$}
We are going to reduce the calculation of the cardinality of ${E}_2(\mathfrak{R})^{\Fr^{*k}}$ to a question of counting certain automorphic representations of $GL_2$ with the help of global Langlands correspondence in rank $2$ established by Drinfeld. 
Note that \[ \overline{X}=(X\otimes_{\mathbb{F}_q}\mathbb{F}_{q^k}) \otimes_{\mathbb{F}_{q^k}}\otimes\overline{\mathbb{F}}_q,  \]
and the Frobenius endomorphism of $\overline{X}$ deduced from $X\otimes_{\mathbb{F}_q}\mathbb{F}_{q^k}$ is $\Fr^k$. Therefore, we can do the calculation for $k=1$ and apply the results to the curves $X\otimes_{\mathbb{F}_q}\mathbb{F}_{q^k}$ over $\mathbb{F}_{q^k}$ ($k\geq 1$) later.

\subsection{Galois representations}\label{Galoisrep}
It has been explained by Deligne \cite[2.1-2.9]{Deligne} how to pass to the automorphic side, and the reader is invited there for more details. 
This section aims to give precise information on the ramifications of automorphic representations determined by the Frobenius action on $\mathfrak{R}$. The data on the local monodromies are carried over to the automorphic side, described by local Langlands correspondence. 
 
 We continue to use notation in the introduction. Let $v\in S$ and $x\in \overline{S}$ that lies over $v$. 
 
 We fix an algebraic closure $\overline{F}$ of $F$. Then $\overline{\eta}:=\Spec(\overline{F})$ is a geometric point lying over the generic point of $X$. 
 Let $\overline{X}_{(x)}$ be the Henselization of $\overline{X}$ in $x$ and $\overline{X}_{(x)}^{\ast}=\overline{X}_{(x)}-\{x\}$. If we choose an embedding of $\overline{\mathbb{F}}_q(\overline{X}_{(x)}^{\ast})$ in $\overline{F}$, then
the \'etale fundamental group $\pi_1(\overline{X}_{(x)}^{*}, \overline{\eta})$ 
is canonically isomorphic to the inertial group $I_x=\Gal(\overline{F}/\overline{\mathbb{F}}_q(\overline{X}_{(x)}^{\ast}) )$. An $\ell$-adic local system over $\overline{X}-\overline{S}$ (resp. $\overline{X}_{(x)}^{\ast}$) is equivalent to an $\ell$-adic  representation of $\pi_1(\overline{X}-\overline{S}, \overline{\eta})$ (resp. $I_x$). 
Let $I^{t}_x=\pi_1(\overline{X}_{(x)}^{\ast}, \overline{\eta})^{t}$ be the tame fundamental group of  $\overline{X}_{(x)}^{\ast}$, i.e. the maximal prime-to-$p$ quotient of $\pi_1(\overline{X}_{(x)}^{\ast}, \overline{\eta})$.  
A tame  $\ell$-adic  local system of $\overline{X}_{(x)}^{\ast}$ is equivalent to an $\ell$-adic representation of $I_x^t$.

For an algebraic closed field $k$, let \[ \hat{\mathbb{Z}}^{p'}(1)(k):= \lim \mu_n(k),\]
the projective limit where the transition map is the norm map. 
We denote $\hat{\mathbb{Z}}^{p'}(1)$ for $\hat{\mathbb{Z}}^{p'}(1)(\overline{\mathbb{F}}_q)$. 
Let $\kappa_x\cong\overline{\mathbb{F}}_q$ be the residue field of the point $x$,
we have a canonical isomorphism \[ I_x^{t}\cong  \hat{\mathbb{Z}}^{p'}(1)(\kappa_x).  \] 
Choose an embedding $\kappa_v\hookrightarrow \overline{\mathbb{F}}_q$, we deduce from
\[\kappa_v\otimes_{\mathbb{F}_q}  \overline{\mathbb{F}}_q \cong \prod_{x\mapsto v}  \overline{\mathbb{F}}_q, \]
 isomorphisms $\kappa_x\cong \overline{\mathbb{F}}_q$. This gives us an isomorphism \[I_x^t\cong \hat{\mathbb{Z}}^{p'}(1). \]
 The morphism $\Fr: \overline{X}_x^{*}\longrightarrow \overline{X}_{\Fr(x)}^{*}$, induces an isomorphism between tame fundamental groups. It is the multiplication by $q$ map on $\hat{\mathbb{Z}}^{p'}(1)$ via the above isomorphisms.

To make the set $E_2(\mathfrak{R})$ $\Fr^\ast$-stable,  the $\ell$-adic local systems $(\mathfrak{R}_x)_{x\mapsto v}$ have to satisfy the compatibility condition \begin{equation}\label{astp}\Fr^{*}(\mathfrak{R}_{\Fr(x)})\cong \mathfrak{R}_x.  \end{equation}
Let $I_v$ and $D_v$ be, respectively, the inertial subgroup and decomposition group of $F$ at $v$. 
The condition \eqref{astp} implies that $(\mathfrak{R}_x)_{x\mapsto v}$ come from a representation $\rho_v$ of $D_v$. By Grothendick's local monodromy theorem, $\rho_v|_{I_v}$  is quasi-unipotent in the sense that it becomes unipotent on an open subgroup. We will use $\mathfrak{R}_v$ to denote $\rho_v|_{I_v}$. 

Let $W_F$ be the Weil group of $F$. Then $\pi_1(\overline{X}-\overline{S}, \overline{\eta})$ is a quotient of the degree $0$ part of $W_F$.  
 Let $\mathcal{G}_2(F)$ be the set of isomorphism classes of $\ell$-adic representation of $W_F$. We call two $\ell$-adic representations $\sigma_1$ and $\sigma_2$ in $\mathcal{G}_2(F)$ inertially equivalent: $\sigma_1\sim \sigma_2$  if there is a character $\lambda: W_F\xrightarrow{\deg} \mathbb{Z}\longrightarrow \overline{\mathbb{Q}}_{\ell}^{\times}$ such that $\sigma_1\cong \sigma_2\otimes\lambda$. 
  
\begin{prop}\label{Galois}
The set \[E_2(\mathfrak{R})^{\Fr^{*k}}\] 
 is in bijection with the subset of inertially equivalent classes in $\mathcal{G}_2(F\otimes \mathbb{F}_{q^k})/\sim$ consisting of $\sigma$ such that 
 \begin{equation}\label{lambda}\sigma\otimes \lambda\cong \sigma \implies \lambda=1, \end{equation}
$\sigma\vert  _{I_v}$ is trivial for $v\notin S$ and \[\sigma\vert  _{I_v}\cong \mathfrak{R}_v \] for $v\in S$. 
\end{prop}
\begin{proof}
It has been explained in \cite[Section 2]{Deligne}. The condition \eqref{lambda} is to make sure that $\sigma$ is irreducible restricting to the degree $0$ part (see \cite[Proposition 2.1.3]{Yu1}). 
 \end{proof}

Let $W_v$ be the local Weil group at $v$. We choose an isomorphism $\iota: \overline{\mathbb{Q}}_{\ell}\xrightarrow{\sim}\mathbb{C}$. 
Recall that the local Langlands correspondence is a canonical bijection between the set of smooth irreducible $\mathbb{C}$-representations of $G_v$ and the set of rank $2$, Frobenius semisimple $\ell$-adic (continuous) representations of $W_v$. 

For any $v\in S$, let $\mathrm{Irr}^{\mathfrak{R}}(G_v)$ be the set of irreducible representations of $G_v$ whose associated $\ell$-adic representation of the local Weil group $W_v$ under local Langlands correspondence extends $\mathfrak{R}_v$. For a place $v\notin S$, we define $\mathrm{Irr}^{\mathfrak{R}}(G_v)$ to be the set of unramified representations of $G_v$, i.e., those representations whose associated $\ell$-adic representation of $W_v$ under local Langlands correspondence is trivial when restricting to $I_v$.

 We have the following theorem that characterizes the set $\mathrm{Irr}^{\mathfrak{R}}(G_v)$ purely by their representation theoretic structures. 

\begin{theorem}\label{mono}
Let $\mathfrak{R}_v$ be tame. We have one of the following cases. 
\begin{enumerate}
\item[(r)] We say that $\mathfrak{R}_{v}$ is (split) regular if 
\[\mathfrak{R}_{v}\cong \chi_1\oplus \chi_2, \] is the direct sum of two distinct characters $\chi_1, \chi_2$ of $I_v$. Each $\chi_i$ ($i=1, 2$)  has exponent $q_v-1$ and can be factored as $I_v\longrightarrow \kappa_v^{\times} \xrightarrow{\chi_{i}'} \overline{\mathbb{Q}}_\ell^{\times}$. 

In this case for any irreducible smooth representation $\pi$ of $G_v$, we have
 $\pi\in \Irr^{\mathfrak{R}}(G_v)$ if and only if $\Hom_{\mathcal{I}_v}(\chi_v, \pi)\cong \Hom_{\mathcal{K}_v}(\rho_v, \pi) \neq 0$, where
 \[\chi_v: \mathcal{I}_v\longrightarrow B(\kappa_v)\xrightarrow{\begin{pmatrix}a&b\\0&d
\end{pmatrix}\mapsto \iota(\chi_{1}'(a)\chi_2'(d))
} {\mathbb{C}}^{\times},  \]
and $\rho_v$ is the induced representation of $\chi_v$ to $\mathcal{K}_v$. 
Moreover,  $\dim \Hom_{\mathcal{I}_v}(\chi_v, \pi) =1$ for any $\pi\in \Irr^{\mathfrak{R}}(G_v)$.

\item[(c)] We say that $\mathfrak{R}_{v}$ is cuspidal (or anisotropic regular), if 
\[\mathfrak{R}_{v}\cong \chi_1\oplus \chi_2,\] is the direct sum of two distinct characters of $I_v$ such that $\chi_1^{q_v}=\chi_2$ (necessarily we also have $\chi_2^{q_v}=\chi_1$), where $\chi_i$ can be factored as ${I}_v\longrightarrow \mathbb{F}_{q_v^2}^{\times} \xrightarrow{\chi_i'} \overline{\mathbb{Q}}_\ell^{\times}$.

In this case for any irreducible smooth representation $\pi$ of $G_v$, we have
 $\pi\in \Irr^{\mathfrak{R}}(G_v)$ if and only if $\dim \Hom_{\mathcal{K}_v}(\rho_v, \pi)\neq 1$, where $\rho_v$ is the irreducible representation of $\mathcal{K}_v$ inflated from the Deligne-Lusztig induced  representation \[- \iota(R_{U(\kappa_v)}^{G(\kappa_v)}\chi_1'), \] 
with $U$ being any non-split maximal subtorus of $G$ defined over $\kappa_v$ so that we have $U(\kappa_v)\cong \mathbb{F}_{q_v^2}^{\times}$.  
Moreover, for any $\pi\in \Irr^{\mathfrak{R}}(G_v)$, we have $\dim \Hom_{\mathcal{K}_v}(\rho_v, \pi)=1$  and $\pi$ is supercuspidal.

\item[(s)] We say that $\mathfrak{R}_{v}$ is scalar if 
\[\mathfrak{R}_{v}\cong \chi^{\oplus 2},\]
where $\chi$ is a character of $I_v$ that can be factored as ${I_v}\longrightarrow \kappa_v^{\times} \xrightarrow{\chi'} \overline{\mathbb{Q}}_\ell^{\times}$.  

In this case for any irreducible smooth representation $\pi$ of $G_v$, we have $\pi\in \Irr^{\mathfrak{R}}(G_v)$ if and only if $\Hom_{\mathcal{K}_v}(\theta_v, \pi)\neq 0$ where \[\theta_v: \mathcal{K}_v\xrightarrow{\det} \mathcal{O}_v^{\times}\longrightarrow \kappa_v^{\times}\xrightarrow{\iota\circ\chi'}{\mathbb{C}}^{\times}. \]
Moreover, we have $\dim \Hom_{\mathcal{K}_v}(\theta_v, \pi)=1$ for any $\pi\in \Irr^{\mathfrak{R}}(G_v)$. And the set $\Irr^{\mathfrak{R}}(G_v)$ consists of one dimensional representations $\eta$ of $G_v$ that extend $\theta_v$ and those twists by $\eta$ of irreducible unramified representations of $G_v$.

\item[(u)] We say that $\mathfrak{R}_v$ is principal quasi-unipotent if 
\[\mathfrak{R}_{v}\cong \chi\otimes \nu,\] is a quasi-unipotent with one principal Jordan block: $\chi$ is a character of $I_v$ that can be factored as ${I_v}\longrightarrow \kappa_v^{\times} \xrightarrow{\chi'} \overline{\mathbb{Q}}_\ell^{\times}$ and $\nu$ is a non-trivial unipotent representation of $I_v$. 

In this case $\pi\in \Irr^{\mathfrak{R}}(G_v)$ if and only if $\pi= St\otimes \lambda$ for some character $\lambda$ of the form $G_v\xrightarrow{\det} F_v^\times \xrightarrow{\lambda'} \mathbb{C}^\times$ so that $\theta_v=\lambda'\vert  _{\mathcal{O}_v^\times}$ inflates $\iota\circ \chi'$. Here $St$ is the Steinberg representation of $G_v$, i.e., the unique irreducible quotient of the parabolic induction of the trivial representation of $B_v$.    

\end{enumerate}
\end{theorem}
\begin{proof}
The cases $(r)$ and $(c)$ are explained in \cite[Theorem 2.4.1, 2.4.4]{Yu}. The case $(s)$ is deduced from the unramified case where $\dim \pi_v^{\mathcal{K}_v}=1$. In fact, it is enough to tensor $\mathfrak{R}_v$ by a rank $1$ local system to make it trivial on $I_v$. Similarly, we can tensor a character to make the case $(u)$ into the unipotent case, which corresponds to the Steinberg representation under the Langlands correspondence. 
\end{proof}

Let $\pi=\otimes' \pi_v$ be a cuspidal automorphic representation of $G(\AAA)$. We will say that $\pi_v$ has the correct ramification type (for our counting problem) if $\pi_v\in \Irr^{\mathfrak{R}}(G_v)$.

\subsection{Automorphic representations}
Let $C_{cusp}(G(\AAA))$ be the space of cuspidal automorphic forms. Recall that a cuspidal automorphic form is a complex-valued function $\varphi$ over $G(F)\backslash G(\AAA)$ that generates a finite-dimensional vector space under $G(\mathcal{O})$-right translation and $Z_{G}(\AAA)$ translation, such that the following cuspidality condition is satisfied 
\[\int_{N(F)\backslash N(\AAA)}\varphi(nx)\d n, \quad \forall x\in G(\AAA).  \]
Note that the above integration is a finite sum because of $G(\mathcal{O})$-finiteness. 
The right translation by $G(\AAA)$ makes $C_{cusp}(G(\AAA))$ into a $G(\AAA)$-representation. It is known to be semisimple and the multiplicity one theorem of Jacquet $\&$ Langlands and Piatetski-Shapiro implies that $C_{cusp}(G(\AAA))$ is multiplicity free. Its irreducible summands are called cuspidal automorphic representations. A cuspidal automorphic representation $\pi$ can be decomposed as a restricted tensor product $\pi=\otimes' \pi_v$ for representations $\pi_v$ of $G_v$ which are called local components of $\pi$.

Let $\mathcal{A}_2(F)$ be the set of isomorphic classes of cuspidal automorphic representations of $G(\AAA)$. We call two cuspidal automorphic representations $\pi_1$ and $\pi_2$ are inertially equivalent $\pi_1\sim\pi_2$ if there are is a  character $\lambda: G(\AAA)\xrightarrow{\deg\circ\det}\mathbb{Z}\longrightarrow \mathbb{C}^\times$ such that $\pi_1\cong \pi_2\otimes \lambda$. 

\begin{theorem}\label{Langlands}
The set $E_2(\mathfrak{R})^{\Fr^{*}}$ is in bijection with the subset of $\mathcal{A}_2(F)/\sim$ consisting of inertial equivalent classes of  cuspidal automorphic representations $\pi$ such that for any character $\lambda: G(\AAA)\xrightarrow{\deg\circ\det}\mathbb{Z}\longrightarrow \mathbb{C}^\times$, 
\[\pi\otimes \lambda \cong \pi \implies \lambda=1 ,\]
and $\pi_v\in \Irr^{\mathfrak{R}}(G_v)$ for all $v\in S$.   
\end{theorem}
\begin{proof}
Applying global Langlands correspondence and the fact that it is compatible with local Langlands correspondence, this is a corollary of Proposition \ref{Galois} and Theorem \ref{mono}. 
\end{proof}

\section{Spectral side of the trace formula}
We will use equality provided by the noninvariant Arthur-Selberg trace formula (a similar but slightly different formula is obtained first by Jacquet-Langlands for $GL_2$) established by Lafforgue \cite{Laf}. Indeed the noninvariant Arthur-Selberg trace formula over a function field is an equality for each $e\in \mathbb{Z}$ between two distributions on ${C}_c^\infty(G(\AAA))$:
\[ J^{e}_{geom}(f)= J^{e}_{spec}(f). \] 
We will construct a function $f\in {C}_c^\infty(G(\AAA))$ using Theorem \ref{mono} and do explicit calculations for  $J^{1}_{spec}(f)$. The result is summarized by Theorem \ref{automorphe}. This theorem says that it is always a sum of $|E_2(\mathfrak{R})^{\Fr^\ast}|$ and an explicit error term. In a later section, we will use a geometric method to study $J^{e}_{geom}(f)$, which induces a relation with the number of $\mathbb{F}_q$-points of Hitchin moduli spaces.

\subsection{Explicit spectral decomposition of $J^{1}_{spec}(f)$}\label{4.1}
Let $M$ be either $T$ or $G$. 
Let $X_M$ be the group of  characters of  $M(\mathbb{A})$ to $\mathbb{C}^{\times}$ that are trivial on  $M(\mathbb{A})^{0}$.  
Let $X_{M}^{G}\subseteq X_M$ be the subgroup consisting of those characters that are trivial on $Z_{G}(\mathbb{A})$, i.e.,	  \[X_M^G=\Hom(M(\mathbb{A})^{0}\backslash M(\mathbb{A})/Z_{G}(\AAA), \mathbb{C}^{\times}). \]
We have \[X_G^G= \{1, \epsilon\}, \]
where $\epsilon(x)=(-1)^{\deg (\det x) }$ is the sign character of $G(\AAA)$. We identify $X_G^G$ with $\{\pm 1\}\subseteq \mathbb{C}^{\times}$. 
We also have an identification 
\[X_T^G\cong \mathbb{C}^{\times}, \]
where for any $\lambda\in \mathbb{C}^{\times}$ we associate a character \[\lambda(\begin{pmatrix} a&0\\0& b
\end{pmatrix}
) = \lambda^{\deg(a)-\deg(b)}. \] 
We will use this isomorphism in our calculations. Let $\Im X_T^G$ be the subgroup of $X_T^G$ consisting of unitary characters. Therefore, it is formed by elements $\lambda\in \mathbb{C}^{\times}$ of absolute value $1$. We endow a Haar measure on $X_G^G$ and $\Im X_T^G$ so that the total volume is $1$.

Following Jacquet-Langlands, we have the spectral decomposition: \[L^2(G(F)\backslash G(\AAA))\cong L^2_{cusp}\oplus L^{2}_{res} \oplus L^{2}_{cont}, \] 
where $L^2_{cusp}\oplus L^{2}_{res}$ is the largest semisimple subspace, $L^{2}_{cont}$ is its orthogonal complement, and  $L^2_{cusp}$ is the completion of the space of cuspidal automorphic forms. 
The residual spectrum is decomposed as \[L^{2}_{res}\cong \hat{\bigoplus_{\chi}} \chi, \]
where $\chi$ are compositions of Hecke characters with the determinant morphism. The sum here is the Hilbert direct sum. 
For continuous spectrum $L^{2}_{cont}$, we have a decomposition: \[L^{2}_{cont}=\hat{\bigoplus_{\psi}} L^{2}_{[B, \psi]},   \]
where the sum is taken over the set of inertial equivalent classes of pairs $(B, \psi)$ with $\psi$ being a Hecke character of $T(\AAA)\cong (\mathbb{A}^{\times})^2$. The explicit construction of $L^{2}_{[B, \psi]}$ is given by the theory of Eisenstein series. 

We fix an idèle $a\in \mathbb{A}^\times $ of degree $1$, viewed as a scalar matrix. 
Let $f\in C_c^\infty(G(\AAA))$, it acts via the regular representation on $L^2(G(F)\backslash G(\AAA)/a^\mathbb{Z})$,  equivalently by convolution from right by $\breve{f}:=(x\mapsto f(x^{-1}))$, which is an integral operator. We denote this action by the function $R(f)$. Therefore its trace, if it exists, will be the integration of the kernel function on the diagonal. However, the trace of $f$ on the whole space $L^2(G(F)\backslash G(\AAA)/a^\mathbb{Z})$ does not exist in general. Arthur defines a truncated kernel function so that its integration on the diagonal will contain the information $\Tr(f|L^2_{cusp})$ and expresses this integration in two ways: a geometric expansion and a spectral expansion so that we have an identity. 
The spectral expansion contains a piece that gives the most interesting part $\Tr(f|L^2_{cusp})$, and we usually hope to obtain information from the geometric expansion and an understanding of the error terms in the spectral expansion. 

Over a function field, we have a decomposition \[ G(\AAA)=\coprod G(\AAA)^{e}, \]
and $G(F)\backslash G(\AAA)^{e}$ has finite volume. The two different ways to express Arthur's truncated integral over the diagonal in $G(F)\backslash G(\AAA)^{e}\times G(F)\backslash G(\AAA)^{e}$ give an identity:
\[J^{e}_{geom}(f)=J^{e}_{spec}(f). \]
It is slightly simpler to consider an odd integer $e$ or simply that $e=1$. The following result is a special case of the formula obtained by L. Lafforgue.

Let $\mathcal{A}_1$ be the set of Hecke characters of $F^\times \backslash \mathbb{A}^{\times}/a^\mathbb{Z} $. Let \[\mathcal{A}_{cont}:= \mathcal{A}_1\times\mathcal{A}_1/\mathfrak{S}_2,  \]
where $\mathfrak{S}_2$ acts by permutation. An element $[(\psi_1, \psi_2)] \in \mathcal{A}_{cont}$ is called regular if $\psi_1\neq \psi_2$ and is called non-regular otherwise. Let \[\mathcal{A}_{res}\] be the inertial equivalent classes of $1$-dimensional 
representations of $G(\AAA)$ trivial on $a^\mathbb{Z}G(F)$. 
Let \[\mathcal{A}_0\] be the set inertial equivalent classes of cuspidal automorphic representations of $G(\AAA)$ whose central characters are trivial on $a^\mathbb{Z}$.

Let $\mathcal{A}_{B,\psi}$ be the space of complex valued functions $\varphi$ over $G(\AAA)$  satisfying that for any $k\in G(\mathcal{O})$, there is a constant $c_k\in \mathbb{C}$ so that for any $n\in N(\AAA)$, $t\in T(\AAA)$ we have \[\varphi(ntk)=c_k\rho_B(t)\psi(t), \]
 where $\rho_B(\begin{pmatrix} a& 0\\ 0& c
 \end{pmatrix})=\frac{\vert  a\vert  ^{\frac{1}{2}}}{\vert  c\vert  ^{\frac{1}{2}}}$.  Equivalently, it is the space of those $\varphi$ such that for any $x\in G(\AAA)$, $n\in N(\AAA)$ and $t\in T(\AAA)$ we have $\varphi(ntx)=   \rho_B(t)\psi(t)\varphi(x)$. 
 Let $w$ be the non-trivial element in the Weyl group of $(G,T)$ and $\lambda\in X_T^G$. 
 We have the intertwining operator $\mathcal{A}_{B,\psi}\longrightarrow \mathcal{A}_{B,w(\psi)}$ defined by  analytic continuation of the integral below which converges when $|Re \lambda|>>0$, 
 \begin{equation}\label{entre}({\M(w,\lambda)}\varphi)(x):=\lambda(x)\int_{ {N}(\mathbb{A})}\varphi(w^{-1}nx) \lambda(w^{-1}nx) \d n, \quad \end{equation} 
where we view $\lambda\in X_T^G$ as a function over $G(\AAA)$ using Iwasawa decomposition, i.e., if $x=ntk$ with $n\in N(\AAA)$, $t\in T(\AAA)$ and $k\in \mathcal{K}$, we define $\lambda(x):=\lambda(t)$. 

\begin{theorem}[Arthur, Lafforgue]\label{spectral}
The spectral expansion is the identity
\[J_{spec}^{1}(f) =\sum_{[\pi]\in \mathcal{A}_0} J_\pi(f) + \sum_{[\chi]\in\mathcal{A}_{res}} J_\chi(f)+\sum_{[\psi]\in \mathcal{A}_{cont}}J_{\psi}(f), \]
where each sum is taken over a set of representatives, and the terms are defined as follows. We denote the three sums by respectively  $J^{1}_{cusp}(f)$, $J_{res}^{1}(f)$ and $J_{cont}^{1}(f)$. 

For $\pi\in \mathcal{A}_0$,  if $\pi\otimes \epsilon \cong \pi$, then
 \[J_{\pi}(f)= \frac{1}{2}(\mathrm{Tr}(R(f)\vert  \pi) - \mathrm{Tr}(R(f)\circ \epsilon \vert  \pi) ), \]
 and  if $\pi\otimes \epsilon \not\cong \pi$, then
 \[J_{\pi}(f) =  \mathrm{Tr}(R(f)\vert  \pi). \]

 For   $\chi\in \mathcal{A}_{res}$, we have 
 \[J_{\chi}(f)=\mathrm{Tr}(R(f)\vert  \chi). \]

 For any $\lambda\in X_T^G$, let $R(f,\lambda)$ be the twisted action on $\mathcal{A}_{B,\psi}$: \[R(f,\lambda)\varphi = (R(f)(\varphi \lambda))\lambda^{-1},  \]
 where we view $\lambda\in X_T^G$ as a function over $G(\AAA)$ by Iwasawa decomposition: $\lambda(x)=\lambda(t)$ if $x=ntk$ for $n\in N(\AAA)$, $t\in T(\AAA)$ and $k\in G(\mathcal{O})$. 
 For $\psi\in \mathcal{A}_{cont}$, if $\psi$ is regular, then
  \[J_{\psi}(f)= \int_{\Im X_{T}^{{G}}}
\lim_{\mu\longrightarrow 1}  \mathbf{Tr}_{\mathcal{A}_{B,{\psi}}}(( - \frac{1}{\mu^{-1} - \mu}    \M(w, \lambda)^{-1}\circ \M(w,\lambda/\mu )+ \frac{1}{\mu^{-1} - \mu} ) \circ  R(f, \lambda)) \d\lambda ,
 \] 
and if $\psi$ is not-regular,  then
  \begin{multline}
J_{\psi}(f)=\frac{1}{2}  \int_{\Im X_{T}^{{G}}}
\lim_{\mu\longrightarrow 1}  \mathbf{Tr}_{\mathcal{A}_{B,{\psi}}}(( - \frac{1}{\mu^{-1} - \mu}    \M(w,\lambda)^{-1}\circ \M(w,\lambda/\mu )+ \frac{1}{\mu^{-1} - \mu} ) \circ  R(f, \lambda)) \d\lambda \\
+\frac{1}{8}  \sum_{       \lambda_{G}\in \{  \pm 1 \}     }    \sum_{\substack{ \lambda_{w}\in \Im X_{T}^{G}   \\ \lambda_{w}^2=\lambda_{G}^{-1}   }       }   {\lambda_{G}}   \mathbf{Tr}_{\mathcal{A}_{B,{\psi}}}(   
  \M(w,  w^{-1}({\lambda_{w}})) \circ R(f, \lambda) )   . 
\end{multline} 
Note that if $f$ is supported in $G(\mathcal{O})$, then $R(f, \lambda)=R(f)$. 
  \end{theorem}
\begin{proof} 
The theorem is established by L. Lafforgue in \cite{Laf}, and we refer the reader to \cite[Th. 5.2.2, Co. 5.2.3]{Yu1}. For the group $G=GL_2$, we can make the result more explicit. In \cite[5.2.1]{Yu1} we have calculate explicitly the functions $\hat{\mathbbm{1}}_{G}^1$, $\hat{\mathbbm{1}}_{B}^1$ and $\hat{\mathbbm{1}}_{\overline{B}}^1$ 
on $X_T^G$, which appears in the spectral expansion of the trace formula.  They are given by the following formula: 
\[\hat{\mathbbm{1}}_{G}^{1}(1) = 1 \quad \text{ and }  \quad  \hat{\mathbbm{1}}_{G}^{1}(\epsilon) = -1  , \]
i.e., 
\[\hat{\mathbbm{1}}_{G}^{1}(\lambda) = \lambda, \quad \forall \lambda\in X_G^G, \]
and 
\[\hat{\mathbbm{1}}_{B}^{1}(\lambda) = - \frac{1}{\lambda-\lambda^{-1}},\]
\[\hat{\mathbbm{1}}_{\overline{B}}^{1}(\lambda) = - \frac{1}{\lambda^{-1} - \lambda}. \]

\end{proof}

\subsection{Summary of main results of this section}
The primary purpose of this section is to construct a specific function $f\in C_c^{\infty}(G(\AAA))$ (Proposition \ref{f}), and calculates explicitly the spectral side of the trace formula $J^{1}_{spec}(f)$ (Theorem \ref{spectral}). 

\begin{prop} \label{f}
We use notation of Theorem \ref{mono}. For each $v\in S$, we define 
 the following functions following the ramification type of $\mathfrak{R}_v$. 
\begin{itemize}
\item[(r)] Let $v\in S_r$.  
The function $f_v\in C_c^{\infty}(G_v)$ is defined by
\[f_v(x)=\begin{cases} 0,\quad  x\notin \mathcal{K}_v ;\\ \Tr(\rho_v(\overline{x}^{-1})), \quad x\in \mathcal{K}_v . \end{cases} \]
where $\overline{x}$ is the image of $x$ in $G(\kappa_v)$ under the projection $\mathcal{K}_v\longrightarrow G(\kappa_v)$. 
\item[(c)] Let $v\in S_c$
The function $f_v\in C_c^{\infty}(G_v)$ is defined by
\[f_v(x)=\begin{cases} 0,\quad  x\notin \mathcal{K}_v ;\\ \Tr(\rho_v(\overline{x}^{-1})), \quad x\in \mathcal{K}_v . \end{cases}  \]
\item[(s)] Let $v\in S_s$
The function $f_v\in C_c^{\infty}(G_v)$ is supported in $\mathcal{K}_v $ such that for any $x\in \mathcal{K}_v $, we have
\[f_v(x)=\theta_v(\det \overline{x}^{-1}); \] 
\item[(u)]  Let $v\in S_u$.
The function $f_v\in C_c^{\infty}(G_v)$ is defined by
\[f_v=\left(\frac{1}{\vol(\mathcal{I}_v)}\mathbbm{1}_{\mathcal{I}_v}(x)-2 \mathbbm{1}_{\mathcal{K}_v}(x)\right)\theta_v(\det \overline{x}^{-1}). \]
\end{itemize}
Let \[f = \otimes f_v\in C_c^{\infty}(G(\AAA)), \]
where $f_v$ is the function defined above if $v\in S_{(?)}$ for $?=r, c, s, u$ and $f_v=\mathbbm{1}_{\mathcal{K}_v}$ if $v\notin S$. 
Then for any cuspidal automorphic representation $\pi$ of $G(\mathbb{A})$, the condition \[\Tr(f\vert  \pi)\neq 0 \]
holds if and only if $\pi$ has correct ramification type. 
 If it is the case, we have \[\Tr(f\vert  \pi)=1.\]
\end{prop}
\begin{proof}
Suppose $\pi=\otimes'\pi_v$ is a cuspidal automorphic representation,
we have \[\Tr(f\vert \pi)= \prod_{v}\Tr(f_v\vert \pi_v).\]
The statement is, therefore, of local nature. 

For $v\notin S$, the function $\mathbbm{1}_{\mathcal{K}_v}$ acts as a projection to $\mathcal{K}_v$-fixed part of $\pi_v$. It is non-zero if and only if $\pi_v$ is unramified by definition. If this is the case, $\pi_v^{\mathcal{K}_v}$ is an $1$ dimensional vector space and the trace of $\mathbbm{1}_{\mathcal{K}_v}$ is one. For $v\in S_s$, it is similar. In fact, suppose that $\pi_v=\pi_v^{\prime}\otimes \eta$ with $\pi_v^{\prime}$ being unramified.  Then 
\[\Tr(f_v\vert \pi_v)=\Tr(\mathbbm{1}_{\mathcal{K}_v}\vert \pi_v^{\prime}). \] 
The trace is $1$ or $0$ depending on whether or not $\pi_v$ has the correct ramification.

For $v\in S_{c}$ and $v\in S_{r}$, this result has been proved in \cite[Theorem 2.4.4]{Yu}.

Finally, we consider the case  $v\in S_{u}$. 
As above, up to a twist, we are reduced to the case that $\chi'$ is trivial. 
Note that if \[\Tr(f_v\vert \pi_v)\neq 0, \]
then $\pi_v$ contains a non-zero fixed vector under $\mathcal{I}_v$. By a result of Casselman \cite[Th. 7.4.4]{Laumon}, $\pi_v$ must be an irreducible subrepresentation of a parabolic induction $\Ind_{B_v}^{G_v}\theta_v$ for a character $\theta_v$ of $T_v$ that is trivial on $T(\mathcal{O}_v)$. If the representation
$\Ind_{B_v}^{G_v}\theta_v$ is irreducible, then it is unramified. Hence its $\mathcal{I}_v$-fixed subspace has dimension $2$. We deduce that \[\Tr(f_v\vert \pi_v)= 0. \]
If $\Ind_{B_v}^{G_v}\theta_v$ is not irreducible, then it is of length $2$ whose irreducible quotients are a $1$-dimensional representation and an unramified twist of the Steinberg representation. Since $\pi$ is cuspidal, $\pi_v$ cannot be $1$-dimensional, therefore $\pi_v$ is an  unramified twist of the Steinberg representation whose $\mathcal{I}_v$-fixed subspace is of $1$ dimensional and $\mathcal{K}_v$-fixed subspace is $0$. This completes the proof. 
\end{proof}

The following result is deduced from Lemma \ref{bck}, Corollary \ref{cuspt}, Proposition \ref{rest}, and Proposition \ref{continoust}. Notation are those of the introduction. 
\begin{theorem}\label{automorphe}
For a finite set of places $V$ of $F$, we define $\deg V:= \sum_{v\in V}\deg v=|\overline{V}|$. 
Let $S_{u,even}$ be the subset of $S_u$ consisting of places $v$ such that $\deg v$ is even. 

The expression $J^{1}_{spec}(f)$ equals the following numbers depending on the cases:
\begin{enumerate} 
\item $S_c\neq \emptyset$,  and $S_{u, even}\neq\emptyset$. 
We have 
\[J^{1}_{spec}(f)=|E_2(\mathfrak{R})^{\Fr^\ast}|. \]

\item $S_c\neq \emptyset$, $S_{u, even}=\emptyset$ but $S_u\neq \emptyset$. 
We have 
\[J^{1}_{spec}(f)=|E_2(\mathfrak{R})^{\Fr^\ast}| +(-1)^{|S_u|+1}  {b_{\mathfrak{R}}(1) }2^{|S_u|-2}(-1)^{\deg S_c}\Pic(2) . \]

\item $S_c\neq \emptyset$, and $S_{u}=\emptyset$. 
We have 
\[J^{1}_{spec}(f)=|E_2(\mathfrak{R})^{\Fr^\ast}|+\frac{b_{\mathfrak{R}}(1)  }{4}(1-(-1)^{\deg S_c})\Pic(2) . \]

\item $S_c=\emptyset$, $S_{r}\neq\emptyset$, and $S_u=\emptyset$. We have
\[J^{1}_{spec}(f)=|E_2(\mathfrak{R})^{\Fr^\ast}|+
\frac{1}{2}c_{\mathfrak{R}}(1) \Pic(1)  ^2(2g-2+\deg S_{r}).
\]

\item $S_{cr}=S_u=\emptyset$. We have
\[J^{1}_{spec}(f)=|E_2(\mathfrak{R})^{\Fr^\ast}|+
c_{\mathfrak{R}}(1)   \Pic(1)  ^2(g-1)+c_{\mathfrak{R}}(1)  \Pic({1}). \] 

\item $S_{cr}=\emptyset$, $S_u=\{v\}$ and $2\mid \deg v$. Then $J^{1}_{spec}(f)$ equals
\[   |E_2(\mathfrak{R})^{\Fr^\ast}|- c_{\mathfrak{R}}(1)  \Pic(1)  + \frac{ \deg v}{2}c_{\mathfrak{R}}(1)    \Pic(1)  ^2  .\]

  \item $S_{cr}=\emptyset$, $S_u=\{v\}$ and $2\nmid \deg v$. Then $J^{1}_{spec}(f)$ equals
\[ 
|E_2(\mathfrak{R})^{\Fr^\ast}| + \frac{c_{\mathfrak{R}}(1)  }{2}\Pic({2}) {-c_{\mathfrak{R}}(1)    \Pic(1)}+c_{\mathfrak{R}}(1)\frac{\deg v}{2}  \Pic(1)  ^2  . 
  \]

  \item $S_{cr}=\emptyset$, $\vert S_u \vert\geq 2$, and $S_{u, even}\neq \emptyset$. Then $J^{1}_{spec}(f)$ equals
  \[ |E_2(\mathfrak{R})^{\Fr^\ast}|+ c_{\mathfrak{R}}(1)  (-1)^{|S_u|}\Pic(1).\]

    \item $S_{cr}=\emptyset$,  $\vert S_u \vert\geq 2$, and $S_{u, even}=   \emptyset$. Then $J^{1}_{spec}(f)$ equals 
\[|E_2(\mathfrak{R})^{\Fr^\ast}|+ c_{\mathfrak{R}}(1) (-1)^{|S_u|} \Pic(1)+   
c_{\mathfrak{R}}(1) (-1)^{|S_u|+1}  2^{|S_u|-2}\Pic(2).\]

\item $S_c=\emptyset$, $S_{r}\neq\emptyset$,  $S_u=\{v\}$, and
$2\mid \deg v$. We have
\[J^{1}_{spec}(f)= |E_2(\mathfrak{R})^{\Fr^\ast}|+\frac{1}{2}c_{\mathfrak{R}}(1)  \Pic(1)  ^2\deg v. 
\]

\item $S_c=\emptyset$, $S_{r}\neq\emptyset$, $S_u=\{v\}$,
and $2\nmid \deg v$. We have

\[J^{1}_{spec}(f)= |E_2(\mathfrak{R})^{\Fr^\ast}|+\frac{1}{2}c_{\mathfrak{R}}(1)  \Pic(1)  ^2\deg v+\frac{b_\mathfrak{R}(1)}{2}\Pic(2). 
\]

\item $S_c=\emptyset$, $S_{r}\neq\emptyset$, $|S_u|\geq 2$, and $S_{u,even}\neq \emptyset$. We have
\[J^{1}_{spec}(f)= |E_2(\mathfrak{R})^{\Fr^\ast}|. \]

\item $S_c=\emptyset$, $S_{r}\neq\emptyset$, $|S_u|\geq 2$, and $S_{u,even}= \emptyset$. 
We have
\[J^{1}_{spec}(f)= |E_2(\mathfrak{R})^{\Fr^\ast}|+(-1)^{|S_u|+1}{b_\mathfrak{R}(1)}{2}^{|S_u|-2}\Pic(2). \]
\end{enumerate}
\end{theorem}

\subsection{Counting $\ell$-adic local systems in rank 1}\label{rank1}
We need to discuss the cases in rank $1$ first, not only for completeness but also because these results will be needed when calculating the cases in rank $2$. It has been dealt with in \cite[Section 6]{Deligne}, thus our discussion will be concise. 

A difference between a number field and a function field is that the function field $F$, hence all its local fields, is an $\mathbb{F}_q$-algebra. Given a character $\chi_v: F_v^\times\longrightarrow \overline{\mathbb{Q}}_{\ell}^{\times}$, we can consider its restriction to $\mathbb{F}_q^{\times}$. Let $\chi$ be an $\ell$-adic Hecke character \[\chi=\prod_{v}\chi_v: F^{\times}\backslash \mathbb{A}^{\times} \longrightarrow \overline{\mathbb{Q}}_{\ell}^{\times} .\] 
It is a character of $\mathbb{A}^{\times}$ trivial on $F^{\times}$, in particular on $\mathbb{F}_q^{\times}$, therefore necessarily we have \begin{equation}\label{ast} \prod_{v}\chi_v\vert  _{\mathbb{F}_q^{\times}}=1.\end{equation}

Suppose that $\mathfrak{R}_1=(\mathfrak{R}_{1,x})_{x\in \overline{S}}$ is a family of rank $1$ $\ell$-adic local systems over $(\overline{X}^{\ast}_{(x)})_{x\in \overline{S}}$ fixed by $\Fr^{\ast}$. Suppose that they are tame and $\varepsilon_x$ are eigenvalues of the local monodromy actions. 
By similar discussion as in \ref{Galoisrep}, but using local class field theory, we obtain for each $v\in S$ a character $\theta_v$ of $\mathcal{O}_v^{\times}$ trivial on $1+\wp_v$ by tameness. The condition \begin{equation}\label{okay00}\prod_{x\in \overline{S}}\varepsilon_x=1, \end{equation} is equivalent to \begin{equation}\label{okay11} \prod_{v}\theta_v\vert  _{\mathbb{F}_q^{\times}}=1.  \end{equation}
Let $A_1(\mathfrak{R}_1)$ be the set of inertial equivalent classes of Hecke characters $\chi$ of $F^\times\backslash \mathbb{A}^\times$ such that $\chi_v$ extends $\theta_v$. The set $A_1(\mathfrak{R}_1)$ is in bijection with $E_1(\mathfrak{R}_1)^{\Fr^{\ast}}$. 

\begin{lemm}
The condition \eqref{okay00} is satisfied is and only if $A_1(\mathfrak{R}_1)$ is non-empty. 
\end{lemm}
\begin{proof}
By \eqref{ast}, we have seen that \eqref{okay00} is a necessary condition for $A_1(\mathfrak{R})$ to be non-empty. 
Conversely, 
note that $ F^{\times} \backslash  \mathbb{A}^{\times}$ is an extension of $ F^{\times}\backslash \mathbb{A}^{\times}/\mathcal{O}^{\times}\cong \mathrm{Pic}_X(\mathbb{F}_q)$ by $\mathcal{O}^{\times}/ \mathbb{F}_q^{\times} $, the condition \eqref{ast} ensures that $\prod_v\theta_v$ defines a character of $\mathcal{O}^{\times}/ \mathbb{F}_q^{\times}$. Taking ($\ell$-adic) Pontryagin dual, we see immediately that $A_1(\mathfrak{R})$ is non-empty. 
\end{proof}

If the condition \eqref{okay00} is satisfied then we have
\begin{equation}\vert  E_1(\mathfrak{R})^{\Fr^*}\vert  =\vert  \mathrm{Pic}^{0}(X)(\mathbb{F}_q)^{\vee}\vert  =\Pic(1) ,\end{equation}
otherwise \[E_1(\mathfrak{R})^{\Fr^*}=\emptyset. \]
In fact, as long as $  E_1(\mathfrak{R})^{\Fr^*}$  is non-empty, it is a principal homogenous space under $ E_1(\emptyset)^{\Fr^*}$. In this case, it has cardinality $\Pic(1)=\vert  \mathrm{Pic}^{0}(X)(\mathbb{F}_q)\vert $.

\subsection{Eulerian decomposition and calculations on Whittaker functions}
Let $f=\otimes f_v\in {C}_c^\infty(G(\AAA))$ be the function defined in Proposition \ref{f}. 
Let $\pi\cong \otimes' \pi_v$ be a cuspidal automorphic representation of $G(\AAA)$. 

Suppose that $\pi\cong \pi\otimes \epsilon$ as representations of $G(\AAA)$, we need to consider
\[\Tr( \epsilon\circ  R(f)\vert  \pi).\]
Note that this trace is not well-defined from a pure representation theoretical point of view because the action of $\epsilon$ on $\pi$ relies on the isomorphism $\pi\cong \pi\otimes \epsilon$. 
For $G=GL_2$, a canonical isomorphism is furnished by the multiplicity one theorem, which says that $\pi$ and $\pi\otimes \epsilon$ have the same underlying space of cusp forms. 

A similar problem arises if we want to get an Eulerian decomposition of the trace $\Tr( \epsilon\circ  R(f)\vert  \pi)$. We need to choose isomorphisms $\pi_v\cong \pi_v\otimes \epsilon $ so that their tensor product is compatible with the global isomorphism (we also need to assume that the isomorphisms can be glued together to a restricted tensor product isomorphism, i.e., they fix the implicit chosen $\mathcal{K}_v$-invariant vector for almost all places $v$). 
A natural way to do so is by using Whittaker models. 

Let $\psi=\otimes \psi_v: \mathbb{A}\longrightarrow \mathbb{C}^\times$ be an additive character of $\AAA$, viewed as a character of $N(\AAA)$, where $\psi_v$ are  additive characters of $F_v$. Suppose that $\psi_v$ has conductor $\wp_v^{-n_v}$, i.e., is trivial on $\wp_v^{-n_v}$ but not on $\wp_v^{-n_v-1}$. We have
\[\sum_v n_v\deg v=2g-2. \]
Recall that $g$ is the genus of the curve $X$.  Let $\mathcal{W}(\pi)$ be the global Whittaker model of $\pi$ with respect to $\psi$, 
i.e., the space of smooth functions  $\varphi$ over $G(\mathbb{A})$ such that 
\[\varphi(ux)=\psi(u)\varphi(x), \quad \forall u\in N(\AAA) , \forall x\in G(\mathbb{A}), \]
and $\mathcal{W}(\pi_v)$ the local Whittaker model of $\pi_v$ with respect to $\psi_v$ (space of functions over $G_v$ similarly defined). 
Then we have a natural decomposition $\mathcal{W}(\pi)=\otimes_v' \mathcal{W}(\pi_v)$ and $\mathcal{W}(\pi_v)=\mathcal{W}(\pi_v\otimes \epsilon)$. Therefore, we have 
\begin{equation}\label{999}
\Tr( \epsilon\circ  R(f)\vert  \pi)=\prod_{v\in |X|}  \Tr( \epsilon_v\circ R(f_v)\vert  \mathcal{W}({\pi_v})),
\end{equation}
where $\epsilon_v(x_v)=(-1)^{\deg v \deg(\det (x_v))}$ for $v\in G_v$.



\begin{theorem}[Paskunas-Stevens]\label{PStevens}
Let $\rho$ be a representation of $\mathcal{K}_v$ that inflates an (irreducible) cuspidal representation of $G(\kappa_v)$. Let $\pi_v$ be an irreducible representation of $G_v$ that contains $\rho$, i.e. $\Hom_{\mathcal{K}_v}(\rho, \pi)\neq 0$. 
Let $\psi'_v$ be an additive character of $F_v$ of conductor $\wp_v$, i.e., is trivial on $\wp_v$ and is non-trivial on $\mathcal{O}_v$. It defines a character of $N_v$ by $\begin{pmatrix}1&x\\0&1\end{pmatrix}\mapsto \psi'_v(x)$. Let $\mathcal{W}$ be the space of Whittaker functions of $\pi_v$ with respect to $\psi'_v$. Let $\mathcal{W}_\rho$ be the $\rho$ isotypic subspace of $\mathcal{W}$. Then every function in $\mathcal{W}_\rho$ is supported in $N(F_v)Z_v\mathcal{K}_v$. 
\end{theorem}
\begin{proof}
This is a corollary of Paskunas and Stevens' result \cite[Theorem 5.8]{PS}.  Let us explain their notation (see p.1233 of \textit{op. cit.}) from type theory that we need to apply to our specific case. The group $J$ is $\mathcal{K}_v$, the group $\mathbf{J}$ is $Z_v\mathcal{K}_v$, the group $U$ is $N_v$, the representation $\Lambda$ of $\mathbf{J}$ is one that extends $\rho$, 
 and
 the character $\Psi_\alpha$ is obtained by restriction of $\psi'_v$ to $N(\mathcal{O}_v)$ then extends to $N(\mathcal{O}_v)(1+\wp_vM_2(\mathcal{O}_v))$. 

Let $X\in \pi_v$ and $Y\in \pi_v^\vee$, where $\pi_v^\vee$ is the contragredient representation of $\pi_v$. 
Let $\Phi_{X,Y}$ be the matrix coefficient of $\pi_v$ defined by $\Phi_{X,Y}(x)=\langle\pi_v(x)X, Y\rangle$, for any $x\in G_v$. If there are $X, Y$ such that $\Phi_{X,Y}\neq 0$, then the map $X\mapsto \Phi_{X,Y}$ embeds $\pi_v$  into $\mathcal{W}\subseteq C_c^\infty(G_v)$. 

The result \cite[Th. 5.8]{PS} shows that $\pi_v^\vee$ contains a special vector $Y_\alpha^\vee$ so that the linear map from $\pi_v$ to the space of Whittaker functions \[X\mapsto \left( x\mapsto \int_{N_v}\psi'_v(u) \Phi_{X,Y_\alpha^\vee}(u^{-1}x) \d u \right)  ,\]
is non-zero. Moreover, Paskunas and Stevens show that  $\pi_v$ contains a vector $Y_\alpha$ contained in $\rho$-isotypic part $(\pi_v)_{\rho}$ of $\pi_v$, so that $\Phi_{Y_\alpha, Y_\alpha^\vee}$ has support in $Z_v\mathcal{K}_v$ (\cite[Prop. 5.7]{PS}) and its associated  Whittaker function 
\[x\mapsto \int_{N_v}\psi'_v(u) \Phi_{Y_\alpha,Y_\alpha^\vee}(u^{-1}x)\d u,\]
is supported in $Z_vN_v\mathcal{K}_v$ and extends the function $\Phi_{Y_\alpha, Y_\alpha^\vee}$ supported on $Z_v\mathcal{K}_v$. Note that since the irreducible representation $\rho$ is contained with multiplicity one in $\pi_v$ (\cite[Theorem 2.4.4]{Yu}), every $X\in (\pi_v)_{\rho}$ is generated by $\langle \pi_v(k)Y_\alpha: k\in \mathcal{K}_v\rangle$. Therefore, every Whittaker function associated to $\Phi_{X,Y_\alpha^\vee}$ for $X\in (\pi_v)_{\rho}$ has support in $Z_vN_v\mathcal{K}_v$. 
\end{proof}

\begin{coro} \label{twist}
Suppose we are in the situation of Theorem \ref{PStevens}. Let $\mathcal{W}(\pi_v)$ be the Whittaker model for the character $\psi_v$ of conductor $\wp_{v}^{-n_v}$, then every function in $\mathcal{W}(\pi_v)_\rho$ is supported in 
\[  \{x\in G_v\vert v(\det(x))\in -n_v-1+2\mathbb{Z} \},  \] 
where $v$ is the valuation of $F_v$ normalized to be surjective to $\mathbb{Z}$. 
\end{coro} 
\begin{proof}
Let $t_v$ be a uniformizer of $\wp_v$. 
Let $\psi'_v:=(y\mapsto\psi_v(t_v^{-n_v-1} y))$. Then $\psi'_v$ has conductor $\wp_v$. 
Let $b=\begin{pmatrix}
t_v^{-n_v-1}&0\\0&1
\end{pmatrix}
$, 
we have an $G_v$-equivariant isomorphism: \begin{align*} \mathcal{W}(\pi_v, \psi_v)&\longrightarrow  \mathcal{W}(\pi_v, \psi_v') ,  \\
                                                              \varphi &\mapsto (x\mapsto \varphi(bx))  .
                                                                   \end{align*}
                                                      By Theorem \ref{PStevens}, we deduce that functions in $\mathcal{W}(\pi_v, \psi_v)_\rho$ are supported in $bZ_vN_v\mathcal{K}_v$. This implies the desired result. 
                                                      
                                                     \end{proof}

\subsection{Cuspidal terms }
We apply the previous preparation works to compute the cuspidal terms $J_{\pi}(f)$ in the spectral expansion. In fact, it is the case that $\pi\otimes\epsilon\cong\pi$ that is non-trivial.

Recall that we have defined in Introduction two commutative actions $\sigma$ and $\Fr^\ast$ on $P_{\mathfrak{R}}$ (see \eqref{PR}) so that \( b_\mathfrak{R}(k)\) is the cardinality of the set of fixed points of $\sigma\circ\Fr^{\ast k}$ and $ c_\mathfrak{R}(k)$ is that of $\Fr^{\ast k}$.  

\begin{lemm}\label{bck}
If $S_{cr}=\emptyset$, then $b_{\mathfrak{R}}(k)=c_{\mathfrak{R}}(k)=c_{\mathfrak{R}}(1)$ is either $0$ or $1$ for all $k\geq 1$. 

If either $S_c$ contains a point of even degree or $S_r$ contains a point of odd degree, we have $b_\mathfrak{R}(1)=0$.

If $S_c\neq \emptyset$, we have $c_{\mathfrak{R}}(1)=0$. 
\end{lemm}
\begin{proof}
The first statement is trivial because, in this case, $P_\mathfrak{R}$ is at most a singleton.

Suppose that $S_r$ contains a point of odd degree, meaning that $\Fr^\ast$ has an orbit of odd length on $\overline{S}_{cr}$. Suppose that $a\in 2\mathbb{Z}+1$ is the length of such an orbit and $x_0\in {S}_r\otimes \overline{\mathbb{F}}_q$ is a point in it. We have $\varepsilon_{x_0}(1)\neq \varepsilon_{x_0}(2)$. Suppose that $\Fr^{\ast a}((\varepsilon_{x}(i_x))_{x})=(\varepsilon_{x}')_x$. Here $i_x\in \{1,2\}$ for each $x\in \overline{S}$. 
By the choice of  $x_0$ and $a$, we have $\varepsilon'_{x_0}=  \varepsilon_{x_0}(i_{x_0})$. 
In particular, \[ \Fr^{\ast a}((\varepsilon_{x}(i_x))_{x})\neq \sigma((\varepsilon_{x}(i_x))_{x}),  \]
for any $(i_x)_x\in \{1,2\}^{\overline{S}}$. 
This implies that $b_\mathfrak{R}(a)=0$. Since $a$ is odd, $\sigma=\sigma^a$, we have $b_{\mathfrak{R}}(1)=0$.

Similarly, if $S_c$ contains a point of degree $a$, we have $c_{\mathfrak{R}}(a)=0$. It implies that $c_{\mathfrak{R}}(1)=0$. If $a$ is even, then $\sigma^{a}=\Id$. Hence $b_{\mathfrak{R}}(1)=0$. 
 \end{proof}

\begin{theorem}\label{twistedt}
Let $\pi$ be a cuspidal automorphic representation of $G(\AAA)$. 
Recall that $\epsilon$ is the sign character of $G(\AAA)$ that factors through $\deg\circ\det$. 
Suppose that $\pi\otimes \epsilon \cong \pi$. Let $f$ be the function introduced in Proposition \ref{f}.
Then \[\Tr(\epsilon \circ R(f)\vert  \pi ) =0 ,  \]
unless all the following conditions are met:
1. $b_{\mathfrak{R}}(1)\neq 0$; 2. every place in $S_u$ has odd degree; 3. $\pi_v\in \Irr^{\mathfrak{R}}(G_v)$ for $v\in |X|-S_u$ and $\pi_v$ has scalar ramification determined by semisimplification of $\mathfrak{R}_v$ for $v\in S_u$. 
If this is the case, we have 
\[\Tr(\epsilon \circ R(f)\vert  \pi ) =(-1)^{|S_u|}2^{|S_u|} (-1)^{\deg S_c }.  \]
\end{theorem} 

\begin{proof}
By Langlands correspondence, if $\pi\cong \pi\otimes \epsilon $, no local component of $\pi$ can be a twisted Steinberg representation.  In fact, since a Hecke character of $F^\times\backslash \mathbb{A}^\times$ is of finite order if and only if it sends an element of degree $1$ to a root of unity, if necessary, by replacing $\pi$ by an inertially equivalent, we may assume that the central character of $\pi$ is of finite order.   
Suppose that $\mathfrak{L}$ is the $\ell$-adic local system over ${X}-{S}$ that corresponds to $\pi$. If $\pi\otimes\epsilon\cong \pi$, then \[\mathfrak{L}|_{\overline{X}-\overline{S}}\cong \mathfrak{L}_1\oplus \mathfrak{L}_2, \]
and $\Fr^\ast$ permutes $\mathfrak{L}_1$ and $\mathfrak{L}_2$ (\cite[Prop. 2.1.3]{Yu1}). 
In particular, the ramification of $\mathfrak{L}$ at every $x\in \overline{S}$ is semisimple. Therefore, $\pi$ does not have a twisted Steinberg component by Theorem \ref{mono}. 

It is clear that if \[\Tr( \epsilon\circ R(f)\vert  \pi)\neq 0, \]
then $\pi_v$ has the desired ramification type for $v\in |X|-S_u$, and the $(\mathcal{I}_v,\iota\circ\chi )$-isotypic subspace $(\pi_v)_{(\mathcal{I}_v,\iota\circ\chi )}$ is non-trivial (notation as in Theorem \ref{mono}) for $v\in S_u$. In particular, for $v\in S_u$, $\pi_v$ is either a twisted Steinberg representation in $\Irr^{\mathfrak{R}}(G_v)$ or a twisted unramified principal series in $\Irr^{\mathfrak{R}^{ss}}(G_v)$ where $\mathfrak{R}^{ss}$ is the semisimplification of $\mathfrak{R}$. We have seen that $\pi_v$ can not be a twisted Steinberg representation. 
Moreover, the product of all eigenvalues of ramifications of $\mathfrak{L}_1$ and $\mathfrak{L}_2$ at all points $x\in \overline{S}$ should be $1$, and they are permuted by Frobenius action. This implies that $b_\mathfrak{R}(1)\neq 0$.

Next, we need to calculate \[ \Tr(\epsilon \circ R(f)\vert   \pi   )\] when $\pi \cong \pi\otimes \epsilon$ and $\pi_v$ has the above described property. The equation \eqref{999} allows us to calculate it locally. 

We note that Theorem \ref{mono} says if \[ \Tr(R(f_v)\vert \pi_v)\neq 0,\]
then $\pi_v\in \Irr^{\mathfrak{R}}(G_v)$. 
Let $v\in S_u$. If $\deg(v)$ is even, then $\epsilon_v$ equals the trivial character of $G_v$ and we have
\[ \Tr(\epsilon_v\circ R(f_v)|\mathcal{W}(\pi_v))=\Tr(R(f_v)\vert \pi_v). \]
It must be zero as $\pi_v$ is not a twisted Steinberg representation.

Suppose now that $v\in S_u$ and $\deg(v)$ is odd. Up to a twist by a character, we may assume that $\pi_v$ is unramified. We take a Whittaker function $\varphi_v\in \mathcal{W}(\pi_v)^{\mathcal{K}_v}$. Note that since both $\epsilon_v\varphi_v$ and $\varphi_v$ are contained in $\mathcal{W}(\pi_v)^{\mathcal{K}_v}=\mathcal{W}(\pi_v\otimes\epsilon_v)^{\mathcal{K}_v}$, they are differed by a scalar:
\[\epsilon_v\varphi_v=c\varphi_v. \]  
Let $x_v$ be any element in the support of $\varphi_v$, we deduce that $c=\epsilon_v(x_v)$. Let $\wp_v^{c(\psi_v)}$ be the conductor of $\psi_v$. 
We have shown in \cite[Lemme 5.5.3]{Yu1} that the support of $\varphi_v$ contains an element of valuation $ c(\psi_v) $.  Therefore $c=(-1)^{\deg (v) c(\psi_v)}$. 
Let $\varphi_v'\in \mathcal{W}(\pi_v)$ be the function defined by \[x\mapsto \varphi_v(x{\begin{pmatrix}
\varpi_v& 0\\ 0& 1
\end{pmatrix}}  ) .\] 
Then $\{\varphi_v, \varphi_v'\}$ is a basis of $ \mathcal{W}(\pi_v)^{\mathcal{I}_v}$. In fact, by 
Bruhat-Iwahori decomposition, we know that $\dim_{\mathbb{C}}\mathcal{W}(\pi_v)^{\mathcal{I}_v}\leq 2$ and it is clear that $\varphi_v, \varphi_v'$ are linearly independent. 
The endomorphisms on $\mathcal{W}(\pi_v)$: \[\epsilon_v \circ R(\frac{1}{\vol (\mathcal{I}_v)}\mathbbm{1}_{\mathcal{I}_v})  \text{ and }   \epsilon_v \circ R(\mathbbm{1}_{\mathcal{K}_v}  )    \]
are composition of a projection onto $\pi_{v}^{\mathcal{I}_v}$ together with a linear map represented respectively by the matrix
\[\begin{pmatrix}
(-1)^{\deg(v) c(\psi_v)} & 0\\ 0& (-1)^{\deg(v)(c(\psi_v)+1) }
\end{pmatrix},   
   \begin{pmatrix}
(-1)^{\deg(v) c(\psi_v)} & 0\\ 0& 0
\end{pmatrix}. 
      \]
We deduce that \[\Tr(\epsilon_v \circ R(f_v)\vert  \mathcal{W}(\pi_v)) = (-1)^{\deg(v)(c(\psi_v)+1)} - (-1)^{\deg(v) c(\psi_v)}.  \]
Since $\deg(v)$ is odd, it equals
\[\Tr(\epsilon_v \circ R(f_v)\vert  \mathcal{W}(\pi_v))=-2(-1)^{c(\psi_v)}.  \]

Similarly, for $v\in |X|-S_{cr}-S_u$, we have 
\[\Tr(\epsilon_v \circ R(f_v)\vert  \mathcal{W}(\pi_v))=(-1)^{c(\psi_v)\deg(v)}.  \]
For $v\in S_c$, we deduce from Corollary \ref{twist} that: 
\[\Tr(\epsilon_v \circ R(f_v)\vert  \mathcal{W}(\pi_v) )= (-1)^{\deg(v)(c(\psi_v)+1)}\Tr(R(f_v)\vert  \mathcal{W}(\pi_v))=(-1)^{\deg(v)(c(\psi_v)+1)}. \] 
For $v\in S_r$, as we have seen in Lemma \ref{bck} that $v$ must have even degree otherwise $b_{\mathfrak{R}}(1)=0$. We have 
\[\Tr(\epsilon_v \circ R(f_v)\vert  \mathcal{W}(\pi_v) )=\Tr(R(f_v)\vert  \mathcal{W}(\pi_v))=1. \] 
In conclusion, by \eqref{999} we have
\[\Tr(\epsilon \circ R(f)\vert  \pi )= (-1)^{|S_u|}2^{|S_u|} (-1)^{\sum_{v\in |X|-S_c}\deg(v) c(\psi_v)+\sum_{v\in S_c} \deg(v)(c(\psi_v)+1)  } .  \]
We have \[\sum_{v} c(\psi_v)\deg(v) =   - (2g-2). \]   
Therefore, 
\[\Tr(\epsilon \circ R(f)\vert  \pi ) = (-1)^{|S_u|} 2^{|S_u|} (-1)^{\deg S_c }. \]

\end{proof}

\begin{coro} \label{cuspt}
If $S_u\neq\emptyset$ and at least one place in it is of even degree, then 
\[J_{cusp}^1(f)=\vert  E_2(\mathfrak{R})^{\Fr^*}\vert. \]
If $S_u\neq\emptyset $ and every place in it has an odd degree, then 
\[J_{cusp}^1(f)=\vert  E_2(\mathfrak{R})^{\Fr^*}\vert + \begin{cases}(-1)^{|{S}_u|+1}{b_{\mathfrak{R}}(1)} 2^{|S_u|-2}  (\Pic(2)-\Pic(1)) ,  \quad S_{cr}= \emptyset; \\ 
(-1)^{|{S}_u|+1} {b_{\mathfrak{R}}(1)}{2}^{|S_u|-2}(-1)^{\deg S_c}\Pic(2)
, \quad S_{cr}\neq  \emptyset.
\end{cases} \]
If $S_u=\emptyset$, then
\[J_{cusp}^1(f)=\vert  E_2(\mathfrak{R})^{\Fr^*}\vert +  \begin{cases}0 ,  \quad S_{cr}= \emptyset; \\ 
\frac{b_{\mathfrak{R}}(1)}{4}(1-(-1)^{\deg S_c})\Pic(2)
, \quad S_{cr}\neq  \emptyset.
\end{cases} \]

\end{coro}
\begin{proof}
In fact, the sum of $J_\pi(f)$ over inertial equivalent classes of cuspidal automorphic representations $\pi$ such that $\pi\otimes\epsilon\not\cong \pi$ gives $|E_2(\mathfrak{R})^{\Fr^\ast}|$  after Theorem \ref{Langlands}.

We need to consider  the sum \[\frac{1}{2}\sum_{\pi} \Tr(R(f)|\pi)  - \frac{1}{2}\sum_{\pi}\Tr(\epsilon\circ R(f)| \pi) , \]
where the sums over $\pi$ are taken over inertial equivalent classes of cuspidal automorphic representations $\pi$ such that $\pi\otimes\epsilon\cong \pi$. The Langlands correspondence (see the proof of Theorem \ref{twistedt}) shows that no such cuspidal automorphic representation $\pi$ can have a twisted Steinberg component. Therefore, if $S_u\neq \emptyset$, by Proposition \ref{f}, 
\[\frac{1}{2}\sum_{\pi} \Tr(R(f)|\pi)=0. \]
If $S_u=\emptyset$, we need to know the number of equivalence classes of such $\pi$. By \cite[2.1.3]{Yu1} and the first paragraph of the proof of Theorem \ref{twistedt}, the set of such $\pi$ is in bijection with the set of  non-ordered pairs $(\mathfrak{L}_1, \mathfrak{L}_2)$  of rank $1$ $\ell$-adic systems over $\overline{X}-\overline{S}$  such that $\Fr^{\ast}{\mathfrak{L}_i}\cong \mathfrak{L}_{3-i}$, $\mathfrak{L}_1\not\cong \mathfrak{L}_2$,  and that the local monodromies of the direct sum $\mathfrak{L}_1\oplus\mathfrak{L}_2$ is given by $\mathfrak{R}$. 
If $S_{cr}=\emptyset$, then $b_{\mathfrak{R}}(1)$ equals $0$ or $1$ and 
there are ${b_{\mathfrak{R}(1)}}\frac{\Pic(2)-\Pic(1)}{2}$ such pairs. If $S_{cr}\neq\emptyset$, then there are $\frac{b_{\mathfrak{R}(1)}}{2}\Pic(2)$ such pairs. 
We have
\[\frac{1}{2}\sum_{\pi} \Tr(R(f)|\pi)=\begin{cases}\frac{b_{\mathfrak{R}}(1)}{4}(\Pic(2)-\Pic(1))  , \quad S_{cr}= \emptyset; \\ 
\frac{b_{\mathfrak{R}}(1)}{4}\Pic(2)
, \quad S_{cr}\neq  \emptyset.
\end{cases}\]

By Theorem \ref{twistedt}, if $b_\mathfrak{R}(1)=0$ or if $S_u\neq \emptyset$ and contains a place of even degree, then the sum
\[\frac{1}{2}\sum_{\pi}\Tr(\epsilon\circ R(f)| \pi) \]
 is $0$. 
Otherwise if $S_u\neq \emptyset$ and does not contain any place of even degree or $S_u=\emptyset$, then by  Theorem \ref{twistedt}, we have
\[\frac{1}{2}\sum_{\pi} \Tr(\epsilon\circ R(f)|\pi)= (-1)^{|S_u|}\begin{cases}{b_{\mathfrak{R}}(1)}2^{|S_u|-2}(\Pic(2)-\Pic(1)) ,  \quad S_{cr}= \emptyset; \\ 
{b_{\mathfrak{R}}(1)}{2}^{|S_u|-2}(-1)^{\deg S_c}\Pic(2)
, \quad S_{cr}\neq  \emptyset.
\end{cases}\]

\end{proof}

\subsection{Residuel terms}
The proposition below describes the contributions to the trace formula from the residual spectrum. 

\begin{prop}\label{rest}
 If $S_{cr}$ is non-empty, then 
\[J_{res}^{1}(f)=0. \]
If $S_{cr}=\emptyset$, then
\[J_{res}^{1}(f)=c_{\mathfrak{R}}(1)(-1)^{\vert S_u\vert}\Pic(1). \]
\end{prop}
\begin{proof}
For any $v\in S_{cr}$ and any character $\mu_v$ of $G_v$, we have \[\Tr(R(f_v)\vert \mu_v)=0. \]
Therefore, the first statement holds, and it suffices to consider the case that $S_{cr}=\emptyset$.

It is clear that if $v\notin S$, then for a character $\mu_v$ of $G_v$ we have
 \[\Tr(R(\mathbbm{1}_{\mathcal{K}_v})\vert \mu_v)=\begin{cases} 1,\quad \mu_v\vert_{\mathcal{K}_v}=1;\\
 0, \quad \mu_v\vert_{\mathcal{K}_v}\neq 1.
 \end{cases}
\]
Let $v\in S_s$, we have  \[\Tr(R(f_v)\vert \mu_v)=\begin{cases} 1,\quad \mu_v\vert_{\mathcal{K}_v}=\theta_v;\\
 0, \quad \mu_v\vert_{\mathcal{K}_v}\neq \theta_v;
 \end{cases}
\]
where $\theta_v$ is the character defined in Theorem \ref{mono}. 
Similarly, let $v\in S_u$, we have 
\[\Tr(R(f_v)\vert \mu_v)=\begin{cases}-1,\quad \mu_v\vert_{\mathcal{K}_v}=\theta_v;\\
 0, \quad \mu_v\vert_{\mathcal{K}_v}\neq \theta_v .
 \end{cases}
\]

Following our discussions in rank $1$ in \ref{rank1}, we deduce that if \begin{equation}\label{equal1}
\prod_{v}\theta_v|_{\mathbb{F}_q^\times}\neq 1,  \end{equation}
then \[J^{1}_{res}(f)=0. \]  
Otherwise, there are $\Pic(1)$ such equivalent classes of $\mu$. Therefore,  we have 
\[J^{1}_{res}(f)=(-1)^{\vert S_u\vert}\Pic(1).  \]
The result is thus proved since \eqref{equal1} is equivalent to $c_{\mathfrak{R}}(1)\neq 0$ (cf. \cite[3.2]{Deligne} or \cite[Proposition 2.4.2]{Yu}).
\end{proof}

\subsection{Continuous terms}

The proposition below describes the contributions to the trace formula from the continuous spectrum.

\begin{prop}\label{continoust}
If $S_c\neq\emptyset$ then \[ J^{1}_{cont}(f)=0. \]

If $S_c=\emptyset$, then we have the following results. 
\begin{enumerate}
\item $S_{r}\neq\emptyset$.

\[J_{cont}^1(f)=\begin{cases} \frac{1}{2} c_{\mathfrak{R}}(1)\Pic(1)^2(2g-2+\deg S_{r}), \quad S_u=\emptyset;
 \\  \frac{1}{2}c_{\mathfrak{R}}(1)\Pic(1)^2   \deg v , \quad S_u=\{v\};
 \\   0, \quad |S_u|\geq 2.
\end{cases}
\]

\item $S_{r}=\emptyset$.

\[J_{cont}^1(f)=\begin{cases} c_{\mathfrak{R}}(1)\Pic(1)^2(g-1), \quad S_u=\emptyset;
 \\ c_{\mathfrak{R}}(1)\Pic(1)^2  \frac{\deg v}{2}+ \frac{1}{2}c_{\mathfrak{R}}(1)\Pic(1)  , \quad S_u=\{v\} \text{ and  } 2\nmid \deg v; 
  \\c_{\mathfrak{R}}(1)\Pic(1)^2  \frac{\deg v}{2}  , \quad S_u=\{v\} \text{ and  } 2\mid \deg v; 
 \\ c_{\mathfrak{R}}(1)\Pic(1) (-1)^{|S_u|+1}2^{|S_u|-2}, \quad |S_u|\geq 2, \text{ and } S_{u, even}= \emptyset;
 \\   0, \quad |S_u|\geq 2, \text{ and } S_{u, even}\neq \emptyset.
\end{cases}
\]

\end{enumerate}

\end{prop}
The proof is given in \ref{Continoust}. We need first  to do some calculations about $L$-functions and intertwining operators. 

The first thing to remark is that if $S_c\neq\emptyset$, then for any $\psi\in \mathcal{A}_{cont}$, the action of $R(f)$ on $\mathcal{A}_{B,\psi}$ must be the $0$ map. In fact, at a place $v\in S_c$, we have $\mathcal{A}_{B,\psi}\cong I_B(\psi_v)$, the induced representation of $\psi_v$. By Theorem \ref{mono}, it does not contain $\rho_v$, as $\rho_v$ is cuspidal. Therefore we assume that $S_c=\emptyset$ in the following.

\subsubsection{$L$-functions of Hecke characters}
We'll need information on $L$-functions of Hecke characters when dealing with the continuous part of the trace formula. Over a function field, we have a complete understanding of them. 

Let $\chi: \mathbb{A}^{\times}/F^{\times}\longrightarrow \mathbb{C}^{\times}$ be a Hecke character of finite order. Let $\chi=\otimes_{v\in \vert  F\vert  }\chi_v$ be its factorisation.
 Recall that the L-function $L(\chi, z)$ can be defined by the formal power series 
\[L(\chi, z):= \prod_{v\in \vert  X\vert  } L_v(\chi, z),  \]
where \[L_v(\chi, z) =  \begin{cases}
\frac{1}{1-\chi_v(\varpi_v)z^{\deg v}}, \quad \text{ if   $\chi_v$ is unramified }   ; \\
1\quad otherwise.
\end{cases}
  \]
The infinite product is absolutely convergent and is holomorphic if $\vert  z\vert  <q^{-1}$. It admits a  meromorphic continuation to the whole $\mathbb{C}$. Moreover, it is a rational function in $z$. Let $R$ be the set of places of ramifications of $\chi$, identified with a subset of closed points of $X$. We have fixed an isomorphism $\iota:\overline{\mathbb{Q}}_{\ell}\xrightarrow{\sim} \mathbb{C}$. 
Let $\mathfrak{L}_\chi$ be an $\ell$-adic local system over ${X} - {R}$ corresponding to $\chi$ obtained by the global class field theory. 
The $L$-function of $\mathfrak{L}_\chi$ equals that of $\chi$: 
\[  \iota( L(\mathfrak{L}_\chi, z)) =L(\chi, z)   .\]
Moreover, we know from Grothendieck's cohomological interpretation (cf., for example, \cite[Th\'eor\`eme VI.1]{Lafforgue}) that 
\[L(\mathfrak{L}_\chi, z) =  \frac{\det(1- z\mathrm{F}_q \vert  H^{1}_{c}(\overline{X}-\overline{R}, \mathfrak{L}_\chi ))}{  \det(1- z\mathrm{F}_q\vert  H^{2}_{c}(\overline{X}-\overline{R}, \mathfrak{L}_\chi )) \det(1- z\mathrm{F}_q \vert  H^{1}_{c}(\overline{X}-\overline{R}, \mathfrak{L}_\chi )) } .  \]
where $\mathrm{F}_q$ is a geometric Frobenius element acting on cohomologies with compact support.

\begin{prop}(Riemann hypothesis)\label{Riemann}
Let $\chi$ be a Hecke character on $F^{\times}\backslash \mathbb{A}^{\times}$ of finite order so that the set of ramified places is $R$. If $\chi$ is inertially equivalent to the trivial character (in particular $R=\emptyset$), then \[ L(\chi, z)=\frac{P(z)}{(1-z)(1-qz)}\] where $P(z)$ is a polynomial of $2g$. 

If $\chi$ is not inertially equivalent to the trivial character, then its $L$-function \(  L(\chi, z) \) is a polynomial in $z$ of degree $2g-2+\deg R$. 

In any case, all of the zeros of \(  L(\chi, z) \) satisfy $\vert  z\vert  =q^{-\frac{1}{2}}. $
\end{prop}
\begin{proof}

If $R$ is empty, then there are two cases. If $\mathfrak{L}_{\chi}|_{\overline{X}}$ is trivial, up to a twist of a rank $1$ sheaf over $\Spec(\mathbb{F}_q)$, we may assume that $\mathfrak{L}_\chi\cong \overline{\mathbb{Q}}_\ell$ is the constant sheaf. We have \[L(\mathfrak{L}_\chi, z)=\zeta_{X}(z)=\frac{P(z)}{(1-z)(1-qz)},  \]
$\deg P(z)=2g$ and all of the zeros of \(  P(z) \) satisfy $\vert  z\vert  =q^{-\frac{1}{2}}$.  
If $\mathfrak{L}_{\chi}|_{\overline{X}}$ is non-trivial then $L(\mathfrak{L}_\chi, z)$ is a polynomial of degree $2g-2$. For reference, see \cite[Prop. 6.1.1]{Yu1}.

We consider the case that $R$ is non-empty. We know that 
\begin{equation}
H^{0}_c(\overline{X}-\overline{R}, \mathfrak{L}_\chi)=0,
\end{equation}
as $ \mathfrak{L}_\chi$ is non-trivial. By Poincar\'e duality \begin{equation}H^{2}_c(\overline{X}-\overline{R},  \mathfrak{L}_\chi)^{\vee} \cong H^{0}(\overline{X}-\overline{R},  \mathfrak{L}_\chi^{\vee})(1)\cong Hom_{ \overline{X}-\overline{R}}(\mathfrak{L}_\chi^{\vee}, \overline{\mathbb{Q}}_{\ell} )(1). \end{equation}
Therefore we deduce that $H^{2}_c(\overline{X}-\overline{R},  \mathfrak{L}_\chi)$ is $0$ as well. 
The dimension of $H^{1}_c(\overline{X}-\overline{R},  \mathfrak{L}_\chi)$ can then be derived from the Euler-Poincar\'e characteristic:
$$\dim H^{1}_c(\overline{X}-\overline{R},  \mathfrak{L}_\chi) = -\chi_c(\overline{X}-\overline{R},  \mathfrak{L}_\chi),$$
which can be calculated by the Grothendieck-Ogg-Shafarevich formula, see \cite[Th\'eor\`eme 1, p.133]{Raynaud}. In fact, since every local monodromy of $\mathfrak{L}_\chi$ is tamely ramified,  the Swan conductor is zero.  We deduce from \textit{loc. cit.} that $\chi_c(\overline{X}-\overline{R},  \mathfrak{L}_\chi)= \chi_c(\overline{X}-\overline{R})=2-2g-\deg R$.

The assertion about the positions of zeros is the Riemann hypothesis for rank $1$ local systems (see \cite[Th\'eor\`eme VI.10]{Lafforgue} for a general statement). 
\end{proof}

\subsubsection{Eulerian expansions of intertwining operators} \label{eulerexpansion}
Let $\psi\in \mathcal{A}_{cont}$. 
Let \[{\M}(w,\lambda): \mathcal{A}_{B,\psi} \longrightarrow  \mathcal{A}_{B,w(\psi)}\] be an intertwining operator. It is a $G(\AAA)$-morphism. 

 Let $I_B(\psi_v)$ be the space of functions $\varphi$ over $G_v$ satisfying $\varphi(ntx)=\rho_B(t)\psi_v(t)\varphi(x)$ for any $n\in N_v$, $t\in T_v$ and $x\in G_v$. 
By definition of intertwining operator, we have an Eulerian expansion. Indeed, let $\M_{v}(w, \lambda): I_B(\psi_v)\longrightarrow I_{B}(w(\psi_v))$ be an operator defined by analytic continuation of the integral which converges when $|\Re \lambda| >>0$, 
\begin{equation}  (\M_{v}(w, \lambda) \varphi )(x) =\lambda(x)    \int_{ {N}_v}  \varphi(w^{-1}nx)\lambda(w^{-1}nx) \d n.   \end{equation}
Choosing isomorphisms $c_\psi: \mathcal{A}_{B,\psi}\longrightarrow \otimes_{v} I_B(\psi_v)$ and $c_{w(\psi)}: \mathcal{A}_{B,w(\psi)}\longrightarrow \otimes_{v} I_B(w(\psi_v))$, there is a constant $c$ depending only on these isomorphisms such that the following diagram is commutative: 
\[\begin{CD}
\mathcal{A}_{B,\psi}@>{\M}(w,\lambda)>> \mathcal{A}_{B,w(\psi)} \\
@Vc_\psi VV @Vc_{w(\psi)}VV \\
\otimes_{v} I_B(\psi_v)@>c\otimes_v  {\M}_v(w,\lambda)>> \otimes_{v} I_B(w(\psi_v)) 
 \end{CD}.  \]

In the special case that $\psi$ is non-regular, i.e., $w(\psi)=\psi$, we have \[c=q^{1-g}.\] 
This number is derived from the difference in the normalizations of the Haar measures on $N_v$ and $N(\AAA)$.

\subsubsection{Intertwining operator on $(\mathcal{K}_v , \theta_v)$-isotypical subspace} \label{Kcase}
In this part, we treat the local intertwining operator when $v\in S_s$ or $v\notin S$. 

Let $\theta_v$ be a character: \[\theta_v: \mathcal{K}_v\xrightarrow{\det}\mathcal{O}_v^\times\longrightarrow \kappa_v^\times\longrightarrow \mathbb{C}^\times. \] 

We have \[\dim I_B(\psi_v)_{(\mathcal{K}_v , \theta_v)}\leq 1. \]
The space $I_B(\psi_v)_{(\mathcal{K}_v , \theta_v)}$  is non-zero if and only if $\psi_v=(\psi_{v,1}, \psi_{v,2})$ with $\psi_{v,1}\vert  _{\mathcal{O}_v^\times}=\psi_{v,2}\vert  _{\mathcal{O}_v^{\times}}=\theta_v$. In this case, there is a $\mu\in \mathbb{C}^{\times}$ such that for any $y\in F_v^{\times}$, we have $\psi_{v1}(y)/\psi_{v2}(y)=\mu^{\deg y}$. Moreover, $I_B(\psi_v)_{(\mathcal{K}_v , \theta_v)}$ is generated by the function $\varphi_{\psi_v}$ defined so that for any $b\in B_v$ and $k\in \mathcal{K}_v $:
\[\varphi_{\psi_v}(bk)= \rho_B(b)\psi_v(b)\theta_v(\det k). \]
For $x\in G_v$ and $\lambda\in X_T^G\cong\mathbb{C}^\times$ (see Section \ref{4.1} for the definition of $X_T^G$), let \[\varphi_{\psi_v, \lambda}(x):=\varphi_{\psi_v}(x)\lambda(x).\]

By dimension $1$, we have \[ \M_v(w,\lambda)\varphi_{\psi_v}=c_\lambda \varphi_{w(\psi_v)}, \]
for some constant $c_\lambda\in\mathbb{C}^\times$. It suffices to evaluate the above equation at $x=1$ to find the value $c_\lambda$: \[ c_\lambda=\int_{{N}_v} \varphi_{\psi_v, \lambda}(w^{-1}n)\d n. \]
The integral can be decomposed:
\[\int_{N(\mathcal{O}_v)} \varphi_{\psi_v, \lambda}(w^{-1}n) \d n +  \int_{F_v- \mathcal{O}_v } \varphi_{\psi_v, \lambda}(\begin{pmatrix}1&y^{-1}\\0&1
\end{pmatrix}
\begin{pmatrix}y^{-1}&0\\0&y
\end{pmatrix}
\begin{pmatrix}-1&0\\y^{-1}&-1
\end{pmatrix}
)\d y.  \]
Note that \[ \int_{N(\mathcal{O}_v)} \varphi_{\psi_v, \lambda}(w^{-1}n) \d n=\vol(N(\mathcal{O}_v))=1.\] Besides, for $y\in F_v-\mathcal{O}_v$, we have \[\varphi_{\psi_v, \lambda}(\begin{pmatrix}1&y^{-1}\\0&1
\end{pmatrix}
\begin{pmatrix}y^{-1}&0\\0&y
\end{pmatrix}
\begin{pmatrix}-1&0\\y^{-1}&-1
\end{pmatrix}
) =\mu^{-\deg y} \lambda^{-2\deg v\deg y}.  \]
Under our additive Haar measure on $\mathcal{O}_v$, we know that $\vol(\varpi_v^{n}\mathcal{O}_v^{\times})=q_v^{-n}(1-q_v^{-1})$. We deduce that 
\[c_\lambda=\int_{{N}(F_v)} \varphi_{\psi_v, \lambda}(w^{-1}n)\d n 
=\frac{1-q_v^{-1} \mu\lambda^{2\deg v}    }{1-\mu\lambda^{2\deg v}}=\frac{L_v( \psi_{1}\psi_{2}^{-1}, \lambda^2 )}{   L_v( \psi_{1}\psi_{2}^{-1}, q^{-1}\lambda^2 )      }  . \]

\subsubsection{Intertwining operator on \((\mathcal{I}_v, \chi_v)\)-typical subspace: regular cases}\label{regcase}
In this part, we treat the case that $v\in S_{r}$.

If $\chi_v$ is regular, then we have 
 \[\dim I_B(\psi_v)_{(\mathcal{I}_v, \chi_{v} )} \leq 1.  \]
In fact, we have $G_v=B_v\mathcal{I}_v\coprod B_vw\mathcal{I}_v$. Any function $\varphi\in I_B(\psi_v)_{(\mathcal{I}_v, \chi_{v} )} $ is entirely determined by $\varphi(1)$ and $\varphi(w)$. 
Moreover, for such a $\varphi$, using the definition of $I_B(\psi_v)$ and \((\mathcal{I}_v, \chi_v)\)-typical condition, we have \[ \varphi(tw)=\psi_v(t)\varphi(w)=\varphi(w)\chi_v(w^{-1}tw)\] and \[ \varphi(t)=\chi_v(t)\varphi(1)=\varphi(1)\psi_v(t).\] 
Therefore if $\psi_v\vert  _{T(\mathcal{O}_v)}\neq \chi_v$ and   $\psi_v\vert  _{T(\mathcal{O}_v)}\neq w(\psi_v)$ then $\varphi$ must be zero. If $\psi\vert  _{T(\mathcal{O}_v)}=\chi_v$, then $\varphi$ is supported in $B_v\mathcal{I}_v$ and if 
$\psi_v\vert  _{T(\mathcal{O}_v)}=w(\chi_v)$ then $\varphi$ is supported in  $B_vw \mathcal{I}_v$. 

Suppose that we have $\psi_v\vert  _{T(\mathcal{O}_v)}=\chi_v$.  The space 
 $I_B(\psi_v)_{(\mathcal{I}_v, \chi_{v} )} $ is generated by the function 
 $\varphi_{\psi_v}$ defined in such a way that for any $n\in N_v$, $t\in T_v$ and $k\in \mathcal{K}_v $:
\[\varphi_{\psi_v}(ntk)=
\begin{cases} \rho_B(t) \psi_v(t)\psi_v(k) , \quad \text{ if  } k\in \mathcal{I}_v;\\
0, \quad  \text{ if  } k \in \mathcal{I}_v w\mathcal{I}_v. 
\end{cases} \]
The space $I_B(w(\psi_v))_{(\mathcal{I}_v, \chi_{v} )} $ is generated by $\varphi_{w(\psi_v)}$ defined in such a way that for any $n\in {N}_v$, $t\in T_v$ and $k\in \mathcal{K}_v $:
\[\varphi_{w(\psi_v)}(ntk)=
\begin{cases}  \rho_B(t)\psi_v(t)\psi(b) , \quad \text{ if  } k=nwb \in  \mathcal{I}_{v+} w\mathcal{I}_v;\\
0, \quad  \text{ if  } k \in \mathcal{I}_v.
\end{cases} \]

The local intertwining operator $\M_v(w,\lambda)$ is a linear map from $I_B(\psi_v)_{(\mathcal{I}_v, \chi_{v} )} $ to $I_B(w(\psi_v))_{(\mathcal{I}_v, \chi_{v} )} $. 
By dimension 1, there is a constant $c_\lambda\in \mathbb{C}$ such that 
\[ \M_{v}(w, \lambda)\varphi_{\psi_v} = c_\lambda\varphi_{w(\psi_v)}. \]
Evaluating at the point $x=w$, we see that  
\[c_\lambda= \int_{{N}_v} \varphi_{\psi_v}(w^{-1}nw) \lambda(w^{-1}nw) \d n. \]
We break the integral into two parts following the union:
 \[w^{-1} {N}_v w =  \begin{pmatrix} 1&0\\
                                                                    \wp_v&1\end{pmatrix} \cup  
                                                                    \begin{pmatrix} 1&0\\
                                                                                  F_v-\wp_v&1  \end{pmatrix}.   \]
The first part is included in $\mathcal{I}_v$. 
For the second part, we need 
\[\begin{pmatrix} 1&0\\
                            x&1\end{pmatrix} 
=\begin{pmatrix} x^{-1}&0\\
                            0&x\end{pmatrix} 
\begin{pmatrix} 1&x\\
                          0&1
\end{pmatrix} 
\begin{pmatrix} 0&-1\\
                          1&x^{-1}
\end{pmatrix}. 
    \]
   Since $\begin{pmatrix} 0&-1\\
                          1&x^{-1}
\end{pmatrix}\in \mathcal{I}_vw\mathcal{I}_v$, the integral over $\begin{pmatrix} 1&0\\
                                                                                  F_v-\wp_v&1  \end{pmatrix}$ vanishes and  $c_\lambda=\vol(\wp_v)=q_v^{-1}$. 
    As the local $L$-factor ${L_v(\psi_{1}\psi^{-1}_2, \lambda^2 )}$ is trivial, we can present the result as  \[c_\lambda=q_v^{-1}\frac{L_v(\psi_{1}\psi^{-1}_2, \lambda^2 )}{L_v(\psi_{1}\psi^{-1}_2, q^{-1}\lambda^2 )}. \]

\subsubsection{Intertwining operator on \( (\mathcal{I}_v, \chi_v)\)-typical subspace: non-regular cases}
\label{nrcase}
In this subsection, we calculate intertwining operator on \( (\mathcal{I}_v, \chi_v)\)-typical subspace of $I_B(\psi_v)$ when $\psi_v=(\psi_{v1}, \psi_{v2})$ is non-regular. That is $\psi_{v1}$ and $\psi_{v2}$ differ by an unramified character. 
Such calculations have already been done in \cite[Lemma 4.7]{Flicker}. In fact, the calculations are similar to the cases we have already treated. Therefore, we'll briefly recall the results. 

If $\chi_v$ is a character of $\mathcal{I}_v$ that factors through determinant, we have
\[\dim I_B(\psi_v)_{( \mathcal{I}_v, \chi_v)}=2 \text{ or  } 0. \]
The dimension is non-zero if and only if $\psi\vert  _{T(\mathcal{O}_v)}$ lifts $\chi_v$. In fact, the double quotient $B_v\backslash G_v/\mathcal{I}_v\cong \mathcal{I}_v\backslash \mathcal{K}_v /\mathcal{I}_v$ has cardinality $2$. We know that \[\dim I_B(\psi_v)_{( \mathcal{I}_v, \chi_v)}\leq 2.    \]
If $\psi_v\vert  _{T(\mathcal{O}_v)}$ lifts $\chi_v$, the space $I_B(\psi_v)_{( \mathcal{I}_v, \chi_v)}$ is generated by the basis $(\varphi_{\psi_v, 1}, \varphi_{\psi_v, w})$, where for any $n\in N_v$, $t\in T_v$ and $k\in \mathcal{K}_v $,  we have: 
\[\varphi_{\psi, 1}(ntk)=
\begin{cases} \rho_B(t)\varphi(t)  \psi_v(\det(k)) , \quad \text{ if  } k\in \mathcal{I}_v;\\
0, \quad  \text{ if  } k \in \mathcal{I}_v w\mathcal{I}_v.
\end{cases} \]
\[\varphi_{\psi, w}(ntk)=
\begin{cases}  \rho_B(t) \psi(t) \psi_v(\det(k)) , \quad \text{ if  } k\in \mathcal{I}_v w\mathcal{I}_v;\\
0, \quad  \text{ if  } k \in \mathcal{I}_v. 
\end{cases} \]
If $\psi_v\vert  _{T(\mathcal{O}_v)}$ does not extend $\chi_v$, then $ I_B(\psi_v)_{( \mathcal{I}_v, \chi_v)} $ is zero since for any $\varphi\in  I_B(\psi_v)_{( \mathcal{I}_v, \chi_v)} $, we have \[\varphi(t)=\psi_v(t)\varphi(1)=\varphi(1.t)=\chi_v(t)\] for any $t\in T(\mathcal{O}_v)$.

The local intertwining operator $\M_v(w, \lambda)$ is a linear map from $I_B(\psi_v)_{( \mathcal{I}_v, \chi_v)}$ to $I_B(w(\psi_v))_{( \mathcal{I}_v, \chi_v)}$. 
In the basis $(\varphi_{\psi_v,1} , \varphi_{\psi_v, w})$ and $(\varphi_{w(\psi_v),1} , \varphi_{w(\psi_v), w})$, we have (see \cite[Lemma 4.7]{Flicker}), 
 \[\M_v(w, \lambda) (\varphi_{\psi_v,1} , \varphi_{\psi_v, w}) =   (\varphi_{w(\psi_v),1} , \varphi_{w(\psi_v), w})\begin{pmatrix} (1-q_v^{-1})\frac{\mu\lambda^{2}}{1-\mu\lambda^2} & 1 \\ q_v^{-1} & (1-q_v^{-1})\frac{1}{1-\mu\lambda^2}
\end{pmatrix}.  \]
Here $\mu\in \mathbb{C}^{\times}$ is the element such that for any $y\in F_v^{\times}$, we have $\psi_{v1}(y)/\psi_{v2}(y)=\mu^{\deg y}$.  
We need to present the result in the form 
 \[\M_v(w, \lambda) (\varphi_{\psi_v,1} , \varphi_{\psi_v, w}) =   (\varphi_{w(\psi_v),1} , \varphi_{w(\psi_v), w})
\frac{L_v(\psi_{1}\psi^{-1}_2, \lambda^2 )}{L_v(\psi_{1}\psi^{-1}_2, q^{-1}\lambda^2 )}
\begin{pmatrix} \frac{(1-q_v^{-1})\mu\lambda^{2\deg v}}{1-q_v^{-1}\mu\lambda^{2\deg v} } & \frac{1-\mu\lambda^{2\deg v}}{1-q_v^{-1}\mu\lambda^{2\deg v}}  \\   \frac{q_v^{-1}(1-\mu\lambda^{2\deg v} )}{1-q_v^{-1}\mu\lambda^{2\deg v}} & \frac{1-q_v^{-1}}{1-q_v^{-1}\mu\lambda^{2\deg v}}
\end{pmatrix}.  \]
In particular, 
\[ \det(\M_v(w, \lambda) )= -q_v^{-1}\frac{1-q_v\mu\lambda^{2\deg v}}{1-q_v^{-1}\mu\lambda^{2\deg v}}(\frac{L_v(\psi_{1}\psi^{-1}_2, \lambda^2 )}{L_v(\psi_{1}\psi^{-1}_2, q^{-1}\lambda^2 )})^2.  \]
and 
\[\Tr(\M_v(w, \lambda) )=\frac{L_v(\psi_{1}\psi^{-1}_2, \lambda^2 )}{L_v(\psi_{1}\psi^{-1}_2, q^{-1}\lambda^2 )}\frac{(1-q_v^{-1})(1+\mu\lambda^{2\deg v})}{1-q_v^{-1}\mu\lambda^{2\deg v}}. 
 \]

\subsubsection{Continuous terms}\label{Continoust}
Now we come to the calculations of contributions from continuous spectrum using the previous preparations.

We are going to consider $J^1_{\psi}(f)$ for $\psi\in \mathcal{A}_{cont}$. We first we observe that if $\psi=(\psi_1, \psi_2)$, then by definition $\psi_1$ and $\psi_2$ are unitary. The same applies to their local components.

Recall that  for each place $v\in S_u$, we have set in Proposition \ref{f} that  \[f_v=\left(\frac{1}{\vol(\mathcal{I}_v)}\mathbbm{1}_{\mathcal{I}_v}(x)-2 \mathbbm{1}_{\mathcal{K}_v}(x)\right)\theta_v(\det \overline{x}^{-1}). \]
We have 
\[J_\psi^1( f )=\sum_{S_0\subseteq S_u} (-1)^{\vert  S_0\vert  }2^{\vert  S_0\vert  } J_\psi^1(f_{S_0})  .  \]
where \[ f_{S_0}=\otimes_v f_{S_0, v}\in C_c^\infty(G(\AAA))\] is the function that $f_{S_0, v}=f_v$ for all places $v$ outside $S_u$ and is equal to $x\mapsto\mathbbm{1}_{\mathcal{K}_v }(x)\theta_v(\det \overline{x}^{-1})$ if $v\in S_0$ and is equal to $ \frac{1}{\vol(\mathcal{I}_v)}\mathbbm{1}_{\mathcal{I}_v}(x)\theta_v(\det \overline{x}^{-1})$ for $v\in S_1=S_u-S_0$.

We need a lemma for our calculations. 
\begin{lemm}\label{determinant}
Let $V$ be a finite-dimensional $\mathbb{C}$-linear space. 
Let $m$ be a meromorphic function over $\mathbb{C}$ with values in $\mathrm{GL}(V)$. Suppose that $m$ is holomorphic at any point in the unit circle.  We use $Z_{\vert  \lambda\vert  <1}(h)$ for the integer defined to be the number of zeros (with multiplicity) minus the number of poles (with multiplicity) in the region $\vert  \lambda\vert  <1$ of a meromorphic function $h$ over $\mathbb{C}$. 
Then \[\int_{\Im X_T^G} \lim_{\mu\rightarrow 1} \mathrm{Tr}_V(\frac{1}{\mu^{-1}-\mu}{m(\lambda)}^{-1}\circ {m(\lambda/\mu)} - \frac{1}{\mu^{-1}-\mu}\Id )  \d \lambda = \frac{1}{2} Z_{\vert  \lambda\vert  <1}(\det(m(\lambda))). \]
\end{lemm}
\begin{proof}
We have \[\lim_{\mu\rightarrow 1} \mathrm{Tr}_V(\frac{1}{\mu^{-1}-\mu}{m(\lambda)}^{-1}\circ {m(\lambda/\mu)} - \frac{1}{\mu^{-1}-\mu}\Id ) =\frac{\lambda}{2}  \mathrm{Tr}_V( m(\lambda)^{-1} \circ m'(\lambda)  ),   \]
where $m(\lambda)$ is the $\mathbb{C}$-linear endomorphism $V$ defined to be the derivative of $m(\lambda)$. We use Jacobi's formula: \[\mathrm{Tr}_V( m(\lambda)^{-1} \circ m'(\lambda)  )= \frac{   {\frac{\d}{\d\lambda}\det (m(\lambda))}  }{\det (m(\lambda))}  . \]
Finally since the volume of $\Im X_T^G$ is normalized to be $1$, by definition of contour integration and argument principle, the  integral \[\int_{\Im X_T^G} \frac{   {\frac{\d}{\d\lambda}\det (m(\lambda))}  }{\det (m(\lambda))} \lambda \d \lambda     \]
equals $Z_{\vert  \lambda\vert  <1}(\det(m(\lambda)))$. 
\end{proof}

We are going to apply this result to intertwining operators. Although, the operator $\M(w,\lambda)$ is a morphism from $\mathcal{\mathcal{A}}_{B,\psi}$ to $\mathcal{\mathcal{A}}_{B, w(\psi)}$, the representation structures of $\mathcal{\mathcal{A}}_{B,\psi}$ and $\mathcal{\mathcal{A}}_{B, w(\psi)}$ are the same. Let $V$ be the parabolic induction of $\psi$ from $B(\AAA)$ to $G(\AAA)$. We fix a $G(\AAA)$-equivariant isomorphism: $ \iota_1: \mathcal{\mathcal{A}}_{B,\psi}\longrightarrow V $ and $ \iota_2: \mathcal{\mathcal{A}}_{B,w(\psi)}\longrightarrow V$. 
Then \[ \mathbf{Tr}_{\mathcal{A}_{B,{\psi}}}(( - \frac{1}{\mu^{-1} - \mu}    \M(w, \lambda)^{-1}\circ \M(w,\lambda/\mu )+ \frac{1}{\mu^{-1} - \mu} ) \circ  R(f_{S_0})) ,   \]
equals 
\[ \mathbf{Tr}_{V}(( - \frac{1}{\mu^{-1} - \mu}   ( \iota_2\M(w, \lambda) \iota_1^{-1} )^{-1} \  \circ \iota_{2}\M(w,\lambda/\mu )\iota_{1}^{-1} + \frac{1}{\mu^{-1} - \mu} ) \circ \iota_1 R(f_{S_0}) \iota_{1}^{-1}  ) ,   \]

We apply Lemma \ref{determinant}. Note that given a finite family of complex vector spaces $V_i$ of dimension $n_i$ and endomorphisms $\phi_i\in \End(V_i)$, we have \[ \det(\otimes_{i}\phi_i)=\prod_i\det(\phi_i)^{\prod_{j\neq i} n_i}.  \]
By Eulerian expansion in \ref{eulerexpansion}, and local calculations in \ref{Kcase}, \ref{regcase}, \ref{nrcase}, we deduce that the integral
\[\int_{\Im X_{T}^{{G}}}
\lim_{\mu\longrightarrow 1}  \mathbf{Tr}_{\mathcal{A}_{B,{\psi}}}(( - \frac{1}{\mu^{-1} - \mu}    \M(w, \lambda)^{-1}\circ \M(w,\lambda/\mu )+ \frac{1}{\mu^{-1} - \mu} ) \circ  R(f_{S_0})) \d\lambda ,
 \] 
 equals 
\begin{equation}\label{okbb}\frac{2^{\vert  S_1\vert  }}{2} Z_{\vert  \lambda\vert  <1}(\frac{L(\psi_1\psi_2^{-1}, \lambda^2)}{ L(\psi_1\psi_2^{-1}, q^{-1}\lambda^2) }) +  \frac{2^{\vert  S_1\vert  -1}}{2} ( 2\sum_{v\in S_1} \deg v  ).  \end{equation}

 Let $\psi\in \mathcal{A}_{cont}$ be regular with the correct ramification.  By Proposition \ref{Riemann}, \eqref{okbb} equals
 \[2^{\vert  S_1\vert  } (   2g-2 +\deg S_{r}  )  + (2^{\vert  S_1\vert  -1})\deg S_1 .  \]
We have
\begin{equation}
\label{S_u} J^1_\psi(f)= \sum_{S_0\subseteq S_u} (-1)^{\vert  S_0\vert  }2^{\vert  S_0\vert  } \left(2^{\vert  S_1\vert  } (   2g-2 +\deg S_{r}  )  + (2^{\vert  S_1\vert  -1})\deg S_1\right)  .  \end{equation}

If $\psi\in \mathcal{A}_{cont}$ is {non-regular}, then $\psi$ can not have correct ramification if $S_{r}$ is non-empty. Suppose that $S_{r}=\emptyset$ and $\psi$ has correct ramification, then by Proposition \ref{Riemann},
$J_\psi^1(f)$ is the sum of 
\begin{equation}\label{13} \frac{1}{2}\sum_{S_0\subseteq S_u} (-1)^{\vert  S_0\vert  }2^{\vert  S_0\vert  } \left(2^{\vert  S_1\vert  } (   2g-1  )  + (2^{\vert  S_1\vert  -1})\deg S_1\right)    \end{equation}
with 
\begin{equation}\label{14}  
\frac{1}{8}\sum_{S_0\subseteq S_u} (-1)^{\vert  S_0\vert  }2^{\vert  S_0\vert  } \biggr(    \sum_{       \lambda_{G}\in \{  \pm 1 \}     }    \sum_{\substack{ \lambda_{w}\in \Im X_{T}^{G}   \\ \lambda_{w}^2=\lambda_{G}^{-1}   }       }   {\lambda_{G}}   \mathbf{Tr}_{\mathcal{A}_{B,{\psi}}}(   
  \M(w,  w^{-1}({\lambda_{w}})) \circ R(f_{S_0}) )\biggr).  \end{equation}

Next, we need a lemma for the sum over $S_0\subseteq S_u$ in \eqref{S_u}, \eqref{13} and \eqref{14}. 

\begin{lemm}
Let $S$ be a finite set of places of $F$. 
We have \[\sum_{I\subseteq S} (-1)^{\vert  I\vert  }=\begin{cases} 1, \quad \text{ if } S=\emptyset; \\
0, \quad \text{ if } S\not=\emptyset. 
\end{cases}
  \]
   We also have 
  \[\sum_{I\subseteq S} (-1)^{\vert  S\vert  -\vert  I\vert  }\deg I =\begin{cases} 0, \quad \text{ if } \vert  S\vert  \geq 2; \\
\deg v, \quad \text{ if } S=\{v\}; \\
0, \quad \text{ if } S=\emptyset.
\end{cases}
  \]
\end{lemm}
\begin{proof}
 The first equality is nothing else than the evaluating of $x_\alpha=1$ ($\alpha\in S$) in the expansion of the expression $\prod_{\alpha\in S}(1-x_\alpha)$. Let us prove the second equality by induction. We may suppose $|S|\geq 1$; otherwise, the equality is trivial. Let $v\in S$, we have 
 \[\begin{split}\sum_{I\subseteq S} (-1)^{\vert  S\vert  -\vert  I\vert  }\deg I &= \sum_{I\subseteq S-\{v\} } (-1)^{\vert  S\vert  -\vert  I\vert  }\deg I +\sum_{v\in I\subseteq S} (-1)^{\vert  S\vert  -\vert  I\vert  }\deg I  \\ 
 &= \sum_{I\subseteq S-\{v\} } (-1)^{\vert  S\vert  -\vert  I\vert  } (-\deg v) . 
 \end{split} \]
 Therefore, the result follows from the first equality. 
\end{proof}

After this lemma, the expression $J^1_{\psi}(f)$ given by \eqref{S_u}  vanishes if $\vert  S_u\vert  \geq 2$. When $S_u=\{v\}$, it equals
\(  \deg v.   \) When $ S_u=\emptyset $, it equals $2g-2+\deg S_{r}$. 

The expression \eqref{13} vanishes if $\vert  S_u\vert  \geq 2$. When $S_u=\{v\}$, it equals
\(  \frac{1}{2}\deg v.   \) When $ S_u=\emptyset $, it equals $\frac{1}{2}(2g-1)$.

Let $\psi\in \mathcal{A}_{cont}$ be non-regular, 
we're going to consider the following expression that appears in \eqref{14}
\begin{equation}\label{discrete1}
\frac{1}{8}  \sum_{       \lambda_{G}\in \{  \pm 1 \}     }    \sum_{\substack{ \lambda_{w}\in \Im X_{T}^{G}   \\ \lambda_{w}^2=\lambda_{G}^{-1}   }       }   {\lambda_{G}}   \mathbf{Tr}_{\mathcal{A}_{B,{\psi}}}(   
  \M(w,  w^{-1}({\lambda_{w}})) \circ R(f_{S_0}) ).  \end{equation}
Note that in this case, $L(\psi_1\psi_2^{-1}, z)=\zeta_X(z)$ where $\zeta_X$ is the zeta function of the curve $X$.
From the local calculations in \ref{Kcase} and \ref{nrcase}, we know that
\[\mathbf{Tr}_{\mathcal{A}_{B,{\psi}}}(   
  \M(w,  w^{-1}({\lambda_{w}})) \circ R(f_{S_0}) ) =  q^{1-g}  \frac{\zeta_X(\lambda_w^{-2})}{ \zeta_X(q^{-1}\lambda_w^{-2})}  \prod_{v\in S_1}  \frac{ (1-q_v^{-1}) ( \lambda_w^{-2\deg v} +1 ) }{1-q_v^{-1}\lambda_w^{-2\deg v}}  .  \]
 Here we should regard $\frac{\zeta_X(z)}{ \zeta_X(q^{-1}z)}$ as a rational function so that the pole at $z=1$ of the denominator and numerator are canceled. 
 For each $\lambda_G\in \{\pm 1\}$, there are two $\lambda_w$ such that $\lambda_w^{2}=\lambda_G^{-1}$. The expression \eqref{discrete1} equals:
  \begin{equation}\label{discrete}
 \sum_{       \lambda_{G}\in \{  \pm 1 \}     } \frac{1}{4}   q^{1-g}   {\lambda_{G}}   \frac{\zeta_X(\lambda_G)}{ \zeta_X(q^{-1}\lambda_G)}  \prod_{v\in S_1}  \frac{ (1-q_v^{-1}) ( \lambda_G^{\deg v} +1 ) }{1-q_v^{-1}\lambda_G^{\deg v}} .  \end{equation}
 
There is a polynomial $P(z)$ of degree $2g$ such that  \[\zeta_X(z)/\zeta_X(zq^{-1})=\frac{P(z)(1-q^{-1}z)}{P(zq^{-1})(1-qz)}.   \]
By functional equation of  zeta function, we have \[\frac{P(1)(1-q^{-1})}{P(q^{-1})(1-q)} =  -q^{g-1}  , \]
and 
\[\frac{P(-1)(1+q^{-1})}{P(-q^{-1})(1+q)} =  q^{g-1}  . \]

 The corresponding summand for $\lambda_G=1$ in the  expression \eqref{discrete} equals 
  \[  \frac{1}{4}  q^{1-g} 2^{\vert  S_1\vert   }    \frac{P(1)(1- q^{-1})}{P( q^{-1})(1- q)} =- 2^{\vert  S_1\vert   -2}.  \]


If there is a place of odd degree in $S_1$, then the corresponding summand for $\lambda_G=-1$ in the expression \eqref{discrete} vanishes. If all the places in $S_1$ are of even degree, then the corresponding summand for $\lambda_G=-1$ in the  expression \eqref{discrete} equals
 \[ \frac{1}{4} q^{1-g}(-1)   \frac{P(-1)(1+ q^{-1})}{P( -q^{-1})(1+ q)} 2^{\vert  S_1\vert   } = -2^{\vert  S_1\vert   -2}.  \]

  
 We decompose $S_u$ as the union of the set of places of odd degree and those of even degree:
 \[S_u=S_{u, odd}\cup S_{u, even}. \]
Following the above discussions,  the expression \eqref{14} equals   
   \begin{equation}\label{trois}  \sum_{S_0\subseteq S_u} (-1)^{|S_0|}2^{|S_0|}(- 2^{\vert  S_1\vert   -2})  +\sum_{S_{u,odd} \subseteq S_0\subseteq S_u } (-1)^{\vert  S_0\vert  }2^{\vert  S_0\vert  } \left(  -2^{\vert  S_1\vert   - 2}    \right) +   
   \sum_{  S_0\in \mathcal{S}  } 0 , \end{equation}
  where $\mathcal{S}=\{ S_0 \vert   S_0\subseteq S_u \}-\{S_0 \vert   S_{u,odd} \subseteq S_0\subseteq S_u  \}$. 
  The first sum in \eqref{trois} is $-\frac{1}{4}$ if $S_u=\emptyset$ and is $0$ if $S_u\neq \emptyset$. The second sum in \eqref{trois} equals $0$ if $S_{u,even}\neq \emptyset$, and equals 
$(-1)^{|S_u|+1}2^{|S_u|-2}$ if $S_{u,even}= \emptyset$. 
  
  In conclusion, if $S_{u}=\emptyset$, then  the expression \eqref{14} equals $-\frac{1}{2}$. If $S_u\neq \emptyset$ but $S_{u,even}=\emptyset$, then  \eqref{14} equals $(-1)^{|S_u|+1}2^{\vert  S_u\vert  -2}$. If $S_{u,even}\neq \emptyset$, \eqref{14}  equals \(0\).

If $S_{r}\neq \emptyset,$ then $J^1_{\psi}(f)\neq 0$ implies that $\psi$ is regular. In this case, there are \[\frac{c_{\mathfrak{R}}(1)}{2}\Pic(1)^2\] equivalent classes of $\psi$ such that $J^1_{\psi}(f)$ can be non-zero.  If $S_{r}=\emptyset$, then there are \[c_\mathfrak{R}(1)\frac{1}{2} \Pic(1)  ( \Pic(1)  -1) \]
regular classes of $\psi$ such that $J^1_{\psi}(f)$ can be non-zero, and \[c_{\mathfrak{R}}(1) \Pic(1)  \]
non-regular such classes.

 To complete the proof of Proposition \ref{continoust}, we summarize that if $S_{r}\neq \emptyset$ then there is no non-regular $\psi$ with correct ramification and $J_{cont}^1(f)$ is equal to $\frac{c_\mathfrak{R}(1)}{2}\Pic(1)^2$ times \eqref{S_u};  
  if $S_{r}=\emptyset$, $J^1_{cont}(f)$ is equal to $c_\mathfrak{R}(1) \frac{1}{2} \Pic(1)  ( \Pic(1)  -1) $ times \eqref{S_u} plus $c_{\mathfrak{R}}(1)\Pic(1)$ times (\eqref{13}+\eqref{14}).

\section{Geometric side of the trace formula and Hitchin moduli spaces}\label{HIG}
In this section, we treat the geometric side of the trace formula $J^{1}_{geom}(f)$ for the function $f$ introduced in Proposition \ref{f}.  The goal is to prove Theorem \ref{Higg}. We need to introduce a Lie algebra analog of the trace formula that helps us to treat $J^{1}_{geom}(f)$. 

\subsection{Moduli of parabolic Hitchin bundles over $\overline{\mathbb{F}}_q$}
Now we introduce the coarse moduli space of semistable Hitchin bundles with parabolic structures. It is constructed by Yokogawa (\cite{Yokogawa0}) using GIT theory.
 Let $V\subset |X|$ be a finite set of closed points. 
 We are interested in moduli of Hitchin bundles with flag structures at $V$.  
 If we can identify $V$ with a subset of $X(\mathbb{F}_q)$, i.e. every place in $V$ has degree $1$, then Yokogawa's construction perfectly suits our needs. Otherwise, we meet a problem since a point in $V$ can be split into several points over an algebraically closed field, and we have to make extra arguments about how to treat with (semi)-stability. 
To remedy it, we work over an algebraically closed field or a large enough extension of $\mathbb{F}_q$ where we can apply Yokogawa's construction to obtain a moduli space. Then we 
 define an $\mathbb{F}_q$-structure on it.

Let $D$ be a divisor over $X$ and $V$ be a set of closed points of $X$. 
 We identify $\overline{V}:=V\times_{\Spec(\mathbb{F}_q)}\Spec(\overline{\mathbb{F}}_q)$ with a subset of $\overline{X}(\overline{\mathbb{F}}_q)$ and view $D$ also as a divisor over $\overline{X}$. 

A parabolic Hitchin triple (or parabolic Hitchin bundle) over $\overline{X}$ is a triple \[(\mathcal{E}, \varphi, (L_x)_{x\in \overline{V}})\] where $(\mathcal{E}, \varphi)$ is a Hitchin pair for the divisor $D$, i.e., a vector bundle $\mathcal{E}$ together with a bundle morphism, called the Higgs field, 
\[ \varphi: \mathcal{E}\longrightarrow \mathcal{E}(D):= \mathcal{E}\otimes\mathcal{O}_{\overline{X}}(D), \]
and for each point $x\in \overline{V}$, $L_x$ is a line in the $\overline{\mathbb{F}}_q$-vector space $\mathcal{E}_x$, the fiber over $x$ of the bundle $\mathcal{E}$, such that \[\varphi_x(\mathcal{E}_x) \subseteq L_x \]
and \[ \varphi_x(L_x)=0 \]
 where we view $L_x$ also as a line in $\mathcal{E}(D)_x$. 

We select a canonical divisor $K_X$ on $X$, which is a divisor that corresponds to the canonical bundle.
If $D=K_X+\sum_{v\in V}v$, we will call them parabolic Higgs bundles.

For each $x\in \overline{V}$, let $\overline{\xi}_x:=(\overline{\xi}_{x,1}, \overline{\xi}_{x,2})\in \mathbb{Q}^2$ such that $\overline{\xi}_{x,1}\geq \overline{\xi}_{x,2}\geq \overline{\xi}_{x,1}-1$. Let $\overline{\xi}=(\overline{\xi}_x)_{x\in \overline{V}}$. 
Let $(\mathcal{E}, \varphi, (L_x)_{x\in \overline{V}})$ be a parabolic Hitchin  bundle over $\overline{X}$. Let $\mathcal{L}$ be a sub-line bundle of $\mathcal{E}$, we define the parabolic degree $\mathrm{p\text{-}deg}(\mathcal{L})$ by
\[\mathrm{p\text{-}deg}(\mathcal{L}):= \deg(\mathcal{L})+\sum_{x\in \overline{V}}\begin{cases}  \overline{\xi}_{x, 1}, \text{ if }  \mathcal{L}_x=L_x; \\
\overline{\xi}_{x, 2},  \text{ if }  \mathcal{L}_x \neq L_x.
\end{cases}
\]
We say that $(\mathcal{E}, \varphi, (L_x)_{x\in \overline{V}})$ is $\overline{\xi}$-semistable if for any sub-line bundle $\mathcal{L}$ of $\mathcal{E}$ satisfying $\varphi(\mathcal{L})\subseteq \mathcal{L}(D)$, we have
 \[\mathrm{p\text{-}deg}(\mathcal{L})\leq \frac{\deg \mathcal{E}+\sum_{x\in \overline{V}}(\overline{\xi}_{x,1}+\overline{\xi}_{x,2}  )}{2}. \] 
 It is said to be $\overline{\xi}$-stable if the strict inequality always holds.   
Note that if \[\deg(\mathcal{E})+\sum_{x\in \overline{V} }\pm (\overline{\xi}_{x,1}-\overline{\xi}_{x,2}  ) \notin 2\mathbb{Z},\] then the equality can never be achieved and $\overline{\xi}$-semistablity coincides with $\overline{\xi}$-stability. 
We say that such cases are in general position.

In \cite{Yokogawa}, 
Yokogawa has constructed a coarse moduli space, which he shows to be a variety defined over $\overline{\mathbb{F}}_q$, that classifies isomorphism classes of $\overline{\xi}$-stable parabolic Hitchin bundles $(\mathcal{E}, \varphi, (L_x)_{x\in \overline{V}})$ with $\mathcal{E}$ being of rank $2$ and degree $e$. We denote the variety by \[{M}_{\overline{V}}^{e,\overline{\xi}}(D).\] 
In this article, we will only be interested in the case where $D=K_X+\sum_{v\in R}v$ for a subset $R$ of $S$ ($R$ can be empty).



\begin{remark}\label{qpw}
Yokogawa's results apply under the assumption (under our terminology) that $\overline{\xi}_{x, 1}>\overline{ \xi}_{x, 2}>\overline{\xi}_{x, 1}-1$. In particular, it does not include the case that $\overline{\xi}_{x,1}=\overline{\xi}_{x,2}$. However, as long as $\overline{\xi}$ stays in general position, this is not an issue because when $\overline{\xi}_{x,1}$ and $\overline{\xi}_{x,2}$ are close enough, semistability of parabolic Hitchin bundles coincide with the case that $\overline{\xi}_{x,1}=\overline{\xi}_{x,2}$.

\end{remark}

\subsection{${\mathbb{F}}_q$-points}\label{Fq}
Suppose that $n$ is a divisible enough integer such that every closed point in $V$ totally splits over $\mathbb{F}_{q^n}$, i.e., the residue field of a place in $V$ can be embedded in $\mathbb{F}_{q^n}$. Yokogawa's construction works if the curve is $X\otimes \mathbb{F}_{q^n}$ and the parabolic structures are imposed at each point of $V\otimes\mathbb{F}_{q^n}$ and in this way ${M}_{\overline{V}}^{e,\overline{\xi}}(D)$ has an $\mathbb{F}_{q^n}$-structure. In the following, we are going to endow ${M}_{\overline{V}}^{e,\overline{\xi}}(D)$ with an $\mathbb{F}_q$-structure when $\overline{\xi}_x$ are the same for points $x\in V\otimes \mathbb{F}_{q^n}$ lying over each $v\in V$.

For a variety $Y$ defined over $\mathbb{F}_{q^n}$, let $F_{Y/\mathbb{F}_{q^n}}$ be the arithemetic Frobenius morphism of $Y\otimes_{\mathbb{F}_{q^n}} \overline{\mathbb{F}}_{q}$. Recall that it is the morphism of schemes $Y\otimes_{\mathbb{F}_{q^n}} \overline{\mathbb{F}}_{q}$ that is identity on $Y$ and is 
 $x\mapsto x^{{q^n}}$ over $\Spec(\overline{\mathbb{F}}_q)$. In particular, it is not a morphism of ${\mathbb{F}}_{q^n}$-schemes.  
 
The pullback by $F_{X/\mathbb{F}_q}$ defines an action on the set of isomorphism classes of Hitchin bundles. To define an action on the parabolic structure, we need to use an equivalent definition that identifies vector bundles with locally free sheaves. Suppose  $(\mathcal{E}, \varphi, (L_x)_{x\in \overline{V}})$ is a parabolic Hitchin bundle over $\overline{X}$, For each $x\in \overline{V}$, $L_x$ defines a unique rank $2$ coherent sub-sheaf $\mathcal{E}^{x}$ of $\mathcal{E}$ such that $\mathcal{E}/\mathcal{E}^{x}$ is a skyscraper sheaf of degree $1$  supported in $\{x\}$ and the inclusion $\mathcal{E}^{x}\longrightarrow \mathcal{E}$ defines a morphism of their fibers at $x$ with image $L_x$ in $\mathcal{E}_x$. Then $F_{X/\mathbb{F}_q}^{\ast}\mathcal{E}^{x}$ defines a parabolic structure of $F_{X/\mathbb{F}_q}^\ast\mathcal{E}$ at $\Fr(x)$.

 \begin{prop-def}\label{ProDef}
 Suppose that the family $(\xi_v)_{v\in V}\in (\mathbb{Q}^2)^{V}$ satisfies that for any $v\in V$, 
\[ 0\leq \xi_{v, 1}-\xi_{v, 2}\leq [\kappa_v: \mathbb{F}_q]. \]
 Let \[ \overline{\xi}_x=\frac{1}{[\kappa_v:\mathbb{F}_q]}\xi_{v},\] for any point $x\in \overline{V}$ lying over $v\in V$.

 Let $\sigma$ be the Frobenius element in $\Gal(\overline{\mathbb{F}}_{q}/\mathbb{F}_q)$. 
 We define an action of $\Gal(\overline{\mathbb{F}}_{q}/\mathbb{F}_q)$ on ${M}_{\overline{V}}^{e,\overline{\xi}}(D)(\overline{\mathbb{F}}_{q})$ so that $\sigma$ sends a parabolic Higgs bundle $(\mathcal{E}, \varphi, (L_x)_{x\in \overline{V}})$ over $\overline{X}$ to \[ F_{{X}/\mathbb{F}_q}^{\ast}(\mathcal{E}, \varphi, (L_x)_{x\in \overline{V}}).\] 
There is a variety defined over $\mathbb{F}_q$ whose $\mathbb{F}_{q^k}$-points are exactly those in ${M}_{\overline{V}}^{e,\overline{\xi}}(D)(\overline{\mathbb{F}}_{q})$ fixed by $\sigma^k$ and whose base change to ${\overline{\mathbb{F}}}_{q}$ is isomorphic to ${M}_{\overline{V}}^{e,\overline{\xi}}(D)$. We denote the variety by ${M}_{V}^{e,{\xi}}(D)$.
 \end{prop-def}
 \begin{proof}
 Let $M={{M}}_{\overline{V}}^{e,\overline{\xi}}(D)$ in this proof.

First, we show that the action is well-defined. 
As $F_{X/\overline{\mathbb{F}}_q}^{\ast}$ sends the set of Hitchin bundle with parabolic structures in $\overline{V}$ to itself. 
It is sufficient to show that $\overline{\xi}$-stable parabolic Hitchin bundle is sent to a $\overline{\xi}$-stable one. Note that for any sub-line bundle $\mathcal{L}$ of $\mathcal{E}$  such that  $\varphi(\mathcal{L})\subseteq \mathcal{L}(D)$, the parabolic degree of $F_{{X}/\mathbb{F}_q}^{\ast}\mathcal{L}$ (as a subbundle of $F_{{X}/\mathbb{F}_q}^{\ast} \mathcal{E}$) equals that of $\mathcal{L}$. Moreover, any subline bundle of $F_{{X}/\mathbb{F}_q}^{\ast} \mathcal{E}$ can be written as $F_{{X}/\mathbb{F}_q}^{\ast} \mathcal{L}$ for some subline bundle $\mathcal{L}\subseteq \mathcal{E}$ because $F_{X/\mathbb{F}_q}$ is an automorphism. 


Let $U$ be an $\overline{\mathbb{F}}_q$-scheme and $\mathcal{E}$ a family of parabolic Hitchin bundles over $\overline{X}_U$. Since $X$ is defined over $\mathbb{F}_q$, we have $(\overline{X}_U)^{(q)}\cong \overline{X}_U$. 
Applying the functor $(\cdot)^{(q)}$, $\mathcal{E}^{(q)}$ is still a family over $\overline{X}_U$. 
By above arguments, \cite[Theorem 4.6, (4.6.5)]{Yokogawa} shows that we have an isomorphism $\phi: M\rightarrow M$ that induces $\sigma$. By definition, it makes the following diagram commutative:
\[ \begin{CD} 
M@>\phi>>  M \\
@VVV@VVV \\
\overline{\mathbb{F}}_q@>x\mapsto x^q>>\overline{\mathbb{F}}_q
\end{CD}. 
\]
Note that for $n$ divisible enough so that $V\otimes \mathbb{F}_{q^n}$ totally split (i.e. $n$ divides the degree of every $v\in V$), $M$ has an $\mathbb{F}_{q^n}$-structure and $\phi^n=Id$. 
By \cite[Example B, p.139]{BLR} and the fact that $M$ is quasi-projective, this defines an $\mathbb{F}_q$-structure on $M$. 
\end{proof}

Next, we are going to study  ${M}_{V}^{e,{\xi}}(D)(\mathbb{F}_{q})$, the fixed points of $\sigma$ on ${M}_{\overline{V}}^{e,\overline{\xi}}(D)(\overline{\mathbb{F}}_{q})$. 
 We define a rank $2$ parabolic Hitchin bundle as a tuple $(\mathcal{E}, \varphi, (L_v)_{v\in {V}})$  over $X$  consisting of a vector bundle   $\mathcal{E}$ over $X$, a bundle morphism $\varphi: \mathcal{E}\longrightarrow \mathcal{E}(D)$ and for each $v\in V$ a $1$ dimensional $\kappa_v$ sub-vector space $L_v$ of $\mathcal{E}_v$ such that $\varphi_v(\mathcal{E}_v)\subseteq L_v$ and $\varphi_v(L_v)=0$. Let $(e,\xi)\in \mathbb{Z}\times (\mathbb{Q}^{2})^{  V  }$. 
Let $\mathcal{L}$ be a sub-line bundle  of $\mathcal{E}$  such that  $\varphi(\mathcal{L})\subseteq \mathcal{L}(D)$, we define the parabolic degree $\mathcal{L}$ by
\[\mathrm{p\text{-}deg}(\mathcal{L}):= \deg(\mathcal{L})+\sum_{v\in {V}}\begin{cases}  {\xi}_{v, 1}, \text{ if }  \mathcal{L}_v=L_v; \\
{\xi}_{v, 2},  \text{ if }  \mathcal{L}_v \neq L_v.
\end{cases}
\]
The parabolic Higgs bundle $(\mathcal{E}, \varphi, (L_v)_{v\in {V}})$ 
is semistable if for all such $\mathcal{L}$, we have
\[ \mathrm{p\text{-}deg}(\mathcal{L})\leq \frac{\deg \mathcal{E}+\sum_{v\in V}(\xi_{v,1}+\xi_{v,2}) }{2}.  \]

\begin{prop}
Under the hypothesis of Proposition-Definition \ref{ProDef}, the set of (isomorphism classes of) $\overline{\xi}$-stable parabolic Higgs bundles $({\mathcal{E}}, {\varphi}, (L_{x})_{x\in \overline{V}})$ over $\overline{X}$ that are fixed by $F_{X/\mathbb{F}_q}^{\ast}$ is in bijection with the set of $\xi$-stable parabolic Higgs bundles $({\mathcal{E}}_0, {\varphi}_0, (L_{v})_{v\in {V}})$
 over $X$.  

\end{prop}
\begin{proof}
Let $\sigma$ be the Frobenius element that generates $\Gal(\overline{\mathbb{F}}_q/\mathbb{F}_q)$. 
By Galois descent (see \cite[Example B. p.139]{BLR}), the category of vector bundles over $X$ is equivalent to the category of $\Gal(\overline{\mathbb{F}}_q/\mathbb{F}_q)$-equivariant vector bundles over $\overline{X}$, that is the category of vector bundles $\mathcal{E}$ over $\overline{X}$ together with an isomorphism for each $i\geq 1$,
\[ \phi_{\sigma^i}: ({\mathcal{E}}, {\varphi}, (L_x)_{x\in \overline{V}}) \longrightarrow F_{X/\mathbb{F}_q}^{\ast i}({\mathcal{E}}, {\varphi}, (L_x)_{x\in \overline{V}}), \]
such that $\sigma^{i}\mapsto\phi_{\sigma^i}$ satisfies the cocycle condition:
\[ \sigma^{i}(\phi_{\sigma^j}) \circ \phi_{\sigma^i}= \phi_{\sigma^{i+j}}, \]
and  for some $d\in \mathbb{N}^{\ast}$, $\phi_{\sigma^d}$ is the identity. Note that the category of vector bundles is equivalent to the category of locally free sheaves, hence every vector bundle defined over $\overline{X}$ is automatically defined over $X\otimes \mathbb{F}_{q^n}$ for some $n\in \mathbb{N}^{\ast}$ and the requirement that $\phi_{\sigma^d}$ is the identity map makes sense when $d$ is divisible by $n$. 

By Galois descent for morphisms of quasi-coherent sheaves (\cite[Proposition 1, p.130]{BLR} and \cite[Example B. p.139]{BLR}), the above equivalence extends to an equivalence between the category of parabolic Hitchin bundles over $X$ is equivalent to the category of $\Gal(\overline{\mathbb{F}}_q/\mathbb{F}_q)$-equivariant parabolic Hitchin bundles over  $\overline{X}$. Indeed, the parabolic structure is determined by a sub-sheaf $\mathcal{F}\subseteq \mathcal{E}$ such that $\mathcal{E}/\mathcal{F}$ is a skyscraper sheaf supported in $\overline{V}$ and that the inclusion  $\mathcal{F}\rightarrow \mathcal{E}$ defines a morphism $\mathcal{F}_x\rightarrow \mathcal{E}_x$ of their fibers at each point $x$ with image $L_x$ in $\mathcal{E}_x$. If both $\mathcal{F}$ and $\mathcal{E}$ come from $X$, i.e. $\mathcal{F}=\mathcal{F}_0|_{\overline{X}}$ and $\mathcal{E}=\mathcal{E}_0|_{\overline{X}}$ then $\mathcal{F}_0\subseteq \mathcal{E}_0$ is a subsheaf such that the quotient $\mathcal{E}_0/\mathcal{F}_0$ is a skyscraper sheaf supported in $V$ and the map of the fibers $\mathcal{F}_{v}\rightarrow \mathcal{E}_v$  has a $1$-dimensional image as $\kappa_v$-vector space.

 By our assumption, 
there is an isomorphism
\[ \phi_{\sigma}: ({\mathcal{E}}, {\varphi}, (L_x)_{x\in \overline{V}}) \longrightarrow F_{X/\mathbb{F}_q}^{\ast}({\mathcal{E}}, {\varphi}, (L_x)_{x\in \overline{V}}). \] 
We can define simply for every $i\geq 1$, \[\phi_{\sigma^i}:=\sigma^{i-1}(\phi_\sigma)\circ \cdots \circ  \sigma (\phi_\sigma) \circ \phi_\sigma.  \]
Note that by cocycle condition, this is the unique possible way to extend $\phi_{\sigma}$ to a $1$-cocycle. 

We will prove that we can choose  $\phi_{\sigma}$ so that for some $n$, the isomorphism $\phi_{\sigma^n}$ defined above is the identity. Suppose that $(\mathcal{E}, \varphi, (L_x)_{x\in \overline{V}})$ is defined over $X\otimes \mathbb{F}_{q^n}$. Then \[F_{X/\mathbb{F}_q}^{\ast n}({\mathcal{E}}, {\varphi}, (L_x)_{x\in \overline{V}})=({\mathcal{E}}, {\varphi}, (L_x)_{x\in \overline{V}}). \]
Any $\phi_{\sigma^n}$ is an automorphism of $({\mathcal{E}}, {\varphi}, (L_x)_{x\in \overline{V}})$ which must be a scalar multiplication because $({\mathcal{E}}, {\varphi}, (L_x)_{x\in \overline{V}})$ is stable.   
Hence there is a $\lambda'\in \overline{\mathbb{F}}_q^{\times}$ such that 
\[\phi_{\sigma^n}=\lambda'.\mathrm{id}_{(\overline{\mathcal{E}}, \overline{\varphi}, (L_{x})_{x\in \overline{V}})}. \]
Suppose $\lambda'\in \mathbb{F}_{q^m}^\times$ for some $m\geq 1$. We deduce that 
$\phi_{\sigma^{nm}} = \lambda. \mathrm{id}_{(\overline{\mathcal{E}}, \overline{\varphi}, (L_{x})_{x\in \overline{V}})}$ with \[ \lambda=(\lambda')^{1+q^n+\cdots +q^{n(m-1)}}\in \mathbb{F}_{q^m}^\times \subseteq \mathbb{F}_{q^{nm}}^\times.   \]
Note that this implies that 
\begin{equation}\sigma^{nm-1}(\phi_\sigma)\circ \cdots \circ  \sigma (\phi_\sigma) \circ \phi_\sigma=\lambda.\mathrm{id}_{(\overline{\mathcal{E}}, \overline{\varphi}, (L_{x})_{x\in \overline{V}})} \end{equation}
Applying $\sigma$ to both sides, we obtain that 
\begin{equation}\label{lamb}\sigma^{nm}(\phi_\sigma)\circ \cdots \circ  \sigma^2 (\phi_\sigma) \circ \sigma(\phi_\sigma)=\sigma(\lambda).\mathrm{id}_{ F_{X/\mathbb{F}_q}^{\ast}(\overline{\mathcal{E}}, \overline{\varphi}, (L_{x})_{x\in \overline{V}})}. \end{equation}
Since $\sigma^{nm}(\phi_\sigma)=\phi_\sigma$, we deduce that $\sigma(\lambda)=\lambda$ and hence $\lambda\in \mathbb{F}_q^\times$. 
As the norm map from $\mathbb{F}_{q^{nm}}$ to $\mathbb{F}_q$ is surjective, we conclude that by modifying $\phi_\sigma$ by a scalar, we can get a $\phi_\sigma$ such that $\lambda=1$. 

It is easy to see that such descent data are unique up to isomorphism. In fact, any two isomorphisms $\phi_{\sigma}$ and $\phi_{\sigma}'$ are differed by a scalar because these are the only automorphisms of $(\mathcal{E},\varphi, (L_x)_{x\in \overline{V}})$. As the map $\lambda\mapsto\lambda/\sigma(\lambda)$ from $\overline{\mathbb{F}}_{q}$ to $\overline{\mathbb{F}}_q$ is surjective (Hilbert's Theorem 90), we conclude that the cocycle given by $\phi_{\sigma}$ is isomorphic to that of $\phi_{\sigma}'$.

It remains to prove that the parabolic Hitchin bundle $(\mathcal{E}_0,\varphi_0, (L_v)_{v\in {V}})$ determined by a descent datum attached to $(\mathcal{E},\varphi, (L_x)_{x\in \overline{V}})$ is stable and reversely if $(\mathcal{E}_0,\varphi_0, (L_v)_{v\in {V}})$ is stable then the parabolic Hitchin bundle $(\mathcal{E},\varphi, (L_x)_{x\in \overline{V}})$ over $\overline{X}$ is stable. The first statement is trivial. 
For the second statement, we need to use the fact that semistability coincides with stability since the parameter $\overline{\xi}$ is in general position. 
We can argue by contradiction. Suppose $(\mathcal{E},\varphi, (L_x)_{x\in \overline{V}})$ is not semistable, then it admits a maximal destabilizing sub line bundle over $\overline{X}$ preserved by $\varphi$. Such a line bundle over $\overline{X}$ is unique, hence is fixed by Galois action. It then descends to a line bundle over $X$. By our definition of parabolic degree, it is again destabilizing. This contradicts the fact that $(\mathcal{E}_0,\varphi_0, (L_v)_{v\in {V}})$ is semistable. 
\end{proof}

\subsection{Residue morphism} \label{semistabler} 
Next, we are going to define and study a residue morphism. We will be interested in the case that $V$ is a subset of $S_u$ and $R=V\cup S_{cr}$. 
We suppose that
\[D_R=K_X + \sum_{v\in R}v.  \]
Let us denote $$M_{V}^{1}(D_R):=M_{V}^{1,0}(D_R),$$ the moduli space of strictly parabolic Hitchin bundles with parabolic structures in $V$ where we set parabolic weights to be trivial.  

Let $(\mathcal{E}, \varphi)$ be a Hitchin  bundle. The morphism $\varphi$ is equivalent to a morphism $$\mathcal{O}_X\longrightarrow \mathcal{E}{nd}(\mathcal{E})(D_R).$$
 Its characteristic polynomial $t^2+at+b$ gives a section $(a, b)$ in \[H^{0}(X, \mathcal{O}_X(D_R))\oplus H^{0}(X, \mathcal{O}_X(2D_R)).  \]
Let $\mathcal{A}_R$ be the affine space $H^{0}(X, \mathcal{O}_X(D_R))\oplus H^{0}(X, \mathcal{O}_X(2D_R))$ defined over $\mathbb{F}_q$. Let \[\mathcal{R}_R=\prod_{v\in R}\mathcal{R}_v,  \]
where $\mathcal{R}_v=\mathcal{O}_X(D_R)_v\oplus \mathcal{O}_X(2D_R)_v$, the fiber of $ \mathcal{O}_X(D_R)\oplus  \mathcal{O}_X(2D_R)$ in $v$. The space $\mathcal{R}_v$ is viewed as an $\mathbb{F}_q$-vector space by forgetting its $\kappa_v$-vector space structure. 
Note that we have
 \[\mathcal{R}_v(\mathbb{F}_q) \cong \{ t^{2}+at+b \vert  a,b \in \kappa_v \}.\]
 We have a morphism \[\mathcal{A}_R\longrightarrow\mathcal{R}_R,\]
 sending $(a,b)$ to $((a_v, b_v))_{v\in R}$. 
 
\begin{prop}Suppose that $\deg R=\vert  R\otimes_{\mathbb{F}_q}\overline{\mathbb{F}}_q\vert  >2-2g$.
The morphism $\mathcal{A}_R \longrightarrow \mathcal{R}_R$ is linear of codimension $1$. The image consists of elements $(t^2+a_vt+b_v)_{v\in R}$ such that \begin{equation}\label{residue} \sum_{v\in R}\Tr_{\kappa_v\vert  \mathbb{F}_q}(a_v)=0.  \end{equation}
Let $\mathcal{R}^1_R$ be the linear sub-scheme of $\mathcal{R}_R$ of elements satisfying the residue condition \eqref{residue}. 
\end{prop}
\begin{proof}
By Riemann-Roch theorem, it is easy to verify that when $\deg S>2-2g$, we have \[H^0(X, \mathcal{O}_X(2D))\longrightarrow \prod_{v\in R}\mathcal{O}_X(2D)_v\] is surjective. The kernel of the map \[H^0(X, \mathcal{O}_X(D))\longrightarrow \prod_{v\in R}\mathcal{O}_X(D)_v,\]
is $H^0(X, \omega_X)$. Therefore, for dimension reasons, the image is a linear subspace of codimension $1$. The last thing to observe is that the condition \eqref{residue} is necessary due to the residue theorem. 
\end{proof}

\begin{remark}\label{pneq2}\normalfont
If $p\neq 2$, there exists $b_1, b_2\in \mathbb{F}_q$ such that the polynomial \( t^2-b_1\)
is irreducible and the polynomial \(t^2-b_2\) is split with distinct roots. 
In fact, if $p\neq 2$ this is trivial because then the group endomorphism $x\mapsto x^2$ of $\mathbb{F}_q^\times$ is not surjective. 
\end{remark}

The Hitchin fibration is the morphism that sends a Hitchin bundle to the characteristic polynomial of the Higgs field which gives us a morphism of $\mathbb{F}_q$-schemes
\[{M}_{V}^{1}(D_R)\longrightarrow \mathcal{A}_R. \]
We define the residue morphism as the composition of the natural morphism $\mathcal{A}_R \longrightarrow \mathcal{R}^1_R$ with Hitchin  fibration: 
\[\mathrm{res}_R: {{M}_{V}^{1}(D_R)}\longrightarrow \mathcal{A}_R \longrightarrow \mathcal{R}_R^1,   \]
It sends a triple $(\mathcal{E}, \varphi, (L_v)_{v\in R})$ to the family of characteristic polnomial of $\varphi_v$, $v\in R$.  
Now suppose $R=S_{cr}\cup V$ and let \[\mathcal{R}^1_{S_{cr}}\] be the linear subspace of $\mathcal{R}^1_R$ whose components in $V$ are zero. As the fiber map is nilpotent at $v\in V$, the residue morphism factors through the morphism
\[  \mathrm{res}_{S_{cr}}: {M}_{V}^{1}(D_R)\longrightarrow \mathcal{R}^1_{S°cr}   . \]
Given a point $o_{S_{cr}}\in \mathcal{R}^1_{S_{cr}}(\mathbb{F}_q)$, we will denote 
\[ {M}_{V}^{1}(o):=\res_{S_{cr}}^{-1}(o_{S_{cr}}). \]


\subsection{Geometric side of trace formula}\label{geomet}
We introduce a variant of the trace formulas with an extra parameter $\xi\in (\mathbb{Q}^2)^R$. 

Let $H_B: B(\AAA)\longrightarrow \mathbb{Q}^2$ be the function defined by \[H_B(\begin{pmatrix} a& b\\ 0&d
\end{pmatrix}) = ( \deg a , \deg d ). \]
The Harish-Chandra's map is the extension of $H_B$ to the whole $G(\AAA)$ by Iwasawa decomposition, i.e., if $x=bk\in G(\AAA)$ with $b\in B(\AAA)$ and $k\in G(\mathcal{O})$, we have $H_B(x)=H_B(b)$. 

Let $\htau_B$ be the characteristic function over $\mathbb{Q}^2$ of the subset \[\{   (x,y)\in \mathbb{Q}^2 \vert  x>y   \}.   \]

For any $x\in G(\AAA)$, let $x=bk$ be its Iwasawa decomposition with $b\in B(\AAA)$ and $k=(k_v)_{v\in |X|}\in G(\mathcal{O})$. For every $v\in R$, we define $s_{x,v}$ to be the identity if $k_v\in \mathcal{I}_v$ and to be the non-trivial permutation in $\mathfrak{S}_2$ otherwise. Therefore, given  $(\xi_1, \xi_2)\in \mathbb{Q}^2$, we have \[s_{x,v} (\xi_1, \xi_2)=\begin{cases} (\xi_1, \xi_2), \quad \text{if } {k}_v\in \mathcal{I}_v; \\
(\xi_2, \xi_1), \quad \text{if } {k}_v\notin \mathcal{I}_v. 
\end{cases}
 \]

Let $f\in {C}_c^{\infty}(\ggg(\AAA))$, and $\xi=(\xi_x)_{x\in |S|}\in (\mathbb{Q}^2)^{R}$.  The $\xi$-variant of truncated trace for Lie algebra is defined by the integral
\[J^{\ggg, e,\xi}(f):=\int_{G(F)\backslash G(\AAA)^{e}} k^{\ggg,\xi}(x) \d x ,\]
where $k^{\ggg,\xi}(x)$ equals 
\[ \begin{split}
&\sum_{\gamma\in \ggg(F)} f(\ad(x)^{-1}\gamma)\\ -& \sum_{\delta\in B(F)\backslash G(F)} \htau_{B}(H_B(\delta x)+\sum_{v\in R} s_{\delta x, v}\xi_v )\sum_{\gamma\in \mathfrak{t}(F)}\int_{\nnn(\AAA)} f(\ad(\delta x)^{-1} (\gamma+U))\d U . \end{split} \]


Let \[\mathcal{E}_\ggg= \{   t^2+at+b \in F[t] \}, \] 
be the set of rank $2$ unitary polynomials. 
Let \[\mathcal{E}_G= \{   t^2+at+b \in F[t] \vert    b\neq 0  \},\] 
be the subset of $\mathcal{E}_\ggg$ consisting of polynomials whose constant term is non-zero.

For any element $\gamma\in \mathfrak{g}(F)$, we define $\chi_\gamma\in \mathcal{E}_\ggg$ to be the characteristic polynomial of $\gamma$. 
Given $\chi\in \mathcal{E}_\ggg$, we define $k_\chi^{\ggg,\xi}(x)$ for $x\in G(\AAA)$ by 
\[\begin{split}
&\sum_{\gamma\in \ggg(F), \chi_\gamma=\chi} f(\ad(x)^{-1}\gamma)  \\
-&\sum_{\delta\in B(F)\backslash G(F)} \htau_{B}(H_B(\delta x)+\sum_{v\in R} s_{\delta x, v}\xi_v )\sum_{\gamma\in \mathfrak{t}(F), \chi_\gamma=\chi}\int_{\nnn(\AAA)} f(\ad(\delta x)^{-1} (\gamma+U))\d U . \end{split}\]
We define \[J_\chi^{\ggg, e,\xi}(f):=\int_{G(F)\backslash G(\AAA)^{e}} k_\chi^{\ggg,\xi}(x) \d x. \]
We have proved in \cite[Section 5]{Yu2} that the integrals converge, and it is non-zero for only finitely many $\chi\in\mathcal{E}_\ggg$. Therefore, we  have
\[J^{\ggg, e,\xi}(f)=\sum_{\chi\in \mathcal{E}_\ggg} J_\chi^{\ggg, e,\xi}(f) . \]
If we set $\xi=0$, we'll omit $\xi$ from the notation. 

We also have the group version. 
Let $\chi \in \mathcal{E}_G$ and $f\in C_c^{\infty}(G(\AAA))$, then we define 
\[\begin{split}
k_\chi^{G, \xi}(x)=
&\sum_{\gamma\in G(F), \chi_\gamma=\chi} \sum_{i\in\mathbb{Z}} f(x^{-1}\gamma a^i x)  \\
-&\sum_{\delta\in B(F)\backslash G(F)} \htau_{B}(H_B(\delta x)+\sum_{v\in R} s_{\delta x, v}\xi_v )\sum_{\gamma\in {T}(F), \chi_\gamma=\chi} \sum_{i\in\mathbb{Z}}  \int_{N(\AAA)} f((\delta x)^{-1} (\gamma n) a^i \delta  x)\d n ,  \end{split} \]
where $a\in \mathbb{A}^\times$ is a fixed degree $1$ idèle. Note that if the support of $f$ is contained in $G(\AAA)^{0}$ (for example in $G(\mathcal{O})$), then only $i=0$ in the sum over $i\in \mathbb{Z}$ contributes. 
We define $k^{G,\xi}(x)$ to be the sum of  $k_{\chi}^{G,\xi}(x)$ over $\chi \in \mathcal{E}_G$. 
We define $J^{G, e, \xi}_\chi(f)$ (resp. $J^{G, e, \xi}(f)$) to be the integral of $k_\chi^{G, \xi}(x)$ (resp. $k^{G, \xi}(x)$) over $x\in G(F)\backslash G(\AAA)^e$. 
We have then, by definition
\[J^{G, e, \xi}(f)=\sum_{\chi \in \mathcal{E}_G} J^{G, e, \xi}_\chi (f). \]
The sum is again a finite sum. If we take $e=1$ and $\xi=0$,  we get the geometric side of the trace formula, which is Arthur's original definition adapted to a function field: \begin{equation}J^{1}_{geom}(f)=J^{G, 1, 0}(f). \end{equation}

\begin{remark}
We call $J^{G,e,\xi}(f)$ the truncated trace since the main term in the integrand of $J^{G,e,\xi}(f)$: 
\[ \sum_{\gamma\in G(F)a^\mathbb{Z}}f(x^{-1}\gamma x) \] is the diagonal evaluation of the kernel function of the regular action $R(f)$ on $L^2(G(F)\backslash G(\AAA)/a^\mathbb{Z})$. The extra term truncates the following function in $x$: \[ \sum_{\gamma\in {T}(F)a^\mathbb{Z}}\int_{N(\AAA)} f( x^{-1} (\gamma n) x)\d n,\] 
which is the diagonal evaluation of the regular action $R(f)$ on $L^2(T(F)N(\AAA)\backslash G(\AAA)/a^\mathbb{Z})$.
\end{remark}

\subsection{A geometric interpretation of $J^{1}_{geom}(f)$} 
Suppose that $f$ is the function constructed in Proposition \ref{f}. We have the following result. 

\begin{theorem}\label{Higg}
Let $o_{S_{cr}}\in \mathcal{R}^{1}_{S_{cr}}({\mathbb{F}}_{q})$ so that every polynomial $o_v$ has distinct roots and is split over $\kappa_v$ if $v\in S_r$, and is irreducible if $v\in S_c$.
Then we have \[J^{1}_{geom}(f)=\sum_{ V\subseteq S_u} (-1)^{|S_u-V|}2^{|S_u-V|} q^{-(4g-3+|\overline{R}|)}    |M^1_{V}(o_{S_{cr}})(\mathbb{F}_q)|. \]
\end{theorem}
\begin{proof}
The main new ingredient of the proof compared to \cite{Yu} is the treatment of the local components for $v\in S_u$.

Since $J^{1}_{geom}$ is a linear in $C_c^\infty(G(\AAA))$, we have \[J^{1}_{geom}(f) =\sum_{V\subseteq S_u} (-1)^{\vert  S_u-V\vert  } 2^{\vert  S_u-V\vert  } J^{1}_{geom}(f^{V}),  \]
where $f^{V}=\otimes f^{V}_v$ is the tensor product of functions $f_v^V\in C_c^\infty(G_v)$ such that \[f_v^V=f_v\]
 at a place $v\notin S_u$, for $v\in S_u-V$ it is defined so that for $x\in G_v$: 
\[f^{V}_v(x)=\mathbbm{1}_{\mathcal{K}_v}(x)\theta_v(\det \overline{x}),\]  and  for $v\in V$ it is defined  so that for $x\in G_v$: 
\[ f^{V}_v(x)=\frac{1}{\vol(\mathcal{I}_v)}\mathbbm{1}_{\mathcal{I}_v}(x)\theta_v(\det \overline{x}). \]

We apply \cite[Th. 6.2.1]{Chau}.  
It says that, as the support of $f^V\in C_c^\infty(G(\AAA))$ is contained in $G(\mathcal{O})$, we have
\[J^{1}_\chi(f^V)=0, \]
except if $\chi=(t-a)^2$ with $a\in \mathbb{F}_q^\times$. Therefore for any $V\subseteq S_u$, we have
\[ J^{1}_{geom}(f^V)=\sum_{\chi=(t-\alpha)^2, \alpha\in \mathbb{F}_q^\times} J^{1}_{\chi}(f^V).  \]
Since the eigenvalues of ramifications $\mathfrak{R}$ satisfy 
\[  \prod_{x\in \overline{S}} \varepsilon_{x,1}\varepsilon_{x,2}=1,    \]
By our construction of $f$, it implies that (see \cite[Proposition 2.4.2]{Yu})
\[f^V(zx)=f^V(x), \quad \forall x\in G(\AAA), z\in Z(\mathbb{F}_q). \]
In particular, we may use the identity to $z=\alpha$, and we obtain
\[ J^{1}_{(t-\alpha)^2}(f^V)=J^{1}_{(t-1)^2}(f^V). \]
Therefore,  \begin{equation}J^{1}_{geom}(f)=(q-1)\sum_{V\subseteq S_u} (-1)^{\vert  S_u-V\vert  } 2^{\vert  S_u-V\vert  }J^{1}_{unip}(f^V) ,   \end{equation}
where we denote $J^{1}_{unip}(f^V) =J^{1}_{(t-1)^2}(f^V)$ because the elements whose characteristic polynomial is $(t-1)^2$ are exactly those unipotent elements. 

We are going to pass to the Lie algebra version. 
For every $V\subseteq S_u$, we will define a function \[ \varphi^{V}=\otimes_{v\in \vert  X\vert  } \varphi_v^{V}\in C_c^\infty(\ggg(\AAA)), \]
whose support is in $\ggg(\mathcal{O})$, and that the map $x\mapsto x+1$ from nilpotent elements in $\ggg(F)$ to unipotent elements in $G(F)$ 
induces an identity
\[J^{\ggg, 1}_{nilp}(\varphi^V)= J^{1}_{unip}(f^V) .\]
Indeed, our main emphasis lies in the Fourier transform of $\varphi^V$.
We need the following proposition \ref{ft}, obtained by direct calculations except Springer's hypothesis proved by Kazhdan, to construct this $\varphi^V$. 

Let  \[K_X=\sum_{v}d_v v.  \]
We have 
\[ \mathbb{A}/(F+\prod_v\wp_v^{-d_v})\cong H^1(X, \Omega_X^{1})\cong H^{0}(X, \mathcal{O}_X)^\ast\cong \mathbb{F}_q, \]
by Serre duality and the fact that $X$ is geometrically connected. We fix a non-trivial additive
character $\psi$ of $\mathbb{F}_q$. Via the above isomorphisms, $\psi$ can be viewed as a character of $\mathbb{A}/F$. We use this $\psi$ in the definition of Fourier transform on $\ggg(\AAA)$ and we use its local component $\psi_v$ in the definition of Fourier transform on $\ggg_v$ for every $v\in |X|$.  
Let $\langle, \rangle$ be the bilinear form on $\ggg$ defined for any two $x,y\in \ggg$ by $\langle x, y\rangle:= \mathrm{Tr}(xy),$ where the product is the product of matrices. Then $\langle, \rangle$ is non-degenerate and $G$-adjoint invariant. We define the Fourier transform of any $\varphi\in {C}_{c}^{\infty}(\ggg(\AAA))$ by
\[\hat{\varphi}(x):=\int_{\ggg(\AAA)}f(y)\psi(\langle x,y\rangle)\d y.  \]
Based on the Poisson summation formula, we have an identity (\cite[Theorem 5.7]{Yu2}):
\begin{equation}\label{Jggg}J^{\ggg, e, \xi}(\varphi) = q^{4-4g} J^{\ggg, e, \xi}(\hat{\varphi}). \end{equation}

\begin{prop}(\cite[Prop. 5.3.1]{Yu}) \label{ft}
We have the following results on Fourier transformation.
\begin{enumerate}  \item
For any place $v$, the Fourier transform of the characteristic function of $\ggg(\ooo_v)$ can be calculated by the following formula:
\[\hat{\mathbbm{1}_{\ggg(\ooo_v)}}=   \mathbbm{1}_{ \wp_{v}^{-d_v}  \ggg(\ooo_v)}. \] 
\item
Let $\mathfrak{I}_v$ be the Iwahori subalgebra of $\ggg_v$ consisting of matrices in $\begin{pmatrix}
\mathcal{O}_v& \mathcal{O}_v\\
\wp_v &\mathcal{O}_v
\end{pmatrix}
$ 
and $\mathfrak{I}_{v+}$ its open subset consisting of matrices in $\begin{pmatrix}
\wp_v& \mathcal{O}_v\\
\wp_v &\wp_v
\end{pmatrix}. 
$ Then we have  
\[    \hat{\mathbbm{1}_{\mathfrak{I}_v}}=q_v^{-1}\mathbbm{1}_{\wp_{v}^{-d_v}\mathfrak{I}_{v+}}. \]

\item (Springer's hypothesis \cite{Kazhdan}\cite{KV}.) Let $U$ be a maximal torus of $G_{\kappa_v}$ and $\theta$ any character of $U(\kappa_v)$, $t$ a regular element in $\mathfrak{u}_v(\kappa_v)$ and $\rho=\epsilon R_{U}^G\theta$ the Deligne-Lusztig virtual representation of $G(\kappa_v)$ induced from $(T_v, \theta)$, where $\epsilon=1$ if $U=T$ and $\epsilon=-1$ if $U$ is not split.   Let $$e_{\rho}:=\begin{cases} \Tr({\rho}(\overline{x}^{-1})), \quad x\in \mathcal{K}_v ;\\ 0, \quad x\notin \mathcal{K}_v ; \end{cases} $$ where $\overline{x}$ denotes the image of $x$ under the map $\mathcal{K}_v \longrightarrow G(\kappa_v)$. Let $\overline{\Omega}_t\subseteq \ggg(\kappa_v)$ the $\Ad (G(\kappa_v))$-orbits of $t$ and $\Omega_t\subseteq \ggg(\ooo_v)$ be the preimage of $\overline{\Omega}_t$ of the map $\ggg(\mathcal{O}_v)\longrightarrow \ggg(\kappa_v)$.

Then for any unipotent element $u\in \mathcal{K}_v\cap G_{v, unip}$, we have $$e_{\rho}(u)= q_v^{-(4d_v+1)  } \hat{\mathbbm{1}}_{\wp_{v}^{-d_v}\Omega_{t}  }(u-1).  $$ \end{enumerate} \end{prop}
Now we come to the construction of $\varphi^V=\otimes_v \varphi_v^V$. The local components  $\varphi_v^{V}$ are defined so that with the notation of Proposition \ref{ft}, we have the follows: 
\begin{itemize} 
\item If $v\notin S_{cr}\cup V$, \[ \varphi_v^{V}=\mathbbm{1}_{\ggg(\mathcal{O}_v)} ;\] 
\item If $v\in S_{cr}$, \[\hat{\varphi_v^{V}}=q_v^{-1  } \hat{\mathbbm{1}}_{\wp_{v}^{-d_v}\Omega_{t_v}  };\]
here, $t_v\in T(\kappa_v)$ for $v\in S_r$ and $t_v\in U(\kappa_v)$ for $v\in S_c$ is a regular element; 
\item If $v\in V$,  \[\hat{\varphi_v^{V}}=\frac{1}{\vol(\mathcal{I}_v)}q_v^{-1}\mathbbm{1}_{ \wp_{v}^{-d_v} {\mathfrak{I}_{v+}}}. \]
\end{itemize}
For the Fourier transform, we have
\[\hat{\hat{f}_v}(X)=q_v^{4d_v}f_v(-X),\]
it is then direct to see that \begin{equation}\label{passL}J^{\ggg, 1}_{unip}(f^V)=J^{\ggg, 1}_{nilp}(\varphi^V).\end{equation}
We require that \[ \sum_{v\in S_{cr}} \Tr_{\kappa_v|\mathbb{F}_q}\Tr(t_v)=0. \] 
This ensures that \[ \varphi^V(z+x)=\varphi^V(x)\] for any $x\in \ggg(\AAA)$ and $z\in \zzz_G(\mathbb{F}_q)$. Note that the support of $\varphi$ is contained in $\ggg(\mathcal{O})$, we deduce by \cite[Th. 6.2.1]{Chau} that  
\[J^{\ggg, 1}(\varphi^V)=qJ^{\ggg, 1}_{nilp}(\varphi^V). \]
Combining with the trace formula for Lie algebra \eqref{Jggg} and \eqref{passL}, we have 
\begin{equation}\label{ahla}J^{1}(f^V)=(q-1)q^{3-4g-\deg V-\deg S_{cr}} ( \prod_{v\in V}\vol(\mathcal{I}_v)^{-1}) J^{\ggg, 1}(\hat{\varphi^V}). \end{equation}
The result is then deduced by a geometric interpretation of  $J^{\ggg, 1}(\hat{\varphi^V})$ using Weil's dictionary. We refer the reader to \cite[Theorem 5.2.1]{Yu} for details (here, the parameter $\xi$ is set to be zero). 
\end{proof}

\begin{theorem}\label{NoSc}
Suppose $S_c=\emptyset$. 
For $V\subseteq S_u$, let $D_{S_{r}\cup V}=K_X+\sum_{v\in S_{r}\cup V}v. $
Then \[J^{1}_{geom}(f)=\sum_{ V\subseteq S_u} (-1)^{|S_u-V|}2^{|S_u-V|} q^{-(4g-3+|\overline{R}|)}    |M^1_{V\cup S_{r}}(D_{S_{r}\cup V})(\mathbb{F}_q)|. \]
In particular, for any $o_{S_{r}}\in \mathcal{R}^{1}_{S_r}(\mathbb{F}_q)$ so that every $o_v$ ($v\in S_r$) has distinct roots, we have
\[  |M^1_{V}(o_{S_r})(\mathbb{F}_q)|=  |M^1_{V\cup S_{r}}(D_{S_{r}\cup V})(\mathbb{F}_q)|.  \]
\end{theorem}
\begin{proof}
Let $v\in S_r$ and $\chi_v$ be the character defined in Theorem \ref{mono}. 
Let 
\[ \widetilde{f}_v(x)=\begin{cases}
\vol(\mathcal{I}_v)^{-1}\chi_v(x^{-1}), \quad x\in \mathcal{I}_v; \\
0 \quad x\notin \mathcal{I}_v.
\end{cases}
\]  
Then for any $x\in G_v$, by the formula for induced character, we have 
\[\int_{\mathcal{K}_v}\widetilde{f}_v(k^{-1}xk)\d k= f_v(x). \]
Here the function $f_v$ ($v\in S_r$) is the function we defined in Proposition \ref{f}.
Therefore, \[ J_{geom}^1(f)=J_{geom}^1(\widetilde{f}). \]
Now the same proof of Theorem \ref{Higg} applies. 
\end{proof}

\subsection{Independence of parabolic weights}\label{indep}
The results in this section will only be needed only in the proof of the Theorem \ref{2}. We suggest that the reader reads this part only when referred to.

We come back to assumptions in Section \ref{semistabler} that $V\subseteq S_{u}$, $R=V\cup S_{cr}$ and $D_R=K_X+\sum_{v\in R}v$.
We consider the case that the parameter $\xi=(\xi_{v})_{v\in R}$ of parabolic weights is not necessarily trivial. Recall that we use $\overline{\xi}=(\overline{\xi}_{x})_{x\in \overline{R}}$ to denote the family such that $\overline{\xi}_x=\frac{1}{[\kappa_v:\mathbb{F}_q]}\xi_v$ for any point $x\in \overline{R}$ lying over $v\in R$.

\begin{theorem}\label{weights}
Suppose that $e$ is an odd integer or there exists a place $v\in R$ such that $\deg v$ is odd. 
Suppose that $\mathbb{F}_q\neq \mathbb{F}_2$. 
Varying the parabolic weights $(e,\overline{\xi})$ but remaining  in general position and  \[ \xi_{v,1}\geq \xi_{v,2}\geq \xi_{v,1}-\deg v\]  for any $v\in R$, the cardinality of the set  
\( M^{e, \xi}_R(D_R)(\mathbb{F}_{q})  \) remains the same. 
 \end{theorem}\begin{proof}
We will need the following lemma, which is a slight variant of \cite[Theorem 3.4.3]{Yu}. 
\begin{lemm}[\cite{Yu}]\label{vanishing}
Suppose that $f\in C_c^\infty(\ggg(\AAA))$. Suppose that the support of $f$ is contained in $\ggg(\mathcal{O}^R)\prod_{v\in R}\mathfrak{I}_v$ (where $\mathcal{O}^R=\prod_{v\notin R}\mathcal{O}_v$) and $(e,\xi)$ is in general position in the sense that
\[e+\sum_{v\in R}\epsilon_{v}(\xi_{v,1}-\xi_{v,2})\notin 2\mathbb{Z},\] for any $(\epsilon_{v})_{v\in R}\in \{1, -1\}^{R}$.
If moreover, $e$ is odd or $\deg v$ is odd for at least one $v\in R$, we have
$$J^{\ggg, e, \xi}(f)= \sum_{a\in \mathbb{F}_q}J_{(t-a)^2}^{\ggg, e, \xi}(f). $$
\end{lemm}
\begin{proof}
The situation is slightly different than \cite[Theorem 3.4.3]{Yu} where we set $\deg v=1$ for every $v\in R$.  
The same proof implies that either 
\begin{equation}\label{vanish}J^{\ggg,e,\xi}_\chi (f)=0, \end{equation} 
or $\chi$ is the characteristic polynomial of a semisimple element $\sigma\in \ggg(\mathbb{F}_q)$ such that   $Z_{G_\sigma}/Z_G$ is anisotropic, where $G_\sigma$ is the adjoint centralizer of $\sigma$. The last case happens either $\sigma$ is a central element or the characteristic polynomial of $\sigma$ is irreducible. When $e$ is odd and $\sigma\in \ggg(\mathbb{F}_q)$ has an irreducible characteristic polynomial, there is no $x\in G(\AAA)^{e}$ such that $\Ad(x^{-1})(\sigma)\in \ggg(\mathcal{O})$ since  $ G_\sigma(\AAA)G(\mathcal{O})\cap G(\AAA)^{e}=\emptyset$. If there exists some $v\in R$ of odd degree, then the characteristic polynomial of $\sigma$ can not be irreducible in $\mathbb{F}_q[t]$. In fact, $\sigma$ is conjugate to an element in $B(\mathbb{F}_{q_v})$, hence its characteristic polynomial splits in $\mathbb{F}_{q_v}[t]$ hence it must be split in $\mathbb{F}_q[t]$ as it is of degree $2$.
\end{proof}

Recall that we have fixed a divisor \[ D _R=K_X+\sum_{v\in R}v=\sum_v n_v v\] on the curve $X$.   
Let $\mathbbm{1}_{R}$ be the function defined to be the tensor product\[ \bigotimes_{v\notin R}\mathbbm{1}_{\wp^{-n_v}\ggg(\mathcal{O}_v)}\otimes \bigotimes_{v\in R}( \frac{1}{\vol(\mathcal{I}_v)} \mathbbm{1}_{\wp_v^{-n_v}\mathfrak{I}_{v+}}),\] where $\mathfrak{I}_{v+}$ consists of elements in $\ggg(\mathcal{O}_v)$ whose reduction modd-$\wp_v$ belongs to $\nnn(\kappa_v)$. We have shown in \cite[Theorem 5.2.1]{Yu} that if $\xi_{v, 1}\geq \xi_{v,2}\geq \xi_{v,1} - [\kappa_{v}: \mathbb{F}_q]$   and $(e,\xi)$ is in general position (hence semistability coincides with stability), we have 
\begin{equation}J^{\ggg, e, \xi}(\mathbbm{1}_{R}) = \frac{1}{q-1} \vert    M^{e, \xi}_R(D_R)(\mathbb{F}_q) \vert  . \end{equation}
The factor $\frac{1}{q-1}$ comes from the fact that there are $q-1$ automorphisms for stable parabolic Hitchin bundles.

We need a lemma first. \begin{lemm}\label{Traces'} Suppose that $\mathbb{F}_{q}\neq \mathbb{F}_2$.  There exists a family of characters $(\chi_v)_{v\in R}$, $\chi_v: T(\kappa_v)\rightarrow \mathbb{C}^\times$, so that \[ \prod_{v\in R}  \chi_v|_{Z(\mathbb{F}_q)} =1,  \] and  the following properties are satisfied. 
Let $\rho$ be the representation of $\prod_v \mathcal{I}_v$ obtained by the tensor product of the representations $\mathcal{I}_v\longrightarrow T_v(\kappa_v)\xrightarrow{\chi_v} \mathbb{C}^\times$. 
For any automorphic representation $\pi$ of $G(\AAA)$, if $\pi$ contains $\rho$, i.e., the $\rho$-isotypic part $\pi_\rho\neq 0$, then $\pi$ is cuspidal. Moreover, for any such $\pi$ and any character $\lambda$ of $G(\AAA)$ that factors through $\deg\circ \det$, we have \[  \pi\otimes \lambda\cong \pi \implies \lambda=1.  \] \end{lemm}
\begin{proof}
Let $v_0\in R$. Let $\chi_1$ be a primitive character of $\kappa_{v_0}^{\times}\longrightarrow \mathbb{C}^\times$, i.e., an injective homomorphism.
Since $\mathbb{F}_q\neq \mathbb{F}_{2}$,the character $ \chi_{1}$ is non-trivial on $\mathbb{F}_{q}^\times$. We set $\chi_v=1$ for any $v\neq v_0$, and $\chi_{v_0}=(\chi_1, \chi_1^{-1})$. 

We prove that the family $(\chi_v)_{v\in R}$ satisfies all the properties we need. 

We have clearly \[ \prod_{v\in R}  \chi_v|_{\mathbb{F}_q^\times} =1.   \]

Given an automorphic representation $\pi$, it is either cuspidal, or it is a sub-quotient of a parabolic induction $\Ind_{B(\AAA)}^{G(\AAA)}\mu$, for a Hecke character $\mu=(\mu_1, \mu_2)$ of $T(\AAA)/T(F)$. The latter case is impossible. In fact, the condition $\pi_{\rho}\neq 0$ implies  $(\Ind_{B(\AAA)}^{G(\AAA)}\mu)_\rho\neq 0$. This, in turn, implies that $\mu$ is unramified outside $\{v_0\}$ and   \[ \Hom_{T(\mathcal{O}_{v_0})}(\chi_{v_0}, \mu|_{T(\mathcal{O}_{v_0})})\neq 0 \text{, or }  \Hom_{T(\mathcal{O}_{v_0})}(\chi_{v_0}^{w}, \mu|_{T(\mathcal{O}_{v_0})})\neq 0, \]  here we review $\chi_{v_0}$ as a character of $T(\mathcal{O}_{v_0})$, and $\chi_{v_0}^{w}=(\chi^{-1}_{1}, \chi_{1})$. In particular, \[\mu_1|_{\mathbb{F}_{q}^\times}=\chi_1|_{\mathbb{F}_{q}^\times} \text{, or } \mu_1|_{\mathbb{F}_{q}^\times}=\chi_1^{-1}|_{\mathbb{F}_{q}^\times}.  \]  This implies that $\mu_1|_{\mathbb{F}_{q}^\times}\neq 1$. It contradicts the fact that $\mu_1$ is a Hecke character: it must be trivial on $F^\times$.

For the last assertion of the lemma, we use Langlands correspondence to prove it.  
Suppose that $(\sigma, i: \sigma\xrightarrow{\sim} F_{X/\mathbb{F}_q}^{\ast}\sigma$) is the Weil sheaf that corresponds to the cuspidal automorphic representation $\pi$. 
The sheaf $\sigma$ has a rank 2 and is smooth over $(X-\{v_0\})\otimes \overline{\mathbb{F}}_q$. The local monodromies of $\sigma$ over punctured discs $\overline{X}_{x}^{(\ast)}$ (defined in the Introduction) centered at points $x$ in $\{v_0\}\otimes \overline{\mathbb{F}}_q$ are semisimple tame local systems so that a tame generator has as eigenvalues $(\zeta^{q^i}, \zeta^{-q^i} )_{i=1, 2, \ldots, d_{v_0}}$, where $\zeta$ is a $(q^{d_{v_0}}-1)^{th}$ primitive root of unity. If the assertion is not correct, then $\sigma=\sigma_1\oplus \sigma_2$ and the Frobenius action $F_{X/\mathbb{F}_q}^{\ast}$ exchanges isomorphism classes of  $\sigma_1$ and $\sigma_2$. Hence, $\sigma_1$ is fixed by $F_{X/\mathbb{F}_q}^{\ast 2}$. Note that the ramifications of $\sigma_1$ at points in $\{v_0\}\otimes \overline{\mathbb{F}}_q$ must be multiplication by $(\zeta^{\epsilon_i q^i})_{i=1, 2, \ldots, d_{v_0}}$, where $\epsilon_i\in \{1, -1\}$. If $d_{v_0}$ is odd, then $\{v_0\}\otimes \overline{\mathbb{F}}_q$ is cyclically permuted by $F_{X/\mathbb{F}_q}^{\ast 2}$, hence $\epsilon_i$ has the same sign and their product is   \[  \zeta^{\pm(q^{d_{v_0}}-1/q-1)}\neq 1,  \] which is impossible. If $d_{v_0}$ is even, then $\epsilon_{2i}$ (resp. $\epsilon_{2i+1}$) have the same sign. This is also not possible because the product of  eigenvalues of ramifications of $\sigma_1$ is one of the following  \[\zeta^{(q^{d_{v_0}}-1/q-1)}, \zeta^{-(q^{d_{v_0}}-1/q-1)}, \zeta^{(q^{d_{v_0}}-1/q+1)}, \zeta^{-(q^{d_{v_0}}-1/q+1)}.  \] This is again not possible because none of them is $1$.   \end{proof}

We choose a family of characters $(\chi_v)_{v\in R}$ as in the Lemma \ref{Traces'}.  Let $h$ be the function \[ \mathbbm{1}_{G(\mathcal{O}^R)}\otimes \bigotimes_{v\in R}(\frac{1}{\vol(\mathcal{I}_v)}\mathbbm{1}_{\mathcal{I}_v} \chi_v).  \]  Since the support of $h$ is contained in $G(\mathcal{O}^R)\prod_{v}\mathcal{I}_v$,  we have \[ J^{G, e, \xi}_{\chi}(h)=0, \] except if $\chi=(t-a)^2\in \mathcal{E}_G$. Since  \[ \prod_{v\in R}  \chi_v|_{Z(\mathbb{F}_q)} =1,  \] we deduce, similar to the Lie algebra case that \[J^{G, e, \xi}(h)=(q-1)J^{G, e, \xi}_{unip}(h),  \] where $J^{G, e, \xi}_{unip}(h)=J^{G, e, \xi}_{(t-1)^2}(h)$. It is direct to see that the map $X\mapsto 1+X$ from the set of nilpotent elements in $\ggg(F)$ to that of unipotent elements gives us an identity:\[  J^{G, e, \xi}_{unip}(h) = q^{-deg R}  J^{\ggg, e, \xi}_{nilp}(\hat{\mathbbm{1}_{R}}).  \] Therefore we have \begin{equation}\label{LieGr}  J^{G,e,\xi}(h)=q^{-\deg R+3-4g} |M^{e,\xi}_{R}(D_R)(\mathbb{F}_q)|.  \end{equation}

Recall that $a\in \AAA^\times$ is a degree $1$ idèle, viewed as a scalar matrix.  By Lemma \ref{Traces'}, we know that the regular action  $R(h)$ on $L^2(G(F)\backslash G(\AAA)/a^\mathbb{Z})$ is a projection whose image lies inside the space of cuspidal automorphic forms and the regular action $R(h)$ on $L^2(T(F)N(\AAA)\backslash G(\AAA)/a^\mathbb{Z})$ is zero. It shows that for any $x\in G(\AAA)$,  \[0=\sum_{\gamma\in {T}(F)} \sum_{i\in\mathbb{Z}}  \int_{N(\AAA)} h((\delta x)^{-1} (\gamma n)\delta a^i x)\d n\] and hence  \begin{align*} J^{G, e, \xi}(h) &=\frac{1}{2}( \Tr(R(h)| L^2_{cusp}(G(F)\backslash G(\AAA)/a^\mathbb{Z})) +(-1)^{e} \Tr(R(h) \epsilon | L^2_{cusp}(G(F)\backslash G(\AAA)/a^\mathbb{Z}))  )\\ & = \frac{1}{2}( \Tr(R(h)| L^2_{cusp}(G(F)\backslash G(\AAA)/a^\mathbb{Z})) .  \end{align*} which is independent of $(e,\xi)$. Recall that $\epsilon$ is the sign character on $G(\AAA)$ that factors through $\deg\circ\det$. By \eqref{LieGr}, this finishes the proof. \end{proof}

\section{Proof of the main theorems} 
We will prove our main result Theorem \ref{1}. 
The main ingredient of the proof is Theorem \ref{automorphe} and Theorem \ref{Higg}. These two results give an expression for $|E_2(\mathfrak{R})^{\Fr^\ast}|$. It remains to apply this result to the curve $X\otimes\mathbb{F}_{q^k}$ over $\mathbb{F}_{q^k}$ and study how $|E_2(\mathfrak{R})^{\Fr^\ast}|$ varies for $k\in \mathbb{N}^\ast$.

\subsection{Functions of Lefschetz type}
We continue to use our notation in the introduction. Let us prove first the following proposition.

\begin{prop}\label{Lefschetz}
Let $A$ be a set with a permutation $\tau$ acting on it. We use $O(\tau|A)$ to denote the number of orbits of $\tau$ acting on $A$. It will be applied to the case that $A=\overline{S}_u$ with Frobenius element acting on it. 
\begin{enumerate}
\item
Let $k \mapsto \alpha_A(k)$ be the function that $\alpha_A(k)=|A|$ if $\tau^{k}$ is a cyclic permutation on $A$ and $\alpha_A(k)=0$ otherwise. It is of Lefschetz type.


\item 
 
We define $\beta_{A}(k)$ by 
$\beta_{A}(k)=0$ if 
$\tau^{k}$ has an orbit of even length,  $\beta_{A}(k)=2^{O(\tau^k|A)-1}$ if all orbits are of odd length. It is of Lefschetz type. 
\item 
We define $\gamma_{A}(k)$ by 
\[ \gamma_A(k)= (-1)^{O(\tau^k|A)}.\]  
It is of Lefschetz type. 

\item
We define $\omega_A=\frac{1}{2}(\alpha_A+(-1)^{|A|}\beta_A)$. It is of Lefschetz type.

\item If $S_{cr} \neq \emptyset$, then $c_{\mathfrak{R}}/2+b_{\mathfrak{R}}/2$ is of Lefscehtz type. 

\item The function $b_{\mathfrak{R}}$ is of Lefschetz type.

\item If $S_{cr} \neq \emptyset$, then the functions \[k\mapsto {c_{\mathfrak{R}}(k)\alpha_{\overline{S}_u}({k})}/{2} + {b_{\mathfrak{R}}(k)\beta_{\overline{S}_u}(k)  }/{2}\] are of Lefschetz type. 

\end{enumerate}
\end{prop}
\begin{proof}
(1) 
If $\tau$ is not a cyclic permutation on $A$, then neither is $\tau^k$ for every $k\geq 1$. In this case, $\alpha_A$ is constantly  $0$, and the assertion is trivial. 

We suppose in the following that $\tau$ is a cyclic permutation.  

Let $|A|=p_1^{a_1}\cdots p_s^{a_s}$ be a prime decomposition of $n$ with $p_i$ being different prime numbers and $a_i>0$. Let $\zeta_{p_i}$ be a primitive $p_i$-th roots of unity.
 Then \[ \alpha_A(k)=|A|,\] if $p_i\nmid k$ for all $p_i$,  otherwise 
 \[\alpha_A(k)=0.\] 
 It is direct to verify that we have the following identity:
\begin{equation}\label{alpha} \alpha_A(k)=\prod_{i}(p^{a_i}_i-p_i^{a_i-1}\sum_{j=1}^{p_i}\zeta_{p_i}^{jk}). \end{equation}
Since roots of unity are $q$-Weil integers of weight $0$, the statement follows. 

(2)  
Let $n$ be an odd integer. We first prove the following assertion by induction on the number of prime divisors of $n$ (counting multiplicities).  

$(\ast)$ For any odd integer $m$ such that $(m,n)=1$, the function \[k\mapsto \frac{1}{m}(2^{\phi(m)(n,k)-1}-1)    \]
is a function of Lefschetz type.  

We only need the case that $m=1$ of the assertion $(\ast)$, but this stronger assertion is easier to prove by induction.  
The case that $n=1$ is trivial since the function is constant in $k$, and it is an integer by Fermat's little theorem.

Let $l$ be any prime number, not dividing $nm$. Let $\beta\in \mathbb{N}$. Suppose by induction that the assertion $(\ast)$ holds for $n$, $nl$, $\ldots$, $nl^{\beta-1}$. 
Note that we have the following identity
 \begin{align*}\frac{1}{m}(2^{\phi(m)(nl^\beta,k)-1}-1)- \frac{1}{m}(2^{\phi(m)(n,k)-1}-1)  &=  \sum_{j=1}^{\beta}\frac{1}{m}(2^{\phi(m)(nl^j,k)-1}-2^{\phi(m)(nl^{j-1},k)-1})  \\
 &=\sum_{j=1}^{\beta}\frac{  2^{\phi(m)(nl^{j-1},k)-1} }{m} (2^{\phi(m)((nl^j,k)-(nl^{j-1},k))  }   -1)\\
 &= \sum_{j=1}^{\beta}\frac{  2^{\phi(m)(nl^{j-1},k)-1} }{ml^j} (2^{  \phi(m l^j)(n,k) }-1)\sum_{s=1}^{l^j}\zeta_{l^j}^{sk}. 
  \end{align*} 
  In the last equality, we have used the fact if $l^j\nmid k$, then $(nl^j,k)-(nl^{j-1},k)=0$ and if $l^j\mid k$, then $(nl^j,k)-(nl^{j-1},k)=(n,k)\phi(l^j)$. We have also used the fact that $\phi(m)\phi(l^j)=\phi(ml^j)$ since $l\nmid m$. 
  Since the product of Lefschetz type functions and integral multiple of Lefschetz type functions are of Lefschetz type, we deduce that the assertion $(\ast)$ is correct for $nl^\beta$ as well. By induction, we obtain the result needed.

To prove that $\beta_A$ is of Lefschetz type, it is sufficient to prove that $\tau$ is a cyclic permutation on $A$ by multiplicativity on orbits and the fact that the integral multiple of the function of Lefschetz type is again of Lefschetz type. Suppose that $|A|=2^al$ with $l$ being an odd integer. 
If $a=0$, then $\beta_A$ is the function
\[k\mapsto 2^{(|A|,k)-1}  \]
which is of Lefschetz type by the above result (by setting $m=1$). Suppose $a\geq 1$. 
In this case, we have 
\[\beta_A(k) =\begin{cases}
2^{2^a(l,k)-1}, \quad 2^a\mid k;\\
0, \quad 2^a \nmid k. 
\end{cases}
 \]
 Since $a\leq 2^a-1$, the function $\beta_A$ is of Lefschetz type, because we have
\begin{equation}\label{beta} \beta_A(k)= 2^{2^a-a-1} (2^{(l,k)-1})^{2^a}  \sum_{j=1}^{2^a} \zeta_{2^a}^k . \end{equation}
We have shown that $k\mapsto 2^{(l,k)-1}$ is of Lefschetz type, therefore so is $\beta_A$. 
 
(3) Since the product of Lefschetz type function is still of Lefschetz type, it suffices to consider the case that $\tau$ is a cyclic permutation of $A$. If $A$ has odd cardinality, then $(-1)^{O(\tau^k|A)}=-1$. Suppose that $|A|=2^am$ with $m$ being $m$ odd and $a\geq 1$. Then $O(\tau^k|A)=(2^am,k)$, and  \[(-1)^{O(\tau^k|A)}=\begin{cases} -1,\quad 2\nmid  k;   \\ 1, \quad 2\mid k.  \end{cases}   \] Therefore  in this case  \[(-1)^{O(\tau^k|A)} = (-1)^k.    \]
This is a function of Lefschetz type, and we have proved the assertion.

(4) 
Let us consider $\omega_A$. If $\tau$ is not a cyclic permutation on $A$, then $\alpha_A=0$. Let $A=A_1\cup A_2$ be a partition of $A$ into non-empty $\tau$-stable subsets. Then \[\omega_A=(-1)^{|A|}\beta_{A_1}\beta_{A_2}. \] 
Therefore $\omega_A$ is of Lefschetz type. 

In the following, we suppose that $\tau$ is a cyclic permutation. It is sufficient to consider $\frac{\alpha_A-\beta_A}{2}$.

Let $f: \mathbb{N}^{\ast}\longrightarrow \mathbb{Z}$ be a periodic function of period $n$. Then we have
\[f(k)=\sum_{i=1}^{n}\frac{\sum_{j=1}^{n}f(j)\zeta_n^{-ij}}{n} \zeta_n^{ki}.      \]
Therefore, $f$ is of Lefschetz type if and only if \begin{equation}\label{integer}\frac{\sum_{j=1}^{n}f(j)\zeta_n^{-ij}}{n}\in \mathbb{Z} \end{equation}
for $i=1, \ldots, n$. We will use this criterion to prove that $\omega_A$ is of Lefschetz type.

If $n:=|A|$ is an odd integer. 
By \eqref{integer} and the fact that $\beta_A$ is of Lefschetz type, we  know that for any $i$, the number
\[{\sum_{j=1}^{n}2^{(n,j)-1}\zeta_n^{-ij}}\]  is an integer that is divisible by \(n \). 
Moreover, \[{\sum_{j=1}^{n}\alpha_{A}(k)\zeta_n^{-ij}} = n c_{n}(i),  \]
where \(c_n(i) \) is the sum $i^{th}$ power of primitive  $n^{th}$ roots of unity, i.e. the Ramanujan's function. The function $i \mapsto c_n(i)$ takes an integral value (since $\alpha_A$ is of Lefschetz type). 
We need to prove that the number 
\begin{equation}\label{OMEGA} 2{\sum_{j=1}^{n}\omega_A(j)\zeta_n^{-ij}}=  -{\sum_{j=1}^{n}2^{(n,j)-1}\zeta_n^{-ij}}+n c_{n}(i) \end{equation}
is divisible by $2n$. 
Since we're in the case that $n$ is an odd number, and \eqref{OMEGA} is divisible by $n$, we need to show that it is divisible by $2$. 
Note that we have
\[{\sum_{j=1}^{n}2^{(n,j)-1}\zeta_n^{-ij}} = \sum_{d\mid n} 2^{d-1}c_{n/d}(i) .  \]
Modulo $2$, the expression \eqref{OMEGA} equals $(n-1)c_{n}(i)$. Since $n$ is odd, this is $0$ modulo $2$. We're done.

Now suppose that $n=|A|$ is an even integer. 
If $4\mid n$, then clearly both $ \frac{1}{2}\alpha_A(k)$ and $\frac{1}{2}\beta_A(k)$ are of Lefschetz type. We can see it from equations \eqref{alpha} and \eqref{beta}.

If $4\nmid n$, we need to prove that for any $i$
\[2{\sum_{j=1}^{n}\omega_A(j)\zeta_n^{-ij}}= nc_n(i) - \sum_{j=1}^{n/2}2^{(n, 2j) -1} \zeta_{n}^{-2ij} \]
is divisible by $2n$. As we have shown before, it is divisible by $n$. Therefore, it remains to show that it is divisible by $4$. 
Note that \[\sum_{j=1}^{n/2}2^{(n, 2j) -1} \zeta_{n}^{-2ij} = \sum_{d\mid \frac{n}{2}}2^{2d-1} c_{n/2d}(i)  .  \]
Modulo $4$, we need to prove that $2c_n(i)-2c_{n/2}(i)$ is divisible by $4$. By Möbius inversion formula, we have
\[c_n(i)=\sum_{d\mid (n,i)} \mu(n/d)d,    \]
where $\mu$ is the Möbius function.
Note that if the divisor $d$ of $n$ is odd, we have \[\mu(n/d)=-\mu(n/2d). \] 
Therefore, if $i$ is odd, then \[c_n(i)=-c_{n/2}(i). \] If $i$ is even, we have
\[c_n(i)-c_{n/2}(i) =\sum_{d\mid (n,i)} \mu(n/d)d-\sum_{d\mid (n/2,i)} \mu(n/2d)d.  \]
The sum over $d\mid (n, i)$ can be decomposed into two parts following $d$ is odd or $d$ is even. 
We deduce that for $i$ being even, 
\[\begin{split}c_n(i) &=  \sum_{d\mid (n/2, i)} \mu(n/d)d+\sum_{d\mid (n/2, i)} \mu(n/2d)2d \\&=-\sum_{d\mid (n/2, i)} \mu(n/2d)d+\sum_{d\mid (n/2, i)} \mu(n/2d)2d \\
&= c_{n/2}(i) .   
\end{split} \]
We conclude that in either case, the number $2c_n(i)-2c_{n/2}(i)$ is divisible by $4$.

(5) Let $P_{\mathfrak{R}}/\mathfrak{S}_2$ be the quotient of $P_{\mathfrak{R}}$ by the action of $\mathfrak{S}_2=\{1, \sigma\}$. Since $\overline{S}_{cr}$ is non-empty, every point in $P_{\mathfrak{R}}/\mathfrak{S}_2$ has a preimage of cardinality $2$ in $P_{\mathfrak{R}}$. 
The action of $\Fr^{\ast}$ defines an action of $P_{\mathfrak{R}}/\mathfrak{S}_2$ since it commutes with $\sigma$. For any $e=(a, \sigma(a))\in P_{\mathfrak{R}}/\mathfrak{S}_2$, if 
\[ \Fr^{\ast k}(e)=e, \]
then we have either $\Fr^{\ast k}(a)=a$ or we have $\Fr^{\ast k}(a)=\sigma(a)$. Therefore,
\[c_{\mathfrak{R}}(k)+b_{\mathfrak{R}}(k) = 2 |(P_{\mathfrak{R}}/\mathfrak{S}_2)^{\Fr^{\ast k}}|, \quad \forall k\geq 1.   \]
This proves the result.

(6) If $\overline{S}_{cr}=\emptyset$, then $b_{\mathfrak{R}}(k)=|P_{\mathfrak{R}}|$ is either $1$ or $0$. If $\overline{S}_{cr}\neq\emptyset$, then it follows from (5), because $b_{\mathfrak{R}}=2\frac{c_\mathfrak{R}+b_\mathfrak{R}}{2}-c_\mathfrak{R}$. 

(7) The function under consideration equals
\[\frac{c_{\mathfrak{R}}+b_{\mathfrak{R}}}{2}\alpha_{\overline{S}_u}  -  b_{\mathfrak{R}} \frac{ \alpha_{\overline{S}_u} -\beta_{\overline{S}_u}   }{2} .     \]
It results from (4), (5), and (6) that this is a function of Lefschetz type.   
  \end{proof}
  
\subsection{$\mathrm{Higg}_\mathfrak{R}$ is of Lefschetz type. }


\begin{theorem}\label{Lefs}
Let $o=(o_v)_{v\in S_{cr}}\in \mathcal{R}^{1}_{S_{cr}}({\mathbb{F}}_{q})$ so that every polynomial $o_v$ has distinct roots and is split over $\kappa_v$ if $v\in S_r$, and is irreducible if $v\in S_c$. 
\begin{enumerate}
\item
The function over $\mathbb{N}^\ast$: 
\[\mathrm{Higg}_{\mathfrak{R}}(k)=
\sum_{V\subseteq {S}_u\otimes\mathbb{F}_{q^k}} (-1)^{|{S}_u\otimes\mathbb{F}_{q^k} - V|} 2^{|V|}  q^{-k(4g-3+|\overline{V}|+|\overline{S}_{cr}|)}    |{M}_{V}^{1}(o)(\mathbb{F}_{q^k})|.
\] is  of Lefschetz type in $k$. 
\item The number $\mathrm{Higg}_{\mathfrak{R}}(k)$ is divisible by $\Pic(k)$ and the quotient function
\[ k\mapsto \mathrm{Higg}_{\mathfrak{R}}(k)/\Pic(k)\] is still of Lefschetz type. 
\end{enumerate}
 \end{theorem}
 
\begin{proof} (1)
Let $F_q:  a\mapsto a^{1/q}$ be the  geometric Frobenius element in $\Gal(\overline{\mathbb{F}}_q/\mathbb{F}_q)$. 
We use the notation \[ U^{(1/q)}:=U\times_{\Spec(\mathbb{F}_{q}), F_q}\Spec(\overline{\mathbb{F}}_{q}). \]

Let us prove a lemma first. 
\begin{lemm} \label{frobedt}
Let $V\subseteq S_{u}\otimes\mathbb{F}_{q^k}$.  
We know that\[   {M}_{\Fr(\overline{V})}^{1}(o)\cong {M}_{\overline{V}}^{1}(o)^{(1/q)}.  \] 
In particular, the relative Frobenius morphism induces a linear map: 
\[F_V^\ast: H_c^{\ast}({M}_{\overline{V}}^{1}(o), \mathbb{Q}_\ell)\longrightarrow H_c^{\ast}({M}_{\Fr(\overline{V})}^{1}(o), \mathbb{Q}_\ell).\]
Let $d$ be the least positive integer such that $\Fr^{d}(V)=V$. Then the $d$-times composition of $F_V^\ast$ coincides with the action of geometric Frobenius element $F_{q^{d}}\in \Gal(\overline{\mathbb{F}}_{q}/\mathbb{F}_{q^{d}})$ (recall that ${M}_{\overline{V}}^{1}(o)$ has an $\mathbb{F}_{q^d}$-structure). 

\end{lemm}
\begin{proof}

Let $R=(S_{cr}\otimes\mathbb{F}_{q^{d}})\cup V$, 
and $\overline{R}=\overline{S}_{cr}\cup \overline{V}$, where $\overline{V}=V\otimes_{\mathbb{F}_{q^{d}}}\overline{\mathbb{F}}_q$. 

Note that the functor $(\cdot)^{(q)}$ induces an equivalence of categories from  $(Sch/{\mathbb{F}}_{q^d})$,  the category of schemes over ${\mathbb{F}}_{q^d}$, to itself. Its $d$-fold iterate is the identity functor. Suppose that $(\mathcal{E}, \varphi, (L_x)_{x\in \overline{V}})$ is a parabolic Hitchin bundle over $\overline{X}$, then $\mathcal{E}^{(1/q)}$ is again a vector bundle over $\overline{X}$, but the parabolic structures are imposed at points in $\Fr(\overline{V})$. The map $\mathcal{E}\mapsto \mathcal{E}^{(1/q)}$ is a bijection between 
${M}_{\Fr(\overline{V})}^{1}(\Fr(D_{\overline{R}}))(\overline{\mathbb{F}}_q)$ and ${M}_{\overline{V}}^{1}(D_{\overline{R}})^{(1/q)}(\overline{\mathbb{F}}_q)$. 
With this bijection $\theta$, we can apply \cite[Th. 4.6]{Yokogawa0}. 
In fact, let $T$ be a scheme defined over $\overline{\mathbb{F}}_q$ and $(\mathcal{E}, \varphi, (L_x)_{x\in \Fr(\overline{V})})_T$ be a flat family of parabolic Hitchin bundle over $\overline{X}_{T}$ with parabolic structures in $\Fr(\overline{V})$. The associated parabolic Hitchin bundle $(\mathcal{E}, \varphi, (L_x)_{x\in \Fr(\overline{V})})_T^{(q)}$ is a flat family over $\overline{X}_{T^{(q)}}$ with parabolic structures in $\overline{V}$. By  \cite[Th. 4.6, (4.6.4)]{Yokogawa0}, we obtain a morphism $T^{(q)}\rightarrow {{M}_{\overline{V}}^{1}}(D_{\overline{R}})$. Applying the functor $(\cdot)^{(1/q)}$, we obtain $T\rightarrow {M}_{\overline{V}}^{1}(D_{\overline{R}})^{(1/q)}$. Then \cite[Th. 4.6, (4.6.5)]{Yokogawa0} implies that ${M}_{\Fr(\overline{V})}^{1}(\Fr(D_{\overline{R}}))\cong {M}_{\overline{V}}^{1}(D_{\overline{R}})^{(1/q)}$.

It is clear that $(\mathcal{R}_{\overline{R}}^{1})^{(1/q)}\cong \mathcal{R}_{\Fr(\overline{R})}^{1}$. We have the relative Frobenius morphism:
\[(\mathcal{R}_{\overline{R}}^{1})^{(1/q)}\longrightarrow (\mathcal{R}_{\overline{R}}^{1}). \]
The scheme $\mathcal{R}^{1}_{S_{cr}}$ is defined over $\mathbb{F}_q$ and the relative Frobenius induces an identity on its $\mathbb{F}_q$-points. 
Since $o\in \mathcal{R}^{1}_{S_{cr}}({\mathbb{F}_q})$, its embedding in $\mathcal{R}_{\overline{R}}^{1}(\overline{\mathbb{F}}_{q})$ is sent to $o$ itself via the relative Frobenius morphism \((\mathcal{R}_{\overline{R}}^{1})^{(1/q)}\longrightarrow (\mathcal{R}_{\overline{R}}^{1})\). These imply the first assertion. 

We still denote the induced relative Frobenius morphism by $F_V$:
\[F_{ V}: M_{\Fr(\overline{V})}^{1}(o) \longrightarrow  M_{\overline{V}}^{1}(o). \]
Note that by definition, the composition 
\[ F_V \circ F_{\Fr(V)}\circ \cdots \circ F_{\Fr^{d-1}(V)}  \]
coincides with the Frobenius endomorphism, i.e. the base change to $\overline{\mathbb{F}}_q$ of the $q^{d}$-Frobenius morphism of $M_{V}^{1}(o)$. On étale cohomology, its action coincides with the geometric Frobenius element $F_{q^{d}}$ of the Galois group $\Gal(\overline{\mathbb{F}}_q/\mathbb{F}_{q^{d}})$. 
The last assertion hence follows. 
\end{proof}

Let us choose a total order on $\overline{S}_u$.   Note that $\Fr$ acts on  $\overline{S}_{u}$. 
 For any $V\subseteq \overline{S}_{u}$, let 
\[ {inv}(\Fr|V)\] be the inversion number of $\Fr$ on $V$, i.e., the number of pairs $(x_1, x_2)$ of points in $V$ such that  $x_1<x_2$ and $\Fr(x_1)>\Fr(x_2)$.

For any subset $V$ of $\overline{S}_u$, let us consider $H^{\ast}(\mathbb{P}^1, \mathbb{Q}_\ell)^{\otimes V}$,
where the tensor product is understood as the tensor product of graded vector space.
Let $\alpha$ be a generator in $H^{0}(\mathbb{P}^1, \mathbb{Q}_\ell)$ and $\beta$ be a generator in $H^{2}(\mathbb{P}^1, \mathbb{Q}_\ell)$. 

We set \[\tau: H^{\ast}(\mathbb{P}^1, \mathbb{Q}_\ell)^{\otimes V}\longrightarrow H^{\ast}(\mathbb{P}^1, \mathbb{Q}_\ell)^{\otimes \Fr(V)}\] as the map of graded vector spaces that sends an element in $H^{\ast}(\mathbb{P}^1, \mathbb{Q}_\ell)^{\otimes V}$ represented by $(a_x\alpha_x+b_x\beta)_{x\in V}$ to  $(a_{x}\alpha_{\Fr(x)}+b_{x}\beta_{\Fr(x)})_{x\in \Fr(V)}$:
\[ \tau(  (a_x\alpha_x+b_x\beta_x)_{x\in V}    ) = (a_{x}\alpha_{\Fr(x)}+b_{x}\beta_{\Fr(x)})_{x\in \Fr(V)}.\]

Let \[ H^\ast_V:= H_c^{\ast}({M}_{\overline{V}}^{1}(o), \mathbb{Q}_\ell)\otimes H^{\ast}(\mathbb{P}^1, \mathbb{Q}_\ell)^{\otimes V}. \]
Let $\varsigma$ be a linear endomorphism: 
\[ \varsigma: \bigoplus_{V\subseteq \overline{S}_u} H_{V}^i\longrightarrow \bigoplus_{V\subseteq \overline{S}_u} H_{V}^i, \] 
defined by \[ \varsigma=\oplus_{V}(-1)^{inv(\Fr|V)}q^{3-4g-|V|-|\overline{S}_{cr}| }F_V^\ast\otimes \tau. \]

We will show that:
\begin{enumerate}
\item[($\ast$)] the eigenvalues of $\varsigma$ are $q$-Weil integers. 
\item[($\ast\ast$)] 
\[ \mathrm{Higg}_{\mathfrak{R}}(k) =  \sum_{i}(-1)^i\Tr(\varsigma^{k}| \bigoplus_{V\subseteq \overline{S}_u} H_{V}^i ).  \]
\end{enumerate}
These two properties suffice to prove the theorem.

Let us prove $(\ast)$. It is sufficient to prove that for $k$ divisible enough, the eigenvalues of $\varsigma^k$ are $q$-Weil integers. 
 Let $V\subseteq \overline{S}_u$. Let $d_V\geq 1$ be the smallest positive integer such that $V$ is defined over $\mathbb{F}_{q^{d_V}}$. 
Since the eigenvalues of Frobenius action on $\ell$-adic cohomology are $q$-Weil integers, 
the only non-trivial point is to show that the eigenvalues of $F_V^{\ast d_V}$ are divisible by \[q^{d_V(4g-3+|\overline{V}|+|\overline{S}_{cr}|)}. \] 
This is a corollary of \cite[Theorem 5.4.1]{Yu} and Grothendieck-Lefschetz fixed point formula.

Now we prove $(\ast\ast)$. Note that for any $k$, only those $V$ fixed by $\varsigma^k$ will contribute a non-trivial trace. These are exactly those coming from subsets of $S_u\otimes \mathbb{F}_{q^k}$.  The alternative trace $\sum_{i} (-1)^{i}\Tr(\varsigma^k| H_V^i) $ equals \[ (-1)^{O(\Fr^k| V)-|V|}q^{3-4g-|V|-|\overline{S}_{cr}| } \]  times
\[ \sum_{i} (-1)^{i}\Tr(F_{q^{d_V}}^{\ast k/d_V}| H_c^{i}({M}_{\overline{V}}^{1}(o), \mathbb{Q}_\ell)) \sum_{i} (-1)^{i}\Tr(\tau^k| H^{\ast}(\mathbb{P}^1, \mathbb{Q}_\ell)^{\otimes V}) .\] 
Here $O(\Fr^k|V)$ means the number of orbits of $\Fr^k$ on $V$ and Lemma \ref{frobedt} is used. 

We have 
\[\sum_{i} (-1)^{i}\Tr( F_{q^{d_V}}^{\ast k/d_V} | H_c^{i}({M}_{V}^{1}(o)_{\overline{\mathbb{F}}_q}, \mathbb{Q}_\ell)) = |{M}_{V}^{1}(o)(\mathbb{F}_{q^k})| . \]
It reduces to prove that 
\[\sum_{i} (-1)^{i}\Tr(\tau^k| H^{\ast}(\mathbb{P}^1, \mathbb{Q}_\ell)^{\otimes V}) =2^{O(\Fr^k| V)}.\]
By multiplicativity on the orbits of $\Fr^k$ of the two sides, it suffices to consider the case that $\Fr^k$ has only one orbit in $V$, in which case we can do the calculation by choosing a basis: $(\delta_{x}\alpha+ (1-\delta_{x})\beta)_{x\in I}$ where $(\delta_x)_{x\in I}\in \{0,1\}^{I}$. It is clear that in this basis, $\tau^k$ is a permutation matrix, and there are exactly two elements in the basis fixed by $\tau^k$: those given by $\delta_x=\delta_{x'}$ for all $x, x'\in I$. Therefore the left hand side is $2$ as well.

(2) The second assertion is a corollary of \cite[Theorem 5.4.1]{Yu} and the above arguments. 

\end{proof}


\subsection{Proof of Theorem \ref{1}}

The proof is based on Theorem \ref{automorphe} where we have computed the number of cuspidal automorphic representations that correspond to elements in $E_2(\mathfrak{R})^{\Fr^{*}}$ (Theorem \ref{Langlands}). We need to have an expression for $E_2(\mathfrak{R})^{\Fr^{*k}}$ for $k\geq 1$. 
Since \[X\otimes_{\mathbb{F}_q}\overline{\mathbb{F}}_{q}\cong (X\otimes_{\mathbb{F}_q}\mathbb{F}_{q^k}  )\otimes_{\mathbb{F}_{q^k}}\overline{\mathbb{F}}_{q},  \]
and the Frobenius endomorphism of $\overline{X}$ obtained from $X\otimes_{\mathbb{F}_q}\mathbb{F}_{q^k} $ is the $k^{th}$ power of the Frobenius obtained from $X$, 
we can apply this theorem to the function field $F\otimes_{\mathbb{F}_q}\mathbb{F}_{q^k}$. The only difficulty remains that the ramification type on the automorphic side may change when $k$ varies. For example, a place can split into several places, and a supercuspidal representation can become non-supercuspidal after base change.

Let us explain how ramification types on the automorphic side change when $k$ varies. First, a place $v\in S$ of degree $n$ corresponds to an orbit of length $n$ of Frobenius endomorphism on $\overline{S}$. A place of $F$ of degree $n$ splits into $(n,k)$-points of degree $n/(n,k)$ of $F\otimes_{\mathbb{F}_q}\mathbb{F}_{q^k}$ for $k\geq 1$. For ramification types, suppose that \[S\otimes_{\mathbb{F}_q}\mathbb{F}_{q^k}= S_{r}(k)  \coprod S_c(k)  \coprod S_s(k)  \coprod S_u(k) ,  \] 
is a decomposition following the ramification type furnished by Theorem \ref{mono}. Then we have 
\[S_u(k)=S_u(1)\otimes_{\mathbb{F}_q}\mathbb{F}_{q^k},\] and \[S_s(k)=S_s(1)\otimes_{\mathbb{F}_q}\mathbb{F}_{q^k}.\] 
The sets $S_{r}(k)$ and $S_c(k)$ behave differently. If $k$ is an odd number, we have
\[S_{r}(k)=S_{r}(1)\otimes_{\mathbb{F}_q}\mathbb{F}_{q^k},\] and \[S_c(k)=S_c(1)\otimes_{\mathbb{F}_q}\mathbb{F}_{q^k}.\] 
However, if $k$ is an even number, we have 
\[S_{r}(k)=(S_{r}(1)\otimes_{\mathbb{F}_q}\mathbb{F}_{q^k}) \cup  (S_c(1)\otimes_{\mathbb{F}_q}\mathbb{F}_{q^k}),\] and \[S_c(k)=\emptyset.\]

We obtain the cardinality $|E_{2}(\mathfrak{R})^{\Fr^{\ast }}|$ by comparing
 Theorem \ref{automorphe} and Theorem \ref{Higg} by the trace formula:
\[J^1_{spec}(f)=J^1_{geom}(f). \]
The theorem is then just a reformulation of the results using definitions of the functions $\Higg_\mathfrak{R}$,  $\alpha_{\overline{S}_u}$, $\beta_{\overline{S}_u}$, $\eta_{\overline{S}_u}$ and $\omega_{\overline{S}_u}$ in Proposition \ref{Lefschetz} and the vanishing on $c_\mathfrak{R}$ and $b_{\mathfrak{R}}$ in Lemma \ref{bck}. 
Note that we have
\[S^2\Pic_X^0(\mathbb{F}_q)= \frac{1}{2}\biggr(\vert\Pic_X^0(\mathbb{F}_{q^2})\vert+\vert \Pic_X^0(\mathbb{F}_q) \vert^2  \biggr).   \]

It is tedious but direct to verify. Let us be satisfied to explain how to verify the most complicated case that $\overline{S}_{cr} \neq \emptyset $ and $\overline{S}_{u}\neq \emptyset$. 
It should be divided into some sub-cases. If $S_c=\emptyset$, then $|E_{2}(\mathfrak{R})^{\Fr^{\ast k}}|$ is given by $\Higg(k)$ minus the error terms in one of the cases (13), (14), (15), (16) of Theorem \ref{automorphe}. If $S_c\neq \emptyset$, then for $k$ odd, $|E_{2}(\mathfrak{R})^{\Fr^{\ast k}}|$  is given by $\Higg(k)$ minus the error terms in one of the cases (1), (2), (4) (5) of Theorem \ref{automorphe} but for $k$ even, it is given by $\Higg(k)$ minus the error terms in one of the cases (13), (14), (15), (16). We use the definition of $\alpha_{\overline{S}_u}$ and $\beta_{\overline{S}_u}$ to write the result in a uniform formula. One needs to note that 
by Lemma \ref{bck},  
$b_{\mathfrak{R}}(2k)=0$ if $\deg S_c$ is odd and $c_{\mathfrak{R}}(2k+1)=0$ if $S_c\neq \emptyset$.

 \section{The case $g=0$}

\subsection{The case $g=0$}
We are going to give another expression for $\mathrm{Higg}_\mathfrak{R}(k)$ when $g=0$. Let $R=S_{cr}\cup S_{u}$ and $D_R=K_X+\sum_{v\in R}v$. 

Suppose that $(e,\xi)\in \mathbb{Z}\times (\mathbb{Q}^2)^{R}$ is in general position. 
We have a $\mathbb{G}_m$-action on $M_{R}^{e,\xi}=M_{R}^{e,\xi}(D_R)$ given by dilation on the Higgs field. 
We have a modular description for the $\mathbb{G}_m$-fixed points due to Hitchin and Simpson. 

Suppose that $(\mathcal{E}, \theta, (L_x)_{x\in \overline{R}})$ is a parabolic Higgs bundle fixed by $\mathbb{G}_m$-action, then $(\mathcal{E}, \theta)\cong (\mathcal{E}, t\theta)$ for any $t\in \mathbb{G}_m$. By arguments of \cite[Lemma 4.1]{SimpsonHL}, either $(1)$ $\theta=0$ and the underlying parabolic bundle $(\mathcal{E}, (L_x)_{x\in \overline{R}})$ is $\xi$-semistable; or (2) $\theta\neq 0$, $\mathcal{E}$ is decomposed as a direct sum of line bundles \[\mathcal{E}=\mathcal{L}_1\oplus \mathcal{L}_2,  \] and $\theta$ equals the composition
\[\theta: \mathcal{E}\xrightarrow{\text{projection}} \mathcal{L}_2\longrightarrow \mathcal{L}_1(K_X+\sum_{x\in \overline{R}}x)\hookrightarrow\mathcal{E}(K_X+\sum_{x\in \overline{R}}x). \] 
Note that if $g=0$, the first case does not happen as there are no semistable parabolic Higgs bundles of rank $2$ when the $(e,\xi)$ is in general position as defined by \eqref{generalposition}. 
Let $f: (\mathcal{E}, \theta, (L_x)_{x\in \overline{R}})\xrightarrow{\sim} (\mathcal{E}, t\theta, (L_x)_{x\in \overline{R}})$ be an isomorphism of parabolic Higgs bundles. 
Then $f$ has constant coefficient in $\overline{\mathbb{F}}_q$.  Then we have
\[\begin{cases}
f\theta=t\theta f; \\
f(L_x)=L_x, \quad \forall x\in R. 
\end{cases} \]
Let $\lambda$ be an eigenvalue of $f$, then $\mathcal{E}_\lambda:= \ker(f-\lambda)^2$ is a subbundle of $\mathcal{E}$ and $\theta$ sends $\mathcal{E}_\lambda$ to $\mathcal{E}_{t\lambda}$. If $\theta$ is non-zero, then $\mathcal{E}_\lambda$ and $\mathcal{E}_{t\lambda}$ are non-zero and therefore $\mathcal{E}=\mathcal{E}_\lambda\oplus \mathcal{E}_{t\lambda}$. In this case, either \[L_x=\mathcal{E}_{\lambda, x}\] or \[ L_{x}=\mathcal{E}_{t\lambda, x}.\] 
Therefore $(M_{R}^{e,\xi})^{\mathbb{G}_m}$ consists of so-called graded parabolic Higgs bundles, which we will denote by ${^{gr}M}_{R}^{e,\xi}$.

\begin{theorem}\label{identi} 
Suppose that $\mathbb{F}_q\neq \mathbb{F}_2$, $\xi_{v,1}=\xi_{v,2}$ for $v\in S_u$, $(e, \xi)$ is in general position defined by \eqref{generalposition}.  
Let $^{gr}M_{R}^{e,\xi}(S_u)$ be the open subvariety of ${^{gr}M}_{R}^{e,\xi}=(M_{R}^{e,\xi})^{\mathbb{G}_m}$ consisting of those graded semistable parabolic Higgs bundles $(\mathcal{E}, \theta, (L_x)_{x\in \overline{R}})$ such that $\theta_x\neq 0$ for any $x$ lying over points in $S_u$. Suppose $g=0$, and $S_c=\emptyset$. 

If either $e$ is an odd integer or there is a place of odd degree in $R$, then we have 
\[   |{^{gr}M}_{R}^{e,\xi}(S_u)(\mathbb{F}_{q^k})| = \mathrm{Higg}_{\mathfrak{R}}(k), \quad \forall k\geq 1.   \]
\end{theorem}
\begin{proof}
From Theorem \ref{NoSc}, for any $k\geq 1$, the expression $\mathrm{Higg}_{\mathfrak{R}}(k)$ is given by
\[\mathrm{Higg}_{\mathfrak{R}}(k)=
\sum_{V\subseteq {S}_u\otimes\mathbb{F}_{q^k}} (-1)^{|{S}_u\otimes\mathbb{F}_{q^k} - V|} 2^{|V|}     |{^{gr}}{M}_{V\cup (S_r\otimes\mathbb{F}_{q^k})}^{e, \xi}(\mathbb{F}_{q^k})|.
\] 
It suffices to verify the Theorem for the case $k=1$.

Let $(\mathcal{E}, \theta, (L_x)_{x\in \overline{R}})$ be a graded parabolic Higgs bundles.  
For each $x\in \overline{R}$, the fiber map
\[\theta_x: \mathcal{E}_x\longrightarrow \mathcal{E}(K_X+\sum_{x\in \overline{R}}x)_x, \] preserves the parabolic structure. It means that $\Im \theta_x \subseteq L_x$ and $\theta_x(L_x)=0$. 
Suppose that $\theta_x$ is zero, then it is possible that $L_x=\mathcal{L}_{1, x}$ or $L_{x}=\mathcal{L}_{2, x}$. If $\theta_x$ is non-zero, then we can only have $L_x=\mathcal{L}_{1, x}$. We obtain a stratification for any $T\subseteq {S}_u$ and $x\in \overline{S}_u-\overline{T}$,
\[^{gr}M_{\overline{R}}^{e,\overline{\xi}}(\overline{T})=N(1_x)\cup N(2_x)  \cup {^{gr}M^{e,\overline{\xi}}_{\overline{R}}(\overline{T}\cup \{x\})}.  \]
where $N(i_x)$ consists of those $(\mathcal{E}, \theta, (L_x)_{x\in \overline{R}})$ in $^{gr}M_{\overline{R}}^{e,\overline{\xi}}(\overline{T})$ such that $\theta_x=0$ and $L_x=\mathcal{L}_{i, x}$. 
Let $v\in S_u-T$ be a closed point in $S_u$. Repeat the procedure above, we obtain a decomposition of $^{gr}M_{\overline{R}}^{e, \overline{\xi}}(\overline{T})$ as a disjoint union by locally closed subvarieties: 
\[ ^{gr}M_{\overline{R}}^{e,\overline{\xi}}(\overline{T}\cup \overline{\{v\}} )  \cup \bigcup_{ (a_x)_{x\in \overline{\{v\}} }\in \{1, 2\}^{   \overline{\{v\}} }} N((a_x)_{x\in \overline{\{v\}}} ) ,  \] 
where $N((a_x)_{x\in \overline{\{v\}}} )$ consists of those $(\mathcal{E}, \theta, (L_x)_{x\in \overline{R}})$ in $^{gr}M_{\overline{R}}^{e,\overline{\xi}}(\overline{T})$ such that  so that $\theta_{x}=0$ for all $x\in \overline{\{v\}}$ and $L_x=\mathcal{L}_{a_x, x}$. 

Now we must consider how to descend to $\mathbb{F}_q$. 
Since for an $\overline{\mathbb{F}}_q$-sub-variety $Z$ defined over $\overline{\mathbb{F}}_q$ of $^{gr}M_{\overline{R}}^{e,\overline{\xi}}=(^{gr}M_{{R}}^{e,{\xi}})_{\overline{\mathbb{F}}_q}$, it comes from a variety defined over $\mathbb{F}_q$ if it is fixed by Frobenius, i.e., $Z^{(q)}=Z$ as sub-varieties of $^{gr}M_{\overline{R}}^{e,\overline{\xi}}$ where $Z^{(q)}$ is defined by the Cartesian diagram:
\[   \begin{CD}
Z^{(q)}@>>> Z\\
@VVV@VVV\\
 \Spec(\overline{\mathbb{F}}_q)@>x\mapsto x^q>> \Spec(\overline{\mathbb{F}}_q)
\end{CD}.
\]
Here it is important that $Z^{(q)}=Z$ instead of just isomorphism; otherwise, we don't have a descent datum. The equality here is another way to express the commutativity of the following diagram: 
\[   \begin{CD}
(M_{\overline{R}}^{e,\xi})^{(q)}@>\sim>> (M_{\overline{R}}^{e,\xi})\\
@AAA@AAA\\
Z^{(q)}@>\sim>> Z
\end{CD}.
\]

We have
\[ N((a_x)_{x\in \overline{\{v\}}} ) =  N((a_{\Fr(x)})_{{x}\in \overline{\{v\}}} )^{(q)}  . \]
Therefore, we see that  
\( N((1_x)_{x\in \overline{\{v\}}} ) \)
and \(   N((2_x)_{x\in \overline{\{v\}}} )\) are defined over $\mathbb{F}_q$. The variety 
\[ \bigcup_{ (a_x)_{x\in \overline{\{v\}} }\in \{1, 2\}^{   \overline{\{v\}} }  - \{( 1_x)_{x\in  \overline{\{v\}} }, (2_x)_{x\in  \overline{\{v\}}     }  \} } N((a_x)_{x\in \overline{\{v\}}} ) \] 
is defined over $\mathbb{F}_q$ which has no $\mathbb{F}_q$-points, since Frobenius  permutes its components without any fixed component. 
Since $\xi_{v,1 }=\xi_{v,2}$ for $v\in S_u$, as varieties, we have
\[N((1_x)_{x\in \overline{\{v\}}} ) \cong  {^{gr}M_{R-\{v\}}^{e,\xi'}(T)}, \]
and 
\[N((1_x)_{x\in \overline{\{v\}}} ) \cong  {^{gr}M_{R-\{v\}}^{e,\xi'}(T)}, \]
where $\xi' =(\xi_v)_{v\in R-\{v\}}\in (\mathbb{Q}^2)^{R-\{v\}}$ (it is still in general position because of our assumption). Here we remove the overline in the notation to indicate they are varieties over $\mathbb{F}_q$.

We deduce that for any $T\subseteq S_u$ and $v\in S_u-T$,  \[ |^{gr}M_{R}^{e,\xi}(T)(\mathbb{F}_q)|= |^{gr}M_{R}^{e,\xi}(T\cup {\{v\}} )(\mathbb{F}_q)|+2 |^{gr}M_{R-\{v\}}^{e,\xi}(T)(\mathbb{F}_q)| . \]
 By repeating this equality, we obtain the desired identity. 
\end{proof}

\subsection{An example}
Let us consider an example that $X=\mathbb{P}^1=\mathbb{P}^1_{\mathbb{F}_q}$. Note that in this case, $\Omega_{\mathbb{P}^1}^1\cong \mathcal{O}_{\mathbb{P}^1}(-2)$. 

 Suppose $S=\{x_1, \ldots, x_n\}\subseteq \mathbb{P}^1$ is a finite set of closed points of degree $1$. Namely, we can identify $S$ with a subset of $\mathbb{P}^1(\mathbb{F}_q)$. 
Let us consider those $\ell$-adic local systems over $\overline{\mathbb{P}}^1-\overline{S}$ fixed by $\Fr^\ast$ whose local monodromies around $x_i$ ($i<n$) are tame and are at most unipotent, i.e., they are either trivial at 
$x_i$ or are unipotent with one Jordan block at $x_i$, and they are at most quasi-unipotent with eigenvalues $-1$ at $x_n$. 
We can deduce either from Theorem \ref{2} or from its proof directly that they are in bijection  with equivalent classes of semistable graded parabolic Higgs bundles composed of the following data (simply because they have the same number): 
\[ \mathcal{E}=\mathcal{O}_{\mathbb{P}^1}(m)\oplus \mathcal{O}_{\mathbb{P}^1}(1-m),\] \[\theta:\mathcal{E}\longrightarrow \mathcal{O}_{\mathbb{P}^1}(m)\longrightarrow \mathcal{O}_{\mathbb{P}^1}(-m+n-1)\longrightarrow \mathcal{E}(n-2),\] and parabolic structures \[ L_{x_i}=\mathcal{O}_{\mathbb{P}^1}(1-m)_{x_i}, \quad 1\leq i\leq n.\]  
We choose parabolic weights to be zero. Then the semistability says that $m$ is an integer such that \[ m>1-m.\] 
Note that $\theta$ exists if and only if \[ m  \leq  -m+n-1,\]
and when such a $\theta$ exists, the pair $(\mathcal{E}, \theta)$ is semistable if and only if $\theta$ is non-zero. 
Therefore we have \[   1\leq m \leq [\frac{n-1}{2}], \]
where $[\frac{n-1}{2}]$ is the largest integer smaller or equal to $\frac{n-1}{2}$. 

Two graded parabolic Higgs bundles are isomorphic if and only if $m$ is the same and $\theta$ are differed by a non-zero scalar. In fact, two isomorphic graded parabolic Higgs bundles have isomorphic underlying vector bundles. Therefore $m$ should be the same. 
Suppose that $(\mathcal{E}, \theta_1)$ and $(\mathcal{E}, \theta_2)$ are graded parabolic Higgs bundles with $\mathcal{E}=\mathcal{O}_{\mathbb{P}^1}(m)\oplus \mathcal{O}_{\mathbb{P}^1}(1-m)$. It is clear that $\theta_1$ and $\theta_2$ are differed by a constant if and only if there is a $\varphi\in \Aut(\mathcal{E})$ such that \[ (\varphi\otimes id_{\Omega_{\mathbb{P}^1}^1}) \circ\theta_1\circ \varphi^{-1}=\theta_2,\] 
i.e. $(\mathcal{E}, \theta_1)$ and $(\mathcal{E}, \theta_2)$  are isomorphic.

We conclude from the above discussion that the moduli space of semistable graded parabolic Higgs bundles of rank $2$, degree $1$, and for the zero parabolic weight is \[ \coprod_{m} \mathrm{PHom}(\mathcal{O}_{\mathbb{P}^1}(m), \mathcal{O}_{\mathbb{P}^1}(-m+n-1))\cong \coprod_{m=1}^{[\frac{n-1}{2}]} \mathbb{P}^{n-1-2m}.  \]
In particular, the number of its $\mathbb{F}_q$-points equals
\begin{equation} \sum_{i=3}^{n}[\frac{i-1}{2}] q^{n-i}. \end{equation}
This number is also the number of $\ell$-adic local systems over $\mathbb{P}^{1}_{\overline{\mathbb{F}}_q}-\overline{S}$ fixed by $\Fr^\ast$ whose local  monodromies around $x_i$ ($i<n$) are tame and are at most unipotent, i.e., they are either trivial at 
$x_i$ or are unipotent with one Jordan block at $x_i$, and they are at most quasi-unipotent with eigenvalues $-1$ at $x_n$. 
We can not provide a natural bijection between these objects from our method, but I get to know from Kang Zuo that in his work under preparation joint with Jinbang Yang, when $n=4$, they can construct a natural injective map from graded parabolic Higgs bundles to $\ell$-adic local systems, which then is bijective.


\end{document}